\documentclass{article}
\usepackage[a4paper, margin=2cm]{geometry}

\usepackage[leqno,intlimits]{amsmath}
\usepackage{amssymb}
\usepackage{amsthm}
\usepackage{hyperref}
\usepackage{paralist}
\usepackage[dvips]{graphicx}
\usepackage{afterpage}
\usepackage{pstricks}
\usepackage{pst-eps}
\usepackage{pst-node}
\usepackage{url}

\usepackage{color}
\usepackage{soul}
\setstcolor{red}
\numberwithin{equation}{section}

\allowdisplaybreaks[1]

\newtheorem{theorem}{Theorem}[section]
\newtheorem{lemma}[theorem]{Lemma}

\newtheorem{corollary}[theorem]{Corollary}

\newtheoremstyle{example}{\topsep}{\topsep}%
  {}%         Body font
  {}%         Indent amount (empty = no indent, \parindent = para indent)
  {\bfseries}% Thm head font
  {}%        Punctuation after thm head
  {\newline}%     Space after thm head (\newline = linebreak)
	{\thmname{#1}\thmnumber{ #2}\thmnote{ #3}}%         Thm head spec

\theoremstyle{example}
\newtheorem{example}{Example}[section]

\newcommand{\pd}[2]{\frac{\partial#1}{\partial#2}} % partial derivatives quicker.
\newcommand{\E}{\mathbb{E}} % Expectation E
\newcommand{\p}{\mathbb{P}} % Probability P
\newcommand{\eps}{\varepsilon} % Epsilon
\DeclareMathOperator{\arccot}{arccot}

\begin{document}

%%%%%%%%%%%%%%% TITLE %%%%%%%%%%%%%%%

\title{Modeling Flocks and Prices: Jumping Particles with an Attractive Interaction}

\author{
	M\'arton Bal\'azs
	\thanks{Institute of Mathematics, Budapest University of Technology; \texttt{balazs@math.bme.hu}; research partially supported by the Hungarian Scientific Research Fund (OTKA) grants K60708, F67729, K100473, by T\'AMOP - 4.2.2.B-10/1-2010-0009, and by the Bolyai Scholarship of the Hungarian Academy of Sciences.}
	\and
	Mikl\'os Z. R\'acz
	\thanks{Department of Statistics, University of California, Berkeley; \texttt{racz@stat.berkeley.edu}; most of this work was done while the author was at the Institute of Mathematics, Budapest University of Technology.}
	\and
	B\'alint T\'oth
	\thanks{Institute of Mathematics, Budapest University of Technology; \texttt{balint@math.bme.hu}; research partially supported by the Hungarian Scientific Research Fund (OTKA) grant K60708, K100473 and by T\'AMOP - 4.2.2.B-10/1-2010-0009.}
	}

\date{\today}

\maketitle

\begin{abstract}
We introduce and investigate a new model of a finite number of particles jumping forward on the real line. The jump lengths are independent of everything, but the jump rate of each particle depends on the relative position of the particle compared to the center of mass of the system. The rates are higher for those left behind, and lower for those ahead of the center of mass, providing an attractive interaction keeping the particles together. We prove that in the fluid limit, as the number of particles goes to infinity, the evolution of the system is described by a mean field equation that exhibits traveling wave solutions. A connection to extreme value statistics is also provided.
\end{abstract}



%%%%%%%%%%%%%%% TITLE %%%%%%%%%%%%%%%
\section{Introduction}

Interacting particle systems always show interesting behavior and represent a challenge for rigorous treatment. In this paper we introduce and investigate a new type of model which---to the best of our knowledge---has not emerged in the literature before. The system consists of a finite number of particles moving forward on the real line $\mathbb R$. The center of mass of these particles is computed and, given a configuration, each particle jumps after an exponential waiting time that depends on its relative position compared to the center of mass. The more a particle lags behind the center of mass, the less its mean waiting time is; on the other hand, the more a particle is ahead of the center of mass, the more its mean waiting time is. When the exponential time has passed, the particle jumps forward a random distance, independent of everything else, drawn from a probability distribution concentrated on the non-negative half-line. While doing so, it naturally moves the center of mass forward by $1/n$ times its jump length as well, where $n$ is the number of particles. Thus the dynamics has an effect of keeping the particles together, while randomness creates fluctuations in the system.

Several real-life phenomena can be modeled this way: the evolution of prices set by agents trying to sell the same product in a market, the position of goats in a herd, etc. Mathematical finance is an area of active research, and many processes arising in finance are modeled with interacting particle systems, see e.g.~the monograph on stochastic portfolio theory by Fernholz \cite{fernholz2002stochastic}, Banner, Fernholz and Karatzas \cite{banner2005atlas}, and the survey by Fernholz and Karatzas \cite{fernholz2008stochastic}. Herding behavior has drawn serious interest in the physics literature, a few examples are Vicsek, Czir\'ok, Ben-Jacob, Cohen and Shochet \cite{vicsek1995ntp}, Czir\'ok, Barab\'asi and Vicsek \cite{czirok1999cms}, Ballerini, Cabibbo, Candelier, Cavagna, Cisbani, Giardina, Lecomte, Orlandi, Parisi, Procaccini, Viale and Zdravkovic \cite{ballerini2008cir}.

The problem of competing individuals is interesting even without dependence of the dynamics on the actual position of the particles. A series of papers have been written on models where the competing particles evolve according to i.i.d.~increments: Ruzmaikina and Aizenman \cite{ruzmaikina2005characterization}, Arguin \cite{arguin2008competing}, Arguin and Aizenman \cite{arguin2009structure}, Shkolnikov \cite{shkolnikov2009competing}. The main point of investigation in these papers is characterizing point processes for which the joint distribution of the gaps between the particles is invariant under the evolution.

Several authors have considered continuous models in which the motion (drift) of the particles depends on their rank in the spatial ordering; see e.g.~Banner, Fernholz and Karatzas \cite{banner2005atlas}, Pal and Pitman \cite{pal2008one}, Chatterjee and Pal \cite{chatterjee2008phase}, Shkolnikov \cite{shkolnikov2010competing,shkolnikov2010large,shkolnikov2011large}. These models are central objects of study in modeling capital distributions in equity markets. Branching diffusion-type models have also been investigated with a linear interaction between particles, see e.g.\ Greven and Hollander \cite{greven2007phase}, Engl\"ander \cite{englander2010center} and the references therein. Linearity of the interaction is an important ingredient in these papers making the analysis somewhat easier.

Closer to the present work are jump processes with interaction between the jumping particles. A related model is introduced in ben-Avraham, Majumdar and Redner \cite{benavraham2007tmr}, where particles compete with a leader who looks back only at the contestant directly behind him. Interaction between the leader and the followers is realized via a linear dependence of the mean jump length on the relative position of the leader and the followers. A process with rank-dependent jump rates is investigated in Greenberg, Malyshev and Popov \cite{greenberg1995stochastic}, and one with similar, but randomized jump rates in Manita and Shcherbakov \cite{manita2005asymptotic}. Maybe the closest model to our system is the one featured in Grigorescu and Kang \cite{grigorescu2008steady}, where particles jump in a configuration-dependent manner, and the authors give special focus to the situation where the dependence of particle advances on the configuration is via the center of mass of the particles.

In our case, positive jump lengths are independently chosen, and the jump rate of a particle depends on its position relative to the center of mass. This \emph{jump rate function} is non-increasing and non-negative, hence clearly cannot be linear in nontrivial cases. The setting is thus genuinely different from situations in the above examples. The non-increasing jump rate function creates a monotonicity in the system which keeps the particles together. This gives rise to a traveling wave-type stationary distribution. We address the following questions.
\begin{itemize}
\item The two-particle case is, of course, tractable. We compute the stationary distribution of the gap between the two particles for several choices of the jump length distribution. However, we cannot handle the situation even for three particles. (The process is not reversible.)
\item If the number of particles goes to infinity, we expect that the motion of the center of mass converges to a deterministic one. This motion, together with the distribution of the particles, satisfy a deterministic integro-differential equation in this so-called \emph{fluid limit}. We call the resulting equation the mean field equation. We write up this equation by heuristic arguments and, in some cases, we solve it for a traveling wave form, i.e.\ the stationary solution as seen from the center of mass. Even finding the traveling wave solution to this mean field equation is nontrivial. The reason for expecting a deterministic limiting equation is the averaging property of the dynamics. We in fact prove this phenomena in this paper later on for some cases. As the motion of the center of mass becomes deterministic, the particles decouple, each independently following the deterministically moving center of mass. Therefore in the limit it does not matter if we consider a deterministic cloud of particles (the limit coming from the empirical distribution of $n\to\infty$ particles), or a single particle in the corresponding distribution that is driven by the deterministic center of mass. The evolution of the fluid limit can therefore be called a mean field dynamics. Of course, the motion of the center of mass must be chosen appropriately to be consistent with the cloud or distribution of particles. In the traveling wave case, the distribution, as seen from the center of mass, becomes stationary, and the speed of the center of mass becomes constant in time.
\item Some traveling wave solutions lead us to Gumbel distributions, which naturally pose the question whether an extreme value theory interpretation can be given to the process. An idea by Attila R\'akos shows that there is indeed such a connection.
\item Under restrictive assumptions on the jump rates, we prove that
the evolution of the process in the fluid limit indeed satisfies the deterministic integro-differential equation we formulated by heuristic arguments. We use weak convergence methods and uniqueness results on the integro-differential equation. Since the center of mass is not a bounded function of the particle positions, standard weak convergence techniques need to be extended to our situation.
\end{itemize}

This last point is the most interesting from a mathematical point of view. Our proof is a standard fluid limit argument, somewhat similar to those in Manita and Shcherbakov \cite{manita2005asymptotic} and Grigorescu and Kang \cite{grigorescu2008steady}: we use an extended notion of weak convergence, tightness, and martingales to show convergence to the unique solution of the limiting integro-differential equation. As opposed to Greenberg, Malyshev and Popov \cite{greenberg1995stochastic}, asymptotic independence of the particles (as the number of particles tends to infinity) will be a consequence of the proof, rather than a method to rely on. This is because our case of dependence through the center of mass does not seem to provide an easy a priori way of showing asymptotic independence of the particles.

We start by precisely describing the model in Section \ref{sec:model}. We formulate our main results in Section \ref{sec:results}:  first, we formulate the limiting mean field equation together with some traveling wave solutions, and we also state the fluid limit theorem. These are the results we believe to be the most important, and additional statements can also be found in subsequent sections. We end this section with a collection of open questions. We proceed by analyzing the two-particle case in Section \ref{sec:two_particle}. The mean field equation is investigated in Section \ref{sec:mf}. In particular, a connection with extreme value statistics is heuristically explained in Section \ref{sec:gumbel}. Section \ref{sec:fluid_limit} contains our fluid limit argument under some restrictive conditions on the jump rates, and finally Section \ref{sec:sim} supports, via simulation results, the idea that the fluid limit should hold in other cases as well.

\section{The model}\label{sec:model}

To model the phenomena described in the Introduction, we present the following model. The model consists of $n$ particles on the real line. Let us denote the position of the particles by $x_{1}\left(t\right), x_{2}\left(t\right),\dots,x_{n}\left(t\right)$, where $t$ represents time and $x_{i}\left(t\right) \in \mathbb{R}$ for $i=1,\dots,n$. We now describe the dynamics. Given a configuration, the mean position (center of mass) of the particles is $m_{n}\left(t\right) = \frac{1}{n}\sum_{i=1}^{n}{x_{i}\left(t\right)}$. Now let us denote by $w : \mathbb{R} \to \mathbb{R^{+}}$ the \emph{jump rate function}, a positive and monotone decreasing function. The dynamics is a continuous-time Markov jump process. The $i^{\text{th}}$ particle, positioned at $x_{i}\left(t\right)$ at time $t$, jumps with rate $w\left(x_{i}\left(t\right) - m_{n}\left(t\right)\right)$. That is, conditioned on the configuration of the particles, jumps happen independently after an exponentially distributed time, with the parameter of the $i^{\text{th}}$ particle being $w\left(x_{i}\left(t\right) - m_{n}\left(t\right)\right)$. 

When a particle jumps, the length of the jump is a random positive number $Z$ from a specific distribution. This is independent of time and the position of the particle and all other particles as well. We scale the jump length so that $\E Z = 1$, and for technical reasons we assume that $Z$ has a finite third moment. If this distribution is absolutely continuous then its density is denoted by $\varphi$.

Our focus is on the \emph{empirical measure} of the $n$ particle process at time $t$:
\begin{equation*}
\mu_{n}\left( t \right) = \frac{1}{n} \sum\limits_{i=1}^{n} \delta_{x_{i} \left( t \right)}.
\end{equation*}
More generally, we are interested in the \emph{path} of the empirical measure, $\mu_{n} \left( \cdot \right) = \left( \mu_{n} \left( t \right) \right)_{t \geq 0}$, with initial condition $\mu_{n} \left( 0 \right) = \frac{1}{n} \sum_{i=1}^{n} \delta_{x_{i}\left( 0 \right)}$.

We introduce the following notation in order to abbreviate formulas, which is also well suited for weak convergence type arguments. For a function $f: \mathbb{R} \to \mathbb{R}$ and a probability measure $P$ on $\mathbb{R}$ let
\begin{equation*}
\left\langle f, P \right\rangle := \int\limits_{-\infty}^{\infty} f \left( x \right) P\left( dx \right).
\end{equation*}
We also use the following notation: $\left\langle f\left( x \right), P \right\rangle \equiv \left\langle f, P \right\rangle$. Using this notation, we have $m_n \left( t \right) = \left\langle x, \mu_n \left( t \right) \right\rangle$.

Now we can present the generator of the process acting on the integrated test functions as above. Let $C_{b}$ denote the space of bounded and continuous functions from $\mathbb{R}$ to $\mathbb{R}$, and let $Id$ denote the identity from $\mathbb{R}$ to $\mathbb{R}$: $Id \left( x \right) = x$. Define
\begin{equation}\label{eq:H_def}
H := \left\{ f \in C_{b}:\ \left|f\right| \leq 1 \right\} \cup \left\{Id\right\}.
\end{equation}
This set of functions will be important later on. For a function $f \in H$ we have:
\begin{equation*}
L \left\langle f, \mu_n \left( t \right) \right\rangle = \left\langle \left( \E \left( f \left( x + Z \right) \right) - f\left( x \right) \right) w\left( x - m_{n}\left( t \right) \right), \mu_{n} \left( t \right) \right\rangle.
\end{equation*}
The particular choice of $H$ is so that we can apply the generator $L$ to the center of mass $m_n \left( t \right)$.

\section{Results}\label{sec:results}

The case of a fixed number $n$ of particles is difficult to treat. Of course the $n = 2$ case is doable: in Section \ref{sec:two_particle} we look at two specific cases of the model, and in both we compute the stationary distribution of the gap between the two particles. At the end of the section we illustrate the difficulties that arise for more than two particles in Figure \ref{fig:haromszogek_sz}.

Our main interest is in the large $n$ limit (so called \emph{fluid limit}) of the particle system. Let us assume that the jump length distribution has a density $\varphi$. Heuristically, in the fluid limit the empirical measure $\mu_n \left( t \right)$ of the particles at time $t$ becomes an absolutely continuous measure with density $\varrho \left( \cdot, t \right)$ whose time evolution is given by the so-called mean field equation
\begin{equation}
\label{eq:master}
\pd{\varrho\left(x,t\right)}{t} = - w\left(x-m\left(t\right)\right) \varrho\left(x,t\right) + \int\limits_{-\infty}^{x} w\left(y-m\left(t\right)\right) \varrho\left(y,t\right) \varphi\left(x-y\right)dy,
\end{equation}
where
\begin{equation}
\label{eq:mean}
m\left(t\right) = \int\limits_{-\infty}^{\infty} x \varrho\left(x,t\right) dx
\end{equation}
is the mean of the probability density $\varrho\left( \cdot ,t\right)$. The first term on the right hand side of \eqref{eq:master} accounts for the particles hopping forward from position $x$ at time $t$: at time $t$ there is a density of $\varrho\left(x,t\right)$ particles at $x$, who jump forward with rate $w\left(x-m\left(t\right)\right)$. The second, integral term on the right hand side of \eqref{eq:master} accounts for the particles who are at position $y < x$ jumping forward at time $t$ to position $x$: at time $t$ there is a density of $\varrho\left(y,t\right)$ particles at $y$, who jump forward with rate $w\left(y-m\left(t\right)\right)$, and need to jump forward exactly $x-y$ to land at $x$, the probability of this being proportional to $\varphi\left(x-y\right)$.

Our first results concern solutions to the mean field equation \eqref{eq:master} in the form of a traveling wave
\begin{equation}
\label{eq:TravellingWave}
\varrho(x,t)=\rho(x-ct),
\end{equation}
where $c$ is the constant velocity of the wave. Traveling waves of this form represent stationary distributions as seen from the center of mass, which travels at a constant speed $c$. Substituting \eqref{eq:TravellingWave} into the mean field equation \eqref{eq:master}, and then writing $x$ instead of $x-ct$ for simplicity everywhere, we arrive at the following equation:
\begin{equation}
\label{eq:stationary}
-c\rho'\left(x\right) = -w\left(x\right)\rho\left(x\right)+\int\limits_{-\infty}^{x} w\left(y\right)\rho\left(y\right)\varphi\left(x-y\right)dy.
\end{equation}
Recall that the jump rate function $w$ is a non-increasing function; any non-constant $w$ gives a strong enough attraction to keep the particles together, i.e.\ there is a probability density traveling wave solution $\rho$. This is formulated in the following theorem for exponentially distributed jump lengths.
\begin{theorem}\label{thm:exp_stac}
In the specific case when the jump length distribution is exponential with mean one, the solution to \eqref{eq:stationary} for general jump rate function $w$ is
\begin{equation}
\label{eq:stationary_dist}
\rho\left(x\right) = Ke^{\int\limits_{0}^{x} \left(\frac{1}{c}w\left(s\right)-1\right) ds},
\end{equation}
where $K$ is a constant. If $w$ is non-constant then for an appropriate normalizing constant $K$, $\rho$ of \eqref{eq:stationary_dist} is a probability density with at least exponentially decreasing tails. In our model the speed of the wave is determined by the fact that $\rho$ is centered, thus $c$ is the unique solution of
\begin{equation}
\label{eq:speed_derivation}
\int\limits_{-\infty}^{\infty} x e^{\int\limits_{0}^{x}\left(\frac{1}{c}w\left(s\right)-1\right)ds} dx = 0.
\end{equation}
\end{theorem}
There are several specific jump rate functions $w$ for which the resulting stationary distribution is of interest. In particular, the case of $w \left( x \right) = e^{- \beta x }$, where $\beta > 0$, leads to the generalized Gumbel distribution:
\begin{corollary}\label{cor:gumbel}
If $w\left(x\right)=e^{-\beta x}$, where $\beta > 0$, then from Theorem \ref{thm:exp_stac} we obtain for the stationary density
\begin{equation}
\label{eq:rhobeta1.1}
\rho_{\beta}\left(x\right) = \frac{\beta}{\Gamma\left(\frac{1}{\beta}\right)}e^{-\left(x-\frac{\psi\left(\frac{1}{\beta}\right)}{\beta}\right)-e^{-\beta\left(x-\frac{\psi\left(\frac{1}{\beta}\right)}{\beta}\right)}},
\end{equation}
where $\psi$ is the digamma function:
\begin{equation*}
\psi\left(x\right) = \frac{\Gamma'\left(x\right)}{\Gamma\left(x\right)}.
\end{equation*}
Specifically for $\beta = 1$ the stationary distribution is the centered standard Gumbel distribution.
The velocity of the wave is
\begin{equation*}
c=\frac{1}{\beta}e^{-\psi\left(\frac{1}{\beta}\right)}.
\end{equation*}
\end{corollary}
The Gumbel distribution often arises in extreme value theory (Gumbel \cite{gumbel}), e.g.\ this is the limiting distribution of the rescaled maximum of a large number of independent identically distributed random variables drawn from a distribution that has no upper bound and decays faster than any power law at infinity (e.g.\ exponential or Gaussian distributions). More generally, the $k$th largest value, once rescaled, follows the generalized Gumbel distribution (Bertin \cite{bertin2005gfgs}, Clusel and Bertin \cite{clusel2008global}, Gumbel \cite{gumbel}). Therefore the question arises whether there is an underlying extremal process in the dynamics. It turns out that there is a connection between the mean field model and extreme value theory: this is discussed in Section \ref{sec:gumbel}.

The validity of the mean field equation is strongly supported by simulation results for various special cases of the model, see Section \ref{sec:sim}. We expect this to be true in quite some generality (for fairly general jump rate functions and jump length distributions).

 The main result of the paper is to rigorously validate this approximation. That is, our goal is to show that in the fluid limit---when the number of particles tends to infinity---the mean field equation \eqref{eq:master} holds. However, when trying to rigorously prove this, technical difficulties may arise. Therefore we impose some restrictions on the jump rate function and the jump length distribution. Our main restriction is that we suppose that $w$ is bounded: $w \left( x \right) \leq a$. We also impose some additional (smoothness) restrictions on $w$, but boundedness is the main restriction. Our restrictions on the jump length distribution are light: we assume that it has a finite third moment.

In order to precisely state our main theorem, some preparations are needed. First notice that the mean field equation \eqref{eq:master} can be rewritten as the evolution of a probability measure $\mu \left( \cdot \right)$ using test functions from $H$. Moreover, we do not have to assume that the jump length distribution has a density either.
Define
\begin{equation}\label{eq:A_def}
\begin{aligned}
A_{t,f} \left( \mu\left( \cdot \right) \right) :&= \left\langle f, \mu \left( t \right) \right\rangle - \left\langle f, \mu \left( 0 \right) \right\rangle - \int\limits_{0}^{t} \left\langle \left( \E \left( f \left( x + Z \right) \right) - f\left( x \right) \right) w\left( x - m\left( s \right) \right), \mu \left( s \right) \right\rangle ds\\
&= \left\langle f, \mu \left( t \right) \right\rangle - \left\langle f, \mu \left( 0 \right) \right\rangle - \int\limits_{0}^{t} L \left\langle f, \mu \left( s \right) \right\rangle  ds
\end{aligned}
\end{equation}
where
\begin{equation*}
%m\left( t \right) = \int x \mu \left( t, dx \right).
m\left( t \right) = \left\langle x, \mu \left( t \right) \right\rangle.
\end{equation*}
We say that $\mu\left( \cdot \right)$ satisfies the mean field equation if
\begin{equation}\label{eq:Atf}
A_{t,f} \left( \mu \left( \cdot \right) \right) = 0
\end{equation}
for every $t \geq 0$ and $f \in H$.

The position of the center of mass of the particles plays an important role in our model, and therefore we always assume it exists (and is finite). Consequently, it is natural to look at $\mu_n \left( t \right)$ as an element of $\mathcal{P}_1 \left(\mathbb{R}\right)$, the space of probability measures on $\mathbb{R}$ having a finite first (absolute) moment. In order to talk about convergence in $\mathcal{P}_1 \left( \mathbb{R} \right)$ we need to define a metric on $\mathcal{P}_1 \left( \mathbb{R} \right)$. We use the so-called \emph{1-Wasserstein metric}. For $\mu, \nu \in \mathcal{P}_{1} \left( \mathbb{R} \right)$ define the set of coupling measures
\begin{equation*}
\Pi \left( \mu, \nu \right) := \left\{ \pi \in M_{1} \left( \mathbb{R} \times \mathbb{R} \right)\ :\ \pi \left( \cdot \times \mathbb{R} \right) = \mu, \pi \left( \mathbb{R} \times \cdot \right) = \nu \right\},
\end{equation*}
where $M_1 \left( \mathbb{R} \times \mathbb{R} \right)$ denotes the space of probability measures on $\mathbb{R} \times \mathbb{R}$. The \emph{$1$-Wasserstein metric} on the space $\mathcal{P}_{1} \left( \mathbb{R} \right)$ is defined as
\begin{equation}\label{eq:was_metric}
d_{1} \left( \mu, \nu \right) := \inf \left\{ \int_{\mathbb{R} \times \mathbb{R}} \left| x - y \right| \pi \left( dx, dy \right)\ :\ \pi \in \Pi \left( \mu, \nu \right) \right\}
\end{equation}
for $\mu, \nu \in \mathcal{P}_{1} \left( \mathbb{R} \right)$. Much is known about the metric space $\left( \mathcal{P}_{1} \left( \mathbb{R} \right), d_{1} \right)$; we will state what we need later.

The natural space in which to consider $\mu_{n} \left( \cdot \right)$ is the Skorohod space $D\left( [0,\infty), \mathcal{P}_{1} \left( \mathbb{R} \right) \right)$: we know that $\mu_{n} \left( \cdot \right)$ is a piece-wise constant function in $\mathcal{P}_{1} \left( \mathbb{R} \right)$, having jumps whenever a particle jumps. Similarly, for any test function $f$ (for instance $f \in H$), the natural space in which to consider $\left\langle f, \mu_{n} \left( \cdot \right) \right\rangle$ is the Skorohod space $D\left( [0,\infty), \mathbb{R} \right)$.

In order to prove weak convergence of $\left\{ \mu_{n} \left( \cdot \right) \right\}_{n\geq 1}$ in the appropriate Skorohod space, we have to make the following assumptions on the initial distribution of the particles.

\textbf{Assumption 1.} Our first assumption is that
\begin{equation}\label{eq:assumption1}
\lim\limits_{L \to \infty} \limsup\limits_{n \to \infty} \p \left( \frac{1}{n} \sum\limits_{i=1}^{n} \left| x_{i}\left( 0 \right) \right| > L \right) = 0.
\end{equation}
%A sufficient condition for this assumption to hold is that there exists an $M < \infty$ such that $\E\left( \left| x_{i} \left( 0 \right) \right| \right) \leq M$ for all $i$.

\textbf{Assumption 2.} Our second assumption is that there exists a constant $\hat{c}$ such that for every $\eps > 0$ 
\begin{equation*}
\sup\limits_{n} \p \left( \frac{1}{n} \sum\limits_{i=1}^{n} \mathbf{1} \left[ \left| x_{i} \left( 0 \right) \right| > \hat{c} \eps^{-2/3} \right] \geq \eps / 4 \right) \leq \eps  /3.
\end{equation*}
%In particular, Assumption 2 is satisfied if the distribution of $\left| x_{i} \left( 0 \right) \right|$ is not too heavy-tailed.
(Note: the exponent 2/3 has no special significance. This is simply because we assume $\E \left( Z^3 \right) < \infty$ for simplicity. We could assume $\E \left( Z^{2+\epsilon} \right) < \infty$ for some small $\epsilon > 0$, in which case we could replace the exponent 2/3 with an exponent of $2/ \left( 2 + \epsilon \right)$ in Assumption 2.)

A sufficient and simple condition for both Assumption 1 and 2 to hold is that
\begin{equation*}
\sup_n \E \left( \frac{1}{n} \sum_{i=1}^{n} \left| x_i \left( 0 \right) \right|^{3} \right) < \infty.
\end{equation*}

\textbf{Assumption 3.} Our third assumption is that
\begin{equation}\label{eq:assumption3}
\mu_{n} \left( 0 \right) \Rightarrow_{n} \nu
\end{equation}
in $\mathcal{P}_{1}\left(\mathbb{R}\right)$, where $\nu \in \mathcal{P}_{1}\left( \mathbb{R} \right)$ is a deterministic initial profile and $\Rightarrow_{n}$ denotes weak convergence.

We are now ready to state our main theorem.

\begin{theorem}\label{thm:main}
Suppose that the non-increasing jump rate function $w$ is bounded, and that it is differentiable with bounded derivative. Suppose the jump length $Z$ is scaled such that $\E Z = 1$ and that it has a finite third moment. Suppose that Assumptions 1, 2 and 3 from above hold. Then
\begin{equation}%\label{eq:weak_convergence_of_mu_n}
\mu_{n} \left( \cdot \right) \Rightarrow_{n} \mu \left( \cdot \right)
\end{equation}
in $D\left( [0,\infty), \mathcal{P}_{1}\left( \mathbb{R} \right) \right)$, where $\mu\left( \cdot \right)$ is the unique deterministic solution to the mean field equation \eqref{eq:Atf} with initial condition $\nu$.
\end{theorem}
The proof of this theorem is in Section \ref{sec:fluid_limit}. The argument follows the standard method for proving weak convergence results. Namely, based on Prohorov's theorem, we need to do three things:
\begin{itemize}
 \item Prove tightness of $\left\{ \mu_{n}\left( \cdot \right) \right\}_{n\geq 1}$ in the Skorohod space $D\left( [0,\infty), \mathcal{P}_{1}\left( \mathbb{R} \right) \right)$.
 \item Identify the limit of $\left\{ \mu_{n}\left( \cdot \right) \right\}_{n\geq 1}$: here we essentially have to show that the limit solves the deterministic mean field equation.
 \item Prove uniqueness of the solution to the mean field equation with a given initial condition.
\end{itemize}

\subsection{Open questions}\label{sec:open}

We end this section with some open questions.
\begin{itemize}
\item Theorem \ref{thm:main} is a law of large numbers type result, so it is natural to ask about the fluctuations around this first order limit. In particular, how does the variance of the empirical mean $m_n \left( t \right)$ scale with $n$ and $t$? Preliminary simulations suggest that this variance grows linearly in $t$, and for fixed $t$ decays in $n$ as $n^{- \alpha}$ for some $\alpha > 0$. (The simulations are inconclusive of what the exponent $\alpha$ should be.)
\item Once the order of fluctuations of the empirical mean $m_n \left( t \right)$ is known, can we determine the limiting distribution of $m_n \left( t \right)$ with appropriate scaling?
\item Going back to the finite particle system: can we really not deal with a fixed number ($n \geq 3$) of particles and determine the stationary distribution viewed from the center of mass in this case?
\item In Theorem \ref{thm:main} we make assumptions on the jump rate function $w$, the jump length distribution, and the initial distribution of the particles. Most of the assumptions are fairly weak, except for the assumption that $w$ is bounded, which is the main restrictive assumption. We believe that these assumptions can be weakened, and in particular we are interested in removing the assumption that $w$ is bounded. Can one prove the same result for unbounded $w$, e.g.\ $w\left( x \right) = e^{-x}$?
\end{itemize}

\section{Two-particle system}\label{sec:two_particle}
As an introduction to the model, we look at the simplest case: when $n=2$, i.e.\ when there are two competing particles. In particular, we look at two specific cases of the model, and in each case we compute the stationary distribution of the gap between the two particles. At the end of the section we illustrate the difficulties that arise in the case of $n \geq 3$ particles.

\subsection{Deterministic jumps}

First, let us look at the model when the length of the jumps is not random but a deterministic positive number. We always have the freedom to scale the system so that this length is equal to unity. So, in this subsection, we look at the model with deterministic jumps of length 1. 

In the case of two particles, $x_{1}\left(t\right)$ and $x_{2}\left(t\right)$, we are interested in how far away they move from each other, thus in the gap $g\left(t\right) = \left|x_{1}\left(t\right) - x_{2}\left(t\right)\right|$. For simplicity let us suppose that $x_{1}\left(0\right) - x_{2}\left(0\right) \in \mathbb{Z}$, so that $g\left(t\right) \in \mathbb{N}$ for all $t\geq 0$. In this case $g\left(t\right)$ is a continuous-time Markov process on $\mathbb{N}$, with infinitesimal generator $Q$ given by the model in the following way. Since the gap can only change by 1, $Q$ is a tridiagonal matrix. If the gap is 0, i.e. $x_{1}\left(t\right) = x_{2}\left(t\right)$, then both particles jump with rate $w\left(0\right)$ and whichever one does, the gap changes to 1. Therefore the rate at which the gap changes from 0 to 1 is $2w\left(0\right)$ and thus
\begin{equation*}
Q_{0,0} = - 2 w\left(0\right), \qquad Q_{0,1} = 2 w\left(0\right).
\end{equation*}
If the gap between the particles is $k>0$, then the one in the lead jumps with rate $w\left(\frac{k}{2}\right)$, while the one which is behind jumps with rate $w\left(-\frac{k}{2}\right)$. This means that the gap increases by 1 with rate $w\left(\frac{k}{2}\right)$ and decreases by 1 with rate $w\left(-\frac{k}{2}\right)$ and so
\begin{equation*}
Q_{k,k-1} = w\left(-\frac{k}{2}\right), \quad Q_{k,k} = -\left(w\left(-\frac{k}{2}\right) + w\left(\frac{k}{2}\right)\right), \quad Q_{k,k+1} = w\left(\frac{k}{2}\right).
\end{equation*}

When the jump rate function $w$ is a monotone decreasing, non-constant function, then $g\left( t \right)$ is a positive recurrent birth-and-death process with stationary distribution $\mathbf{\pi} = \left(\pi_{0},\pi_{1},\pi_{2},\dots \right)$ given by
\begin{align*}
\pi_{0} &= \frac{1}{1 + 2 \sum\limits_{k=1}^{\infty}{\frac{\prod_{i=0}^{k-1}{w\left(\frac{i}{2}\right)}}{\prod_{i=1}^{k}{w\left(-\frac{i}{2}\right)}}}}\\
\pi_{k} &= 2 \pi_{0}\frac{\prod_{i=0}^{k-1}{w\left(\frac{i}{2}\right)}}{\prod_{i=1}^{k}{w\left(-\frac{i}{2}\right)}}, \quad \text{for } k\geq 1.
\end{align*}

\begin{example}[\textmd{\textit{Exponential jump rate function}}]\label{ex:two_particle_exp}
If $w\left(x\right) = e^{-\beta x}$, where $\beta > 0$, then the stationary distribution becomes
\begin{align*}
\pi_{0} &= \frac{1}{1 + 2 \sum\limits_{k=1}^{\infty}{e^{-\frac{\beta k^{2}}{2}}}}\\
\pi_{k} &= \frac{2 e^{-\frac{\beta k^{2}}{2}}}{1 + 2 \sum\limits_{k=1}^{\infty}{e^{-\frac{\beta k^{2}}{2}}}}, \quad \text{for } k\geq 1,
\end{align*}
which shows Gaussian decay.
\end{example}

\begin{example}[\textmd{\textit{Jump rate function is a step function}}]\label{ex:two_particle_step}
Let
\begin{equation*}
w\left(x\right) = 
\begin{cases}
a & \text{if $x < 0$,}
\\
b &\text{if $0 \leq x$,}
\end{cases}
\end{equation*}
where $a>b>0$.  Then a short calculation shows that the stationary distribution is
\begin{align*}
\pi_{0} &= \frac{a-b}{a+b}\\
\pi_{k} &= 2\frac{a-b}{a+b} \left(\frac{b}{a}\right)^{k}, \quad \text{for } k\geq 1,
\end{align*}
which shows exponential decay.
\end{example}

\subsection{Jump length distribution having a density}

Consider again the case of two particles, but now let us suppose that the distribution of the length of the jumps has a density, and denote this by $\varphi$ as mentioned in Section \ref{sec:model}. As before, we are interested in the gap $g\left(t\right) = \left|x_{1}\left(t\right) - x_{2}\left(t\right)\right|$ between the two particles. Denote by $p\left(g,t\right)$ the probability density of the gap $g$ at time $t$, which is concentrated on $[0,\infty)$.%
{
\psset{unit=1pt}
\psset{linewidth=0.5pt}
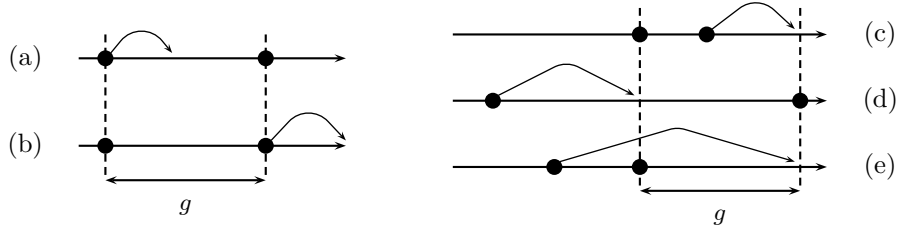
\begin{figure}[ht]
  \centering
  
	\begin{pspicture*}(342,100)
	
		% left side of picture
		\psline[linewidth=0.8pt]{->}(30,66)(130,66)
		\psline[linewidth=0.8pt]{->}(30,33)(130,33)
		\psline[linewidth=0.8pt,linestyle=dashed, dash=3 2]{-}(40,75)(40,22)
		\psline[linewidth=0.8pt,linestyle=dashed, dash=3 2]{-}(100,75)(100,22)
		\psline[linewidth=0.8pt]{<->}(40,20)(100,20)
		\rput{0}(70,10){\small{$g$}}
		\qdisk(40,66){3}
		\qdisk(40,33){3}
		\qdisk(100,66){3}
		\qdisk(100,33){3}
		\psline[linearc=10]{->}(42,68)(53.5,82)(65,68)
		\psline[linearc=10]{->}(102,35)(116,50)(130,35)
		
		\rput{0}(10,66){(a)}
		\rput{0}(10,33){(b)}
	
	% right side of picture
		% part 1
		\psline[linewidth=0.8pt]{->}(170,75)(310,75)
		\psline[linewidth=0.8pt,linestyle=dashed, dash=3 2]{-}(240,85)(240,18)
		\psline[linewidth=0.8pt,linestyle=dashed, dash=3 2]{-}(300,85)(300,18)
		\psline[linewidth=0.8pt]{<->}(240,16)(300,16)
		\rput{0}(270,6){\small{$g$}}
		\qdisk(240,75){3}
		\qdisk(265,75){3}
		\psline[linearc=10]{->}(267,77)(284,90)(298,77)
		
		\rput{0}(330,75){(c)}

		% part 2
		\psline[linewidth=0.8pt]{->}(170,50)(310,50)
		\qdisk(185,50){3}
		\qdisk(300,50){3}
		\psline[linearc=10]{->}(187,52)(212.5,65)(238,52)
		
		\rput{0}(330,50){(d)}

		% part 3
		\psline[linewidth=0.8pt]{->}(170,25)(310,25)		
		\qdisk(240,25){3}
		\qdisk(208,25){3}				
		\psline[linearc=10]{->}(210,27)(254,40)(298,27)
		
		\rput{0}(330,25){(e)}

	\end{pspicture*}
	\caption{Events that change the gap length.}\label{fig:jumps}
\end{figure}
}%
The time evolution of the gap is completely described by the time evolution of $p$, which is described by the master equation %\cite{gardiner:hsm}
\begin{align*}
\dot{p}\left(g\right) = &-p\left(g\right) \left( w\left( -g/2 \right) + w\left( g/2 \right) \right) + \int\limits_{0}^{g}{p\left(y\right) w\left( y/2 \right) \varphi\left( g-y \right) dy}\\
&+ \int\limits_{g}^{\infty}{p\left(y\right) w\left( -y/2 \right) \varphi\left( y-g \right) dy} + \int\limits_{0}^{\infty}{p\left(y\right) w\left( -y/2 \right) \varphi\left( y+g \right) dy},
\end{align*}
where the overdot denotes time derivative. The first term on the right hand side accounts for the case when the gap length is $g$ and one of the two particles jumps (Fig.\ \ref{fig:jumps}a and \ref{fig:jumps}b). The second term accounts for the case when a gap of length $g$ is created from a previous gap of length $y < g$ due to the leader jumping forward (Fig.\ \ref{fig:jumps}c). The third term accounts for the case when a gap of length $g$ is created from a previous gap of length $y > g$ due to the laggard jumping forward, but still remaining the laggard (Fig.\ \ref{fig:jumps}d). The fourth term accounts for the case when a gap of length $g$ is created from a previous gap of length $y > 0$ due to the laggard jumping forward and overtaking the leader (Fig.\ \ref{fig:jumps}e). It is easy to check that this conserves the total probability, i.e. $\int_{0}^{\infty}{\dot{p}\left(g\right) dg} = 0$.

\subsubsection{Stationary distribution of the gap}

Our goal is to compute the stationary density, i.e.\ $p$ such that $\dot{p}\left(g\right) = 0$. It is unlikely to be able to find a closed formula for the stationary density in full generality (for general $\varphi$ and $w$), and so we restrict ourselves to the special case when $\varphi(x)=e^{-x} \mathbf{1} [x\geq 0]$ and $w\left( x \right) = e^{- \beta x}$, where $\beta > 0$. In this case the equation for stationarity becomes
\[
0 = -p\left(g\right) \left( e^{\frac{\beta g}{2}} + e^{- \frac{\beta g}{2}} \right) + \int\limits_{0}^{g}{p\left(y\right) e^{- \frac{\beta y}{2}} e^{-\left( g - y \right)} dy} + \int\limits_{g}^{\infty}{p\left(y\right) e^{\frac{\beta y}{2}} e^{-\left( y-g \right)} dy} + \int\limits_{0}^{\infty}{p\left(y\right) e^{\frac{\beta y}{2}} e^{-\left( y+g \right)} dy}.
\]
If we look at the last integral as the sum of the integral from $0$ to $g$ and the integral from $g$ to $\infty$, then the equation above can rearranged to get
\[
0 = -p\left(g\right) \left( e^{\frac{\beta g}{2}} + e^{- \frac{\beta g}{2}} \right) + \int\limits_{0}^{g}{p\left(y\right) e^{-g} \left[ e^{\left(- \frac{\beta}{2} + 1\right) y} + e^{\left( \frac{\beta}{2} - 1\right) y} \right] dy} + \int\limits_{g}^{\infty}{p\left(y\right) e^{\left(\frac{\beta}{2} - 1\right) y} \left[ e^{g} + e^{-g} \right] dy},
\]
or more simply, after dividing by 2,
\begin{equation}\label{eq:stac_2part_exp}
0 = -p\left(g\right) \cosh\left( \frac{\beta g}{2} \right) + e^{-g} \int\limits_{0}^{g}{p\left(y\right) \cosh\left( \left(\frac{\beta}{2} - 1\right) y \right) dy} + \cosh\left( g \right) \int\limits_{g}^{\infty}{p\left(y\right) e^{\left(\frac{\beta}{2} - 1\right) y} dy}.
\end{equation}
%\begin{align}
%0 = &-p\left(g\right) \cosh\left( \frac{\beta g}{2} \right) + e^{-g} \int\limits_{0}^{g}{p\left(y\right) \cosh\left( \left(\frac{\beta}{2} - 1\right) y \right) dy}\notag\\
%&+  \cosh\left( g \right) \int\limits_{g}^{\infty}{p\left(y\right) e^{\left(\frac{\beta}{2} - 1\right) y} dy}.\label{eq:stac_2part_exp}
%\end{align}

\subsubsection{\texorpdfstring{Special case of $\beta = 2$}{Special case of beta = 2}}

In the special case of $\beta = 2$, \eqref{eq:stac_2part_exp} simplifies to
\begin{equation}\label{eq:b2}
p\left( g \right) \cosh\left( g \right) = e^{-g} \int\limits_{0}^{g}{p\left(y\right) dy} + \cosh\left( g \right) \int\limits_{g}^{\infty}{p\left(y\right) dy}.
\end{equation}
This equation can be solved explicitly. Using the fact that $\int_{0}^{\infty} p\left(g\right) dg = 1$, and thus $\int_{g}^{\infty} p\left(y\right) dy = 1 - \int_{0}^{g} p\left(y\right) dy$, we arrive at
\begin{equation*}
p\left( g \right) \cosh\left( g \right) = - \sinh\left( g \right) \int\limits_{0}^{g}{p\left(y\right) dy} + \cosh\left( g \right),
\end{equation*}
and thus dividing by $\sinh\left( g \right)$ we have (for $g > 0$)
\begin{equation}\label{eq:b2_2}
\coth\left( g \right) p\left( g \right) = - \int\limits_{0}^{g}{p\left(y\right) dy} + \coth\left( g \right).
\end{equation}
Let us differentiate both sides of \eqref{eq:b2_2} with respect to $g$. After rearranging, using the identity $1/\sinh^{2}\left( g \right) - 1 = \coth^{2}\left( g \right) - 2$, and dividing by $\coth\left( g \right)$, we arrive at
\begin{equation*}%\label{eq:b2_3}
p'\left( g \right) = \left( \coth\left( g \right) - \frac{2}{\coth\left( g \right)} \right) p\left( g \right) - \frac{1}{\sinh\left( g \right) \cosh\left( g \right)}.
\end{equation*}
The general solution to this linear equation is
\begin{equation*}
p\left( g \right) = c \frac{\sinh\left( g \right)}{\cosh^{2}\left( g \right)} + \frac{1}{\cosh^{2}\left( g \right)}.
\end{equation*}
The restriction $\int_{0}^{\infty} p\left(g\right) dg = 1$ specifies the constant:
\begin{equation*}
\int\limits_{0}^{\infty} p\left(g\right) dg = \left[ -c \frac{1}{\cosh\left( g \right)} + \tanh\left( g \right) \right]_{0}^{\infty} = 1 + c,
\end{equation*}
thus $c=0$ and consequently the density $p$ is
\begin{equation}\label{eq:stac_eo_2part_beta2}
p\left( g \right) = \frac{1}{ \cosh^{2} \left( g \right) }.
\end{equation}

\subsubsection{\texorpdfstring{The case of $\beta \neq 2$}{The case of beta != 2}}

In the case of $\beta \neq 2$, one way of attacking \eqref{eq:stac_2part_exp} is differentiating both sides with respect to $g$ twice, and then using \eqref{eq:stac_2part_exp} again to get rid of the integrals. This way we arrive at a second-order homogeneous differential equation with non-constant coefficients:
\[
\cosh\left( \frac{\beta g}{2} \right) p''\left( g \right) + \left( \beta + 1 \right) \sinh\left( \frac{\beta g}{2} \right) p'\left( g \right) + \left( \frac{\beta^{2}}{4} + \frac{\beta}{2} \right) \cosh\left( \frac{\beta g}{2} \right) p\left( g \right) = 0.
\]
Dividing by $\cosh\left( \frac{\beta g}{2} \right)$:
\begin{equation}\label{eq:2orddiffb_2}
p''\left( g \right) + \left( \beta + 1 \right) \tanh\left( \frac{\beta g}{2} \right) p'\left( g \right) + \left( \frac{\beta^{2}}{4} + \frac{\beta}{2} \right) p\left( g \right) = 0.
\end{equation}
Before looking at the solutions to this, it is worth looking at large $g$ asymptotics. If $g$ is large, then $\tanh\left( \frac{\beta g}{2} \right) \approx 1$, and therefore the following equation approximately holds:
\begin{equation*}
p''\left( g \right) + \left( \beta + 1 \right) p'\left( g \right) + \left( \frac{\beta^{2}}{4} + \frac{\beta}{2} \right) p\left( g \right) = 0.
\end{equation*}
The two solutions to
\begin{equation*}
\lambda^{2} + \left( \beta + 1 \right) \lambda +  \left( \frac{\beta^{2}}{4} + \frac{\beta}{2} \right) = 0
\end{equation*}
are $\lambda_{1} = -\beta / 2$ and $\lambda_{2} = -\left( 1 + \beta / 2 \right)$, and therefore the density $p\left( g \right)$ is asymptotically a constant multiple of either $e^{- \beta g / 2}$ or $e^{-\left( 1 + \beta / 2 \right) g}$. Since for $\beta = 2$ we know by \eqref{eq:stac_eo_2part_beta2} that $p\left( g \right)$ is asymptotically $e^{-2g}$, we suspect that for general $\beta$ the density $p\left( g \right)$ is asymptotically a constant multiple of $e^{-\left( 1 + \beta / 2 \right) g}$. This suspicion will be confirmed shortly. By differentiating, it can be checked that a solution to \eqref{eq:2orddiffb_2} is
\begin{equation}\label{eq:sol1b}
p\left( g \right) = \frac{1}{\left( \cosh\left( \frac{\beta g}{2} \right) \right)^{1 + 2 / \beta}}.
\end{equation}
Using the method of variation of constants, the other solution can be found as well, and therefore the general solution to \eqref{eq:2orddiffb_2} is
\begin{equation}\label{eq:solb}
p\left( g \right) = c_{1} \frac{1}{\left( \cosh\left( \frac{\beta g}{2} \right) \right)^{1 + 2 / \beta}} + c_{2} \frac{\int\limits_{0}^{g}{ \left( \cosh\left( \frac{\beta y}{2} \right) \right)^{2 / \beta} dy} }{\left( \cosh\left( \frac{\beta g}{2} \right) \right)^{1 + 2 / \beta}},
\end{equation}
where $c_{1}$ and $c_{2}$ are constants. In our case, one constraint on the constants is the fact that $\int_{0}^{\infty} p\left(g\right) dg = 1$. The other constraint is given by the following. We know that $p\left( g \right)$ is a continuous density function on $\left( 0 , \infty \right)$, and due to \eqref{eq:stac_2part_exp} we know that
\begin{equation}\label{eq:g_to_0}
\lim\limits_{g \to 0+ } p \left( g \right) = \int\limits_{0}^{\infty} p \left( y \right) e^{\left( \frac{\beta}{2} - 1 \right) y } dy.
\end{equation}
So from \eqref{eq:solb} and \eqref{eq:g_to_0} the second constraint on the constants is
\begin{equation*}
c_{1} = c_{1} \int\limits_{0}^{\infty} \frac{ e^{\left( \frac{\beta}{2} - 1 \right) y } }{ \left( \cosh\left( \frac{\beta y}{2} \right) \right)^{1 + 2 / \beta} } dy + c_{2} \int\limits_{0}^{\infty} \frac{\int\limits_{0}^{y}{ \left( \cosh\left( \frac{\beta u}{2} \right) \right)^{2 / \beta} du} }{\left( \cosh\left( \frac{\beta y}{2} \right) \right)^{1 + 2 / \beta}} e^{\left( \frac{\beta}{2} - 1 \right) y } dy.
\end{equation*}
However, using the substitution $u = e^{\frac{\beta}{2} y}$ one can see that
\begin{equation*}
\int\limits_{0}^{\infty} \frac{ e^{\left( \frac{\beta}{2} - 1 \right) y } }{ \left( \cosh\left( \frac{\beta y}{2} \right) \right)^{1 + 2 / \beta} } dy = 1,
\end{equation*}
showing that $c_{2} = 0$. So for general $\beta$ the stationary density of the gap is
\begin{equation*}
p \left( g \right) = c \frac{1}{ \left( \cosh\left( \frac{\beta g}{2} \right) \right)^{1 + 2 / \beta} },
\end{equation*}
with $c$ being the normalizing constant.

\subsection{Many-particle case}

In the $n \geq 3$ case with deterministic jumps, the particles $x_{1}\left(t\right), \dots , x_{n}\left(t\right)$ form a continuous-time Markov process on an $n$ dimensional lattice. Viewed from the mean position of the particles, this becomes a Markov process on an $n-1$ dimensional lattice. However, even in the $n=3$ case the combinatorics of the problem becomes too difficult to find a closed form for the stationary distribution (see Figure \ref{fig:haromszogek_sz}). The model with the jump length distribution having a density presents similar difficulties in the $n \geq 3$ case. (As we have seen, the problem is not completely solved even in the $n=2$ case.) Understanding the finite particle case is the subject of future work.
\begin{figure}[ht]
\centering

\begin{pspicture}(12,12)
	%% Nodes
		%% Inner hexagon
		\cnodeput{0}(6,6){0}{\tiny{(0,0,0)}}
		\rput{0}(8.5,6){\circlenode{1}{\tiny{(1,-2,1)}}}
		\rput{0}(7.25,8.16506){\circlenode{2}{\tiny{(2,-1,-1)}}}
		\rput{0}(4.75,8.16506){\circlenode{3}{\tiny{(1,1,-2)}}}
		\rput{0}(3.5,6){\circlenode{4}{\tiny{(-1,2,-1)}}}
		\rput{0}(4.75,3.83493){\circlenode{5}{\tiny{(-2,1,1)}}}
		\rput{0}(7.25,3.83493){\circlenode{6}{\tiny{(-1,-1,2)}}}

%		%% Outer hexagon
		\cnodeput{0}(11,6){7}{\tiny{(2,-4,2)}}
		\cnodeput{0}(8.5,10.330127){9}{\tiny{(4,-2,-2)}}
		\cnodeput{0}(3.5,10.330127){11}{\tiny{(2,2,-4)}}
		\cnodeput{0}(1,6){13}{\tiny{(-2,4,-2)}}
		\cnodeput{0}(3.5,1.66987){15}{\tiny{(-4,2,2)}}
		\cnodeput{0}(8.5,1.66987){17}{\tiny{(-2,-2,4)}}
		
		\cnodeput{0}(9.75,8.16506){8}{\tiny{(3,-3,0)}}
		\cnodeput{0}(6,10.330127){10}{\tiny{(3,0,-3)}}
		\cnodeput{0}(2.25,8.16506){12}{\tiny{(0,3,-3)}}
		\cnodeput{0}(2.25,3.83493){14}{\tiny{(-3,3,0)}}
		\cnodeput{0}(6,1.66987){16}{\tiny{(-3,0,3)}}
		\cnodeput{0}(9.75,3.83493){18}{\tiny{(0,-3,3)}}
		
		%%Farout points
		\pnode(11.4,6.6928){out71}
		\pnode(11.8,6){out72}
		\pnode(11.4,5.30717){out73}
		\pnode(10.15,8.85788){out81}
		\pnode(10.55,8.16506){out82}
		\pnode(8.1,11.02294){out91}
		\pnode(8.9,11.02294){out92}
		\pnode(9.3,10.330127){out93}
		\pnode(5.6,11.02294){out101}
		\pnode(6.4,11.02294){out102}
		\pnode(2.7,10.330127){out111}
		\pnode(3.1,11.02294){out112}
		\pnode(3.9,11.02294){out113}
		\pnode(1.425,8.16506){out121}
		\pnode(1.825,8.85788){out122}
		\pnode(0.6,5.30717){out131}
		\pnode(0.2,6){out132}
		\pnode(0.6,6.6928){out133}
		\pnode(1.825,3.142109){out141}
		\pnode(1.425,3.83493){out142}
		\pnode(3.9,0.977049){out151}
		\pnode(3.1,0.977049){out152}
		\pnode(2.7,1.66987){out153}
		\pnode(6.4,0.977049){out161}
		\pnode(5.6,0.977049){out162}
		\pnode(9.3,1.66987){out171}
		\pnode(8.9,0.977049){out172}
		\pnode(8.1,0.977049){out173}
		\pnode(10.55,3.83493){out181}
		\pnode(10.15,3.142109){out182}
			
	%% Lines
		%% Inner hexagon inside
		\ncline[linecolor=blue]{<-}{2}{0}
		\ncput*{\tiny{$w(0)$}}
		\ncline[linecolor=blue]{<-}{4}{0}
		\ncput*{\tiny{$w(0)$}}
		\ncline[linecolor=blue]{<-}{6}{0}
		\ncput*{\tiny{$w(0)$}}
		\ncline[linecolor=red]{<-}{0}{1}
		\ncput*{\tiny{$w(-2)$}}
		\ncline[linecolor=red]{<-}{0}{3}
		\ncput*{\tiny{$w(-2)$}}
		\ncline[linecolor=red]{<-}{0}{5}
		\ncput*{\tiny{$w(-2)$}}
		
		%% Inner hexagon boundary
		\ncline[linecolor=green]{<-}{3}{2}
		\ncput*{\tiny{$w(-1)$}}
		\ncline[linecolor=green]{<-}{1}{2}
		\ncput*{\tiny{$w(-1)$}}
		\ncline[linecolor=green]{<-}{3}{4}
		\ncput*{\tiny{$w(-1)$}}
		\ncline[linecolor=green]{<-}{5}{4}
		\ncput*{\tiny{$w(-1)$}}
		\ncline[linecolor=green]{<-}{5}{6}
		\ncput*{\tiny{$w(-1)$}}
		\ncline[linecolor=green]{<-}{1}{6}
		\ncput*{\tiny{$w(-1)$}}
		
		%% Outer hexagon inside
		\ncline[linecolor=red]{<-}{1}{7}
		\ncput*{\tiny{$w(-4)$}}
		\ncline[linecolor=blue]{<-}{8}{1}
		\ncput*{\tiny{$w(1)$}}
		\ncline[linecolor=red]{<-}{2}{8}
		\ncput*{\tiny{$w(-3)$}}
		\ncline[linecolor=blue]{<-}{9}{2}
		\ncput*{\tiny{$w(2)$}}
		\ncline[linecolor=red]{<-}{2}{10}
		\ncput*{\tiny{$w(-3)$}}
		\ncline[linecolor=blue]{<-}{10}{3}
		\ncput*{\tiny{$w(1)$}}
		\ncline[linecolor=red]{<-}{3}{11}
		\ncput*{\tiny{$w(-4)$}}
		\ncline[linecolor=blue]{<-}{12}{3}
		\ncput*{\tiny{$w(1)$}}
		\ncline[linecolor=red]{<-}{4}{12}
		\ncput*{\tiny{$w(-3)$}}
		\ncline[linecolor=blue]{<-}{13}{4}
		\ncput*{\tiny{$w(2)$}}
		\ncline[linecolor=red]{<-}{4}{14}
		\ncput*{\tiny{$w(-3)$}}
		\ncline[linecolor=blue]{<-}{14}{5}
		\ncput*{\tiny{$w(1)$}}
		\ncline[linecolor=red]{<-}{5}{15}
		\ncput*{\tiny{$w(-4)$}}
		\ncline[linecolor=blue]{<-}{16}{5}
		\ncput*{\tiny{$w(1)$}}
		\ncline[linecolor=red]{<-}{6}{16}
		\ncput*{\tiny{$w(-3)$}}
		\ncline[linecolor=blue]{<-}{17}{6}
		\ncput*{\tiny{$w(2)$}}
		\ncline[linecolor=red]{<-}{6}{18}
		\ncput*{\tiny{$w(-3)$}}
		\ncline[linecolor=blue]{<-}{18}{1}
		\ncput*{\tiny{$w(1)$}}

		%% Outer hexagon boundary
		\ncline[linecolor=green]{<-}{7}{18}
		\ncput*{\tiny{$w(0)$}}
		\ncline[linecolor=green]{<-}{7}{8}
		\ncput*{\tiny{$w(0)$}}
		\ncline[linecolor=green]{<-}{8}{9}
		\ncput*{\tiny{$w(-2)$}}
		\ncline[linecolor=green]{<-}{10}{9}
		\ncput*{\tiny{$w(-2)$}}
		\ncline[linecolor=green]{<-}{11}{10}
		\ncput*{\tiny{$w(0)$}}
		\ncline[linecolor=green]{<-}{11}{12}
		\ncput*{\tiny{$w(0)$}}
		\ncline[linecolor=green]{<-}{12}{13}
		\ncput*{\tiny{$w(-2)$}}
		\ncline[linecolor=green]{<-}{14}{13}
		\ncput*{\tiny{$w(-2)$}}
		\ncline[linecolor=green]{<-}{15}{14}
		\ncput*{\tiny{$w(0)$}}
		\ncline[linecolor=green]{<-}{15}{16}
		\ncput*{\tiny{$w(0)$}}
		\ncline[linecolor=green]{<-}{16}{17}
		\ncput*{\tiny{$w(-2)$}}
		\ncline[linecolor=green]{<-}{18}{17}
		\ncput*{\tiny{$w(-2)$}}
		
		%% Short lines linking outwards
		\ncline[linecolor=blue]{-}{7}{out71}
		\ncline[linecolor=red]{<-}{7}{out72}
		\ncline[linecolor=blue]{-}{7}{out73}
		\ncline[linecolor=red]{<-}{8}{out82}
		\ncline[linecolor=blue]{-}{8}{out81}
		\ncline[linecolor=red]{<-}{9}{out91}
		\ncline[linecolor=blue]{-}{9}{out92}
		\ncline[linecolor=red]{<-}{9}{out93}
		\ncline[linecolor=red]{<-}{10}{out101}
		\ncline[linecolor=blue]{-}{10}{out102}
		\ncline[linecolor=blue]{-}{11}{out111}
		\ncline[linecolor=red]{<-}{11}{out112}
		\ncline[linecolor=blue]{-}{11}{out113}
		\ncline[linecolor=blue]{-}{12}{out121}
		\ncline[linecolor=red]{<-}{12}{out122}
		\ncline[linecolor=red]{<-}{13}{out131}
		\ncline[linecolor=blue]{-}{13}{out132}
		\ncline[linecolor=red]{<-}{13}{out133}
		\ncline[linecolor=red]{<-}{14}{out141}
		\ncline[linecolor=blue]{-}{14}{out142}
		\ncline[linecolor=blue]{-}{15}{out151}
		\ncline[linecolor=red]{<-}{15}{out152}
		\ncline[linecolor=blue]{-}{15}{out153}
		\ncline[linecolor=blue]{-}{16}{out161}
		\ncline[linecolor=red]{<-}{16}{out162}
		\ncline[linecolor=red]{<-}{17}{out171}
		\ncline[linecolor=blue]{-}{17}{out172}
		\ncline[linecolor=red]{<-}{17}{out173}
		\ncline[linecolor=red]{<-}{18}{out181}
		\ncline[linecolor=blue]{-}{18}{out182}		
		
\end{pspicture}

\caption{In the $n=3$ case with deterministic jumps the particles viewed from the mean position form a continuous-time Markov process on a 2 dimensional lattice with directed edges, as seen above. The coordinates in the nodes show $\left( 3\left(x_{1}\left(t\right) - m_{n}\left(t\right)\right), 3\left(x_{2}\left(t\right) - m_{n}\left(t\right)\right), 3\left(x_{3}\left(t\right) - m_{n}\left(t\right)\right) \right)$. (We multiply by 3 only so that we need not write fractions in the figure.) The rate of going from one node to another on a particular edge is indicated on the edge itself, however, once again the argument of the jump rate function $w\left( \cdot \right)$ is multiplied by 3 so that we do not need to write fractions in the figure. The edges going outward from the origin are colored in blue, the ones which are going inward toward the origin are colored in red, while the remaining edges are colored green. Since the rates on the incoming red arrows are larger than those on the outgoing blue arrows, one can see that there is a drift towards the origin, and therefore one expects there to be a stationary distribution on this lattice. However, there can be no detailed balance since there are directed cycles in this graph, and therefore calculating the stationary distribution becomes difficult.}
\label{fig:haromszogek_sz}
\end{figure}
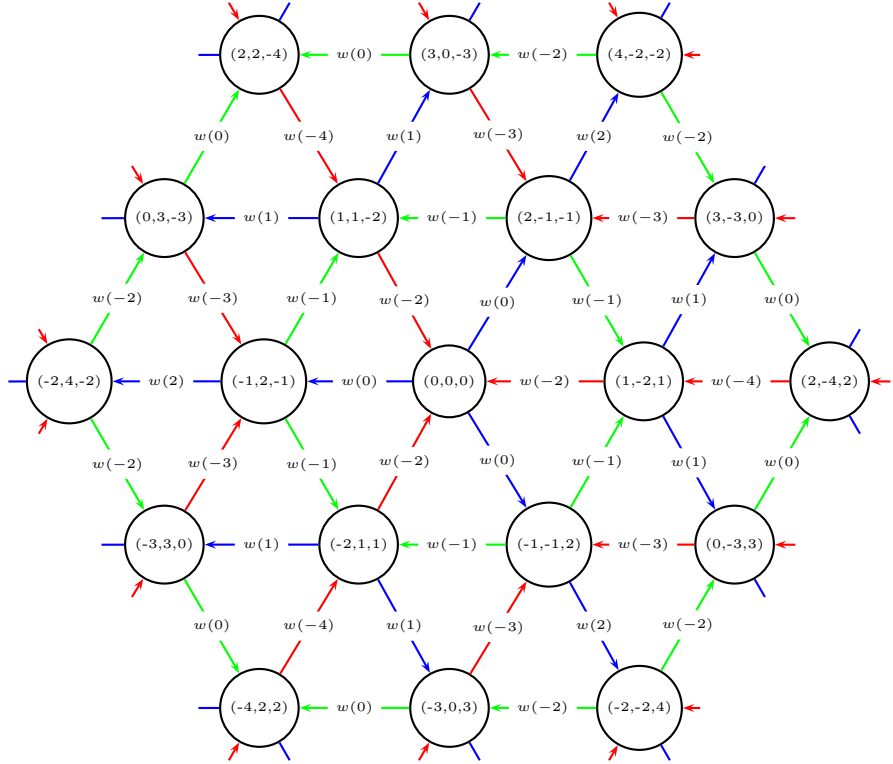

\afterpage{\clearpage}

\section{Mean field}\label{sec:mf}

We now turn our attention to the mean field approximation of the model and study the solutions of the mean field equation \eqref{eq:master}. First of all, using the fact that $\varphi$ is a probability density function, we can check that the mean field equation \eqref{eq:master} conserves the total probability, i.e.\ $\int_{-\infty}^{\infty} \pd{\varrho\left(x,t\right)}{t} dx = 0$.

An important quantity is the average speed of the mean of the density. This speed can be expressed as a simple integral by differentiating equation \eqref{eq:mean} with respect to time and using the mean field equation \eqref{eq:master}:
\begin{align}
\dot{m}\left(t\right) &= \int\limits_{-\infty}^{\infty} x \pd{\varrho\left(x,t\right)}{t} dx\notag\\
&= - \int\limits_{-\infty}^{\infty} x w\left(x-m\left(t\right)\right) \varrho\left(x,t\right) dx + \int\limits_{-\infty}^{\infty} \int\limits_{-\infty}^{x} x w\left(y-m\left(t\right)\right) \varrho\left(y,t\right) \varphi\left(x-y\right)dy dx\notag\\
&= - \int\limits_{-\infty}^{\infty} x w\left(x-m\left(t\right)\right) \varrho\left(x,t\right) dx + \int\limits_{-\infty}^{\infty} w\left(y-m\left(t\right)\right) \varrho\left(y,t\right) \int\limits_{y}^{\infty} x \varphi\left(x-y\right) dx dy\notag\\
&= - \int\limits_{-\infty}^{\infty} x w\left(x-m\left(t\right)\right) \varrho\left(x,t\right) dx +  \int\limits_{-\infty}^{\infty} w\left(y-m\left(t\right)\right) \varrho\left(y,t\right) dy + \int\limits_{-\infty}^{\infty} y w\left(y-m\left(t\right)\right) \varrho\left(y,t\right) dy\notag\\
&= \int\limits_{-\infty}^{\infty} w\left(y-m\left(t\right)\right) \varrho\left(y,t\right) dy.\label{eq:speed_general}
\end{align}

Our goal is to find a stationary solution to the mean field equation \eqref{eq:master} in the form of a traveling wave, see \eqref{eq:TravellingWave}. In this case $\rho$ is centered, since
\[
\int\limits_{-\infty}^{\infty} x \rho\left(x\right) dx = \int\limits_{-\infty}^{\infty} \left(x-ct\right) \rho\left(x-ct\right) dx = \int\limits_{-\infty}^{\infty} x \rho\left(x-ct\right) dx - ct = \int\limits_{-\infty}^{\infty} x \varrho\left(x,t\right) dx - ct = m\left(t\right) - ct = 0.\\
\]
We now turn to proving Theorem \ref{thm:exp_stac}.
\begin{proof}[Proof of Theorem \ref{thm:exp_stac}]
In the specific case when the jump length distribution is exponentially distributed we solve \eqref{eq:stationary} for general jump rate function $w$. Substituting $\varphi(x)=e^{-x} \mathbf{1} [x\geq 0]$ into \eqref{eq:stationary} we get
\begin{equation}
\label{eq:stationary_exp}
-c\rho'\left(x\right) = -w\left(x\right)\rho\left(x\right)+\int\limits_{-\infty}^{x} w\left(y\right)\rho\left(y\right)e^{-\left(x-y\right)}dy.
\end{equation}
First, we take the derivative of the integral term in \eqref{eq:stationary_exp}. Using the property of the exponential function, that its derivative is itself, we get
\[
\begin{aligned}
\frac{d}{dx}\left(\int\limits_{-\infty}^{x} w\left(y\right)\rho\left(y\right)e^{-\left(x-y\right)}dy\right) &= w\left(x\right)\rho\left(x\right) - \int\limits_{-\infty}^{x} w\left(y\right)\rho\left(y\right)e^{-\left(x-y\right)}dy\\
&= w\left(x\right)\rho\left(x\right) - w\left(x\right)\rho\left(x\right) + c\rho'\left(x\right) = c\rho'\left(x\right),
\end{aligned}
\]
where we used \eqref{eq:stationary_exp} in the second line. It follows that for some constant $D$ we have
\begin{equation*}
\int\limits_{-\infty}^{x} w\left(y\right)\rho\left(y\right) e^{-\left(x-y\right)}dy = c\rho\left(x\right) + D.
\end{equation*}
Plugging this into \eqref{eq:stationary_exp} we arrive at
\begin{equation}\label{eq:thm1biz_eq}
-c\rho'\left(x\right)+w\left(x\right)\rho\left(x\right) = c\rho\left(x\right) + D.
\end{equation}
To show that $D=0$, let us integrate both sides of the equation with respect to $x$ from $-\infty$ to $\infty$ (by integrating from $-R$ to $R$ and taking the limit as $R\to\infty$). The integral of the first term on the left hand side is $0$ since the total probability is conserved by the mean field equation, and the integral of the second term on the left hand side is $c$ according to \eqref{eq:speed_general}. Since $\rho$ is a density we also have that the integral of the first term on the right hand side is $c$ as well, therefore
\begin{equation*}
\lim_{R\to\infty} \int\limits_{-R}^{R} D dx = 0,
\end{equation*}
and thus $D=0$. Now rearranging \eqref{eq:thm1biz_eq} we get
\begin{equation*}
\rho'\left(x\right) = \left(\frac{1}{c}w\left(x\right)-1\right)\rho\left(x\right),
\end{equation*}
from which we arrive at \eqref{eq:stationary_dist}.

If $w$ is non-constant, then from \eqref{eq:speed_general} it follows that
\begin{equation*}
w\left(-\infty\right) > c > w\left(\infty\right).
\end{equation*}
From this it immediately follows that
% Therefore there exist an $x_{0} > 0$ and an $\varepsilon > 0$ such that for $x > x_{0}$ we have $\frac{1}{c}w\left(x\right) < 1 - \varepsilon$, and so for $x > x_{0}$ (and $K \geq 0$) we have
% \begin{equation*}
% \rho\left(x\right) \leq K e^{\int\limits_{0}^{x_{0}} \left(\frac{1}{c}w\left(s\right)-1\right) ds} e^{-\varepsilon \left(x-x_{0}\right)},
% \end{equation*}
% which means that 
$\rho$ decays at least exponentially as $x \to \pm \infty$.% The same can be seen as $x \to -\infty$, and therefore $\rho \in L^{1}\left(\mathbb{R}\right)$ so it can be normalized.

The last assertion of the theorem follows from the fact that $\rho$ is centered, which we have seen above.
\end{proof}
\begin{proof}[Alternative proof of Theorem \ref{thm:exp_stac} by I.P.\ T\'oth]
Let us consider a moving walkway that goes to the left with speed $c$, above which there is also a rope that goes to the right with speed $v >> c$ (see Figure \ref{fig:mogy}). We define the following stochastic process, which in the $v \to \infty$ limit is the same process that is described by \eqref{eq:stationary_exp}:
\begin{itemize}
\item All the particles are on the moving walkway going left with speed $c$ and a particle at $x$ jumps up and catches the rope with rate $w\left(x\right)$.
\item If a particle is hanging to the rope, then it lets go and falls back with rate $v$.
\end{itemize}

\def\ManLeft{%
	%% Little man going left
	\psline[linewidth=1pt](0,0)(0.3,0.5) % left leg
	\psline[linewidth=1pt](0.6,0)(0.3,0.5) % right leg
	\psline[linewidth=1pt](0.3,0.5)(0.3,1.2) % body
	\psline[linewidth=1pt](0.3,1)(0,0.7) % arm
%	\psline[linewidth=1pt](0.4,1)(0.3,0.5) % other arm
	\pscircle[linewidth=1pt](0.3,1.4){0.2} % head
}

\def\ManRight{%
	%% Little man going right
	\psline[linewidth=1pt](2.2,2.2)(1.4,1.4) % right arm
	\psline[linewidth=1pt](1.4,1.4)(1.52,1.76) % upper body
	\psline[linewidth=1pt](1.4,1.4)(1.04,1.16) % left arm
	\psline[linewidth=1pt](1.4,1.4)(1.16,0.92) %  lower body
	\psline[linewidth=1pt](1.16,0.92)(0.86,0.5) % right leg
	\psline[linewidth=1pt](1.16,0.92)(0.62,0.74) % left leg
	\pscircle[linewidth=1pt](1.58,1.94){0.1897} % head
}

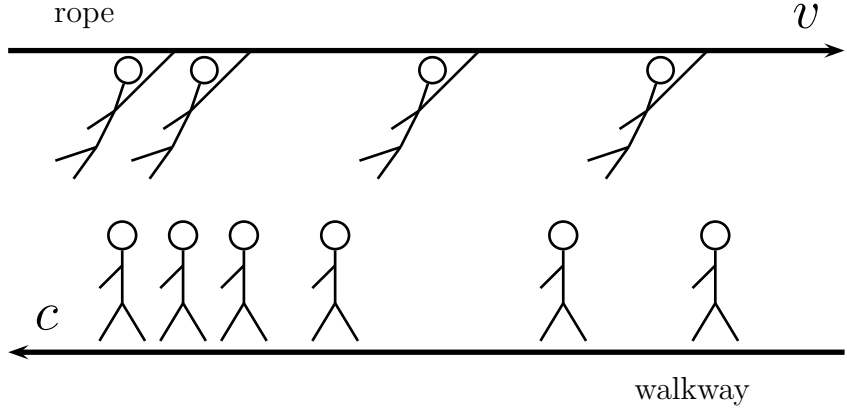
\begin{figure}[ht]
\centering

\begin{pspicture}(12,6)

	%% Rope
	\psline[linewidth=2pt]{->}(1,5)(12,5)
	\uput*{0.3}[90]{0}(11.5,5){\huge{$v$}}
	\uput*{0.3}[90]{0}(2,5){\large{rope}}
		
	%% People going right
	\uput*{0.2}[270]{0}(1,3){\ManRight}
	\uput*{0.2}[270]{0}(2,3){\ManRight}
	\uput*{0.2}[270]{0}(5,3){\ManRight}
	\uput*{0.2}[270]{0}(8,3){\ManRight}

	%% Walkway
	\psline[linewidth=2pt]{<-}(1,1)(12,1)
	\uput*{0.3}[90]{0}(1.5,1){\huge{$c$}}
	\uput*{0.3}[90]{0}(10,0){\large{walkway}}
	
	%% People going left
	\uput*{0.15}[90]{0}(2.2,1){\ManLeft}
	\uput*{0.15}[90]{0}(3,1){\ManLeft}
	\uput*{0.15}[90]{0}(3.8,1){\ManLeft}
	\uput*{0.15}[90]{0}(5,1){\ManLeft}
	\uput*{0.15}[90]{0}(8,1){\ManLeft}
	\uput*{0.15}[90]{0}(10,1){\ManLeft}

\end{pspicture}
\caption{The process described above. Particles move to the left with speed $c$ and a particle at $x$ jumps up and catches the rope with rate $w\left(x\right)$. The rope is moving to the right with speed $v >> c$. Particles on the rope let go with rate $v$.}\label{fig:mogy}
\end{figure}
%\afterpage{\clearpage}
It is clear that since the speed of the rope is $v$, the jump length of the particle to the right is exponentially distributed with parameter $v\cdot\frac{1}{v} = 1$, which is the same length as in our model. Furthermore, in the $v\to\infty$ limit the jump is instantaneous, which is exactly what we are studying. Let us denote by $\rho_{d}^{v}\left(x\right)$ and $\rho_{u}^{v}\left(x\right)$ the density of the particles down on the walkway and up on the rope, respectively. The sum of the two is a probability density therefore
\begin{equation*}
\int\limits_{-\infty}^{\infty} \left( \rho_{d}^{v}\left(x\right) + \rho_{u}^{v}\left(x\right) \right) dx = 1.
\end{equation*}
Our plan of action is the following: we calculate the stationary densities of this process for fixed finite $v$, and then look at the result as $v \to \infty$. What we expect is that $\lim_{v\to\infty} \rho_{u}^{v} \left(x\right) = 0$ and $\lim_{v\to\infty} \rho_{d}^{v} \left(x\right) = \rho\left(x\right)$, where $\rho\left(x\right)$ is the solution to \eqref{eq:stationary_exp}. The master equations governing the time evolution of the densities $\rho_{d}^{v}$ and $\rho_{u}^{v}$ are
\begin{subequations}\label{eq:mogy1}
\begin{align}
\pd{\rho_{d}^{v}\left(x\right)}{t} &= c \pd{\rho_{d}^{v}\left(x\right)}{x} - w\left(x\right) \rho_{d}^{v}\left(x\right) + v \rho_{u}^{v}\left(x\right)\label{eq:mogy1a}\\
\pd{\rho_{u}^{v}\left(x\right)}{t} &= \left(-v\right) \pd{\rho_{u}^{v}\left(x\right)}{x} + w\left(x\right) \rho_{d}^{v}\left(x\right) - v \rho_{u}^{v}\left(x\right).\label{eq:mogy1b}
\end{align}
\end{subequations}
The first term on the right hand sides corresponds to the particles that move horizontally (along the walkway or along the rope), while the second term corresponds to the particles that jump up and the third term corresponds to the particles that jump down. Setting the left hand sides to $0$ in \eqref{eq:mogy1} we arrive at the equations for the stationary densities:
\begin{subequations}\label{eq:mogy2}
\begin{align}
0 &= c \pd{\rho_{d}^{v}\left(x\right)}{x} - w\left(x\right) \rho_{d}^{v}\left(x\right) + v \rho_{u}^{v}\left(x\right)\label{eq:mogy2a}\\
0 &= \left(-v\right) \pd{\rho_{u}^{v}\left(x\right)}{x} + w\left(x\right) \rho_{d}^{v}\left(x\right) - v \rho_{u}^{v}\left(x\right).\label{eq:mogy2b}
\end{align}
\end{subequations}
Summing \eqref{eq:mogy2a} and \eqref{eq:mogy2b} we get:
\begin{equation*}
0 = \frac{\partial}{\partial x} \left( c\rho_{d}^{v}\left(x\right) - v\rho_{u}^{v}\left(x\right)  \right),
\end{equation*}
and thus
\begin{equation*}
K = c\rho_{d}^{v}\left(x\right) - v\rho_{u}^{v}\left(x\right)
\end{equation*}
for some $K$ constant. However, taking the limit $x\to\infty$ we see that $K = 0$ and thus
\begin{equation}
\label{eq:up_down}
\rho_{u}^{v}\left(x\right) = \frac{c}{v} \rho_{d}^{v}\left(x\right).
\end{equation}
Putting \eqref{eq:up_down} back into \eqref{eq:mogy2a} and rearranging we arrive at
\begin{equation*}
\pd{\rho_{d}^{v}\left(x\right)}{x} = \left(\frac{1}{c}w\left(x\right)-1\right)\rho_{d}^{v}\left(x\right),
\end{equation*}
from which
\begin{equation*}
\rho_{d}^{v}\left(x\right) = Ke^{\int\limits_{0}^{x} \left(\frac{1}{c}w\left(s\right)-1\right) ds}
\end{equation*}
follows for some $K$ constant. Note how $\rho_{d}^{v}$ is independent of $v$, and also that $\lim_{v\to\infty} \rho_{u}^{v} \left(x\right) = 0$ as expected.
\end{proof}

%%%%%%%%%%%%%%%%%%%%%%

\subsection{Specific examples}

%%%%%%%%%%%%%%%%%%%%%%

In this subsection we look at the stationary distribution $\rho$ of \eqref{eq:stationary_dist} in the case of specific jump rate functions. In all four cases we also determine the speed of the traveling wave from Theorem \ref{thm:exp_stac}. Details of calculations are given only in the case of Example \ref{ex:exp}, the calculations in the other examples are quite simple from Theorem \ref{thm:exp_stac}.

\begin{example}[\textmd{\textit{Exponential jump rate function}}]\label{ex:exp}
If $w\left(x\right)=e^{-\beta x}$, where $\beta > 0$, then from Theorem \ref{thm:exp_stac} we obtain
\begin{equation*}
\rho_{\beta}\left(x\right) = K e^{- x - e^{-\beta\left(x+\frac{\log{\left(c\beta\right)}}{\beta}\right)}},
\end{equation*}
where $K$ is an appropriate normalizing constant. Integrating this over $\mathbb{R}$ gives us the value for $K$ and thus the density becomes
\begin{equation*}
\rho_{\beta}\left(x\right) = \frac{\beta}{\Gamma\left(\frac{1}{\beta}\right)}e^{-\left(x+\frac{\log{\left(c\beta\right)}}{\beta}\right)-e^{-\beta\left(x+\frac{\log{\left(c\beta\right)}}{\beta}\right)}}.
\end{equation*}
To check this let us calculate the integral:
\[
\begin{aligned}
\int\limits_{-\infty}^{\infty}\rho_{\beta}\left(x\right)dx &= \int\limits_{-\infty}^{\infty}\frac{\beta}{\Gamma\left(\frac{1}{\beta}\right)}e^{-\left(x+\frac{\log{\left(c\beta\right)}}{\beta}\right)-e^{-\beta\left(x+\frac{\log{\left(c\beta\right)}}{\beta}\right)}}dx\\
&= \frac{\beta}{\Gamma\left(\frac{1}{\beta}\right)}\int\limits_{-\infty}^{\infty}e^{-y-e^{-\beta y}}dy = \frac{\beta}{\Gamma\left(\frac{1}{\beta}\right)}\int\limits_{\infty}^{0}z^{\frac{1}{\beta}}e^{-z}\left(-\frac{dz}{\beta z}\right) =\frac{1}{\Gamma\left(\frac{1}{\beta}\right)}\int\limits_{0}^{\infty}z^{\frac{1}{\beta}-1}e^{-z}dz = 1,
\end{aligned}
\]
using the substitutions $y = x+\frac{\log{\left(c\beta\right)}}{\beta}$ and $z = e^{-\beta y}$. From Theorem \ref{thm:exp_stac} the speed of the wave $c$ is determined by the fact that $\rho_{\beta}$ is centered, so to compute $c$ we must compute the mean of the density. Using similar substitutions as before we get:
\[
\begin{aligned}
\int\limits_{-\infty}^{\infty}x\rho_{\beta}\left(x\right)dx &= -\frac{\log{\left(c\beta\right)}}{\beta} + \int\limits_{-\infty}^{\infty}\left(x+\frac{\log{\left(c\beta\right)}}{\beta}\right)\rho_{\beta}\left(x\right)dx = -\frac{\log{\left(c\beta\right)}}{\beta} + \frac{\beta}{\Gamma\left(\frac{1}{\beta}\right)}\int\limits_{-\infty}^{\infty}ye^{-y-e^{-\beta y}}dy\\
&= -\frac{\log{\left(c\beta\right)}}{\beta} + \frac{\beta}{\Gamma\left(\frac{1}{\beta}\right)}\int\limits_{\infty}^{0}-\frac{\log{\left(z\right)}}{\beta}z^{\frac{1}{\beta}}e^{-z}\left(-\frac{dz}{\beta z}\right) = -\frac{\log{\left(c\beta\right)}}{\beta} - \frac{1}{\beta\Gamma\left(\frac{1}{\beta}\right)}\int\limits_{0}^{\infty}\log{\left(z\right)}z^{\frac{1}{\beta}-1}e^{-z}dz\\
&= -\frac{1}{\beta}\left(\log{\left(c\beta\right)}+\frac{\Gamma'\left(\frac{1}{\beta}\right)}{\Gamma\left(\frac{1}{\beta}\right)}\right),
\end{aligned}
\]
which is zero when
\begin{equation*}
c=\frac{1}{\beta}e^{-\psi\left(\frac{1}{\beta}\right)},
\end{equation*}
where $\psi$ is the digamma function:
\begin{equation*}
\psi\left(x\right) = \frac{\Gamma'\left(x\right)}{\Gamma\left(x\right)}.
\end{equation*}
Therefore the stationary distribution in this case is given by \eqref{eq:rhobeta1.1}.
%\begin{equation}
%\label{eq:rhobeta1.1}
%\rho_{\beta}\left(x\right) = \frac{\beta}{\Gamma\left(\frac{1}{\beta}\right)}e^{-\left(x-\frac{\psi\left(\frac{1}{\beta}\right)}{\beta}\right)-e^{-\beta\left(x-\frac{\psi\left(\frac{1}{\beta}\right)}{\beta}\right)}}.
%\end{equation}
Note the exponential tail of $\rho_{\beta}$ on one side, and the double-exponential tail on the other side. Specifically, in the case of $\beta = 1$ the stationary distribution is the centered standard Gumbel distribution:
\begin{equation*}
\rho_{1}\left(x\right) = e^{ -\left(x + \gamma\right) - e^{ -\left(x + \gamma\right) } },
\end{equation*}
where $\gamma$ is Euler's constant. Moreover, it can be seen that $\rho_{\beta}$ is a rescaled generalized Gumbel distribution with parameter $1/\beta$ (Bertin \cite{bertin2005gfgs}, Clusel and Bertin \cite{clusel2008global}, Gumbel \cite{gumbel}). As mentioned in Section \ref{sec:results}, the (generalized) Gumbel distribution often arises in extreme value theory, so it is natural to ask whether there is an underlying extremal process in the dynamics. It turns out that there is a connection between the mean field model and extreme value theory when $k = 1/\beta$ is an integer: this is discussed in Section \ref{sec:gumbel}.
\end{example}

\begin{example}[\textmd{\textit{Jump rate function is a step function}}]\label{ex:step}
Let
\begin{equation*}
w\left(x\right) = 
\begin{cases}
a & \text{if $x < 0$,}
\\
b &\text{if $0 \leq x$,}
\end{cases}
\end{equation*}
where $a>b>0$. Then from Theorem \ref{thm:exp_stac} the solution to \eqref{eq:stationary_exp} is
\begin{equation*}
\rho_{a,b}\left(x\right) =
\begin{cases}
K e^{\left(\frac{a}{c} - 1\right) x} & \text{if $x < 0$,}
\\
K e^{\left(\frac{b}{c} - 1\right) x} & \text{if $0 \leq x$,}
\end{cases}
\end{equation*}
where the normalizing constant is
\begin{equation*}
K = \frac{\left(a - c\right) \left(c - b\right)}{c\left(a - b\right)}.
\end{equation*}
The speed of the wave follows from Theorem \ref{thm:exp_stac}:
\begin{equation*}
c = \frac{a+b}{2},
\end{equation*}
and thus the stationary distribution is
\begin{equation}\label{eq:stac_step}
\rho_{a,b}\left(x\right) = \frac{a-b}{2\left(a+b\right)} e^{-\frac{a-b}{a+b}|x|},
\end{equation}
which is a Laplace distribution.
\end{example}

\begin{example}[\textmd{\textit{Jump rate function is ...}}]\label{ex:lin}
Let
\begin{equation}\label{eq:jump_rate_lin}
w\left(x\right) = 
\begin{cases}
a & \text{if $x < 1$,}
\\
-\frac{a-b}{2}x + \frac{a+b}{2} & \text{if $-1 \leq x \leq 1$,}
\\
b &\text{if $1 < x$,}
\end{cases}
\end{equation}
where $a > b > 0$, thus $w$ is a continuous function which is constant on the intervals $(-\infty,-1)$ and $(1,\infty)$ and linear on $\left[-1,1\right]$. Then from Theorem \ref{thm:exp_stac} the solution to \eqref{eq:stationary_exp} is
\begin{equation*}
\rho_{a,b}\left(x\right) = 
\begin{cases}
K e^{\frac{a-b}{4c}+\left(\frac{a}{c}-1\right)x} & \text{if $x < -1$,}
\\
K e^{-\frac{a-b}{4c}x^2+\left(\frac{a+b}{2c}-1\right)x} & \text{if $-1 \leq x \leq 1$,}
\\
K e^{\frac{a-b}{4c}+\left(\frac{b}{c}-1\right)x} &\text{if $1 < x$,}
\end{cases}
\end{equation*}
where K is the normalizing constant. Note that for $c = \frac{a+b}{2}$ the density $\rho_{a,b}$ is symmetric, therefore according to Theorem \ref{thm:exp_stac} this is the speed of the wave. Therefore the stationary distribution is
\begin{equation*}
\rho_{a,b}\left(x\right) = 
\begin{cases}
K e^{\frac{a-b}{2\left(a+b\right)}-\frac{a-b}{a+b}|x|} & \text{if $|x| > 1$,}
\\
K e^{-\frac{a-b}{2\left(a+b\right)}x^2} & \text{if $|x| \leq 1$.}
\end{cases}
\end{equation*}
This distribution has exponential tails, but around $0$ it behaves as a normal distribution.
\end{example}

\begin{example}[\textmd{\textit{Jump rate function is arccotangent}}]\label{ex:arccot}
Let
\begin{equation*}
w\left(x\right) = \arccot{x},
\end{equation*}
where $\arccot{x}$ is defined as the inverse of $\cot{x}$ on the interval $\left(0,\pi\right)$, therefore $w$ is a monotone decreasing smooth function. From Theorem \ref{thm:exp_stac} the solution to \eqref{eq:stationary_exp} is
\begin{equation*}
\rho\left(x\right) = K e^{x\left(\frac{1}{c}\arccot{x} - 1\right) + \frac{1}{2c}\log{\left(1+x^2\right)}},
\end{equation*}
where K is the normalizing constant. Note that for $c = \frac{\pi}{2}$ the density $\rho$ is symmetric, therefore according to Theorem \ref{thm:exp_stac} this is the speed of the wave.
Consequently the stationary distribution is
\begin{equation*}
\rho\left(x\right) = K e^{x\left(\frac{2}{\pi}\arccot{x} - 1\right) + \frac{1}{\pi}\log{\left(1+x^2\right)}}.
\end{equation*}
\end{example}

%%%%%%%%%%%%%%%%%%%%%%

\subsection{Connection with extreme value theory}\label{sec:gumbel}

In the previous subsection we have seen that when $\varphi(x)=e^{-x} \mathbf{1} [x\geq 0]$ and $w\left(x\right) = e^{- \beta x}$ then the (rescaled generalized) Gumbel distribution arises as the stationary distribution of the mean field system. This is interesting, because at first sight there is no extremeness in the model. In this section we discuss the connection between the model and extreme value theory.

\subsubsection{Standard Gumbel case}

Let us start with the standard Gumbel case, i.e.\ $\beta = 1$. Recall that according to our observation in the introduction, we can assume that the center of mass moves with a constant velocity $c$, and we consider a single mean field particle $X(t)$ that jumps with rate $e^{ct - X\left(t\right)}$, so
\begin{equation*}
\p\left(X\left(t\right) \text{ changes between } t \text{ and } t+dt\right) =  e^{ct - X\left(t\right)} dt + o\left(dt\right).
\end{equation*}
When this jump occurs, $X(t)$ increases by a random number drawn from an exponential distribution with parameter 1 (denoted by Exp(1) in the sequel).

How can we view this as an extreme process? The answer to this question was given by Attila R\'akos, whose idea is the following. As time goes forward, let us take gradually more and more new independent and identically distributed Exp(1) random variables, and let us look at the distribution of the maximum of these random variables, which we denote by $Y\left(t\right)$. Specifically, at time $t$ let there be $N\left(t\right) = e^{ct} / c$ random variables altogether, so $Y\left(t\right)$ is the maximum of $N\left(t\right)$ i.i.d.\ random variables. (In this subsection we neglect mathematical precision (such as the fact that we should have defined $N\left(t\right)$ to be integer-valued: $N\left(t\right) = \left[e^{ct} / c\right]$) for the sake of clarity of the argument.) There are $N'\left(t\right) dt = e^{ct} dt$ new random variables trying to ``break the record'' between $t$ and $t + dt$, and thus
\begin{align*}
\p\left(Y\left(t\right) \text{ changes between } t \text{ and } t+dt\right) &= 1 - \left( 1 - e^{-Y\left(t\right)} \right)^{N'\left(t\right) dt}\\
&=  N'\left(t\right) dt e^{-Y\left(t\right)} + o\left(N'\left(t\right) dt e^{-Y\left(t\right)}\right)\\
&= e^{ct - Y\left(t\right)} dt + o\left(e^{ct - Y \left( t \right)} dt\right),
\end{align*}
and if $Y\left(t\right)$ changes, then it jumps according to Exp(1). Here we use the memoryless property of the exponential distribution, i.e.\ the fact that if $T$ is exponentially distributed then
\begin{equation*}
\p\left( T > s+t\; | \; T > s \right) = \p\left( T > t\right).
\end{equation*}

Thus we can conclude that the two processes are basically the same. Furthermore, we know that $Y\left(t\right) - \log N\left(t\right)$ follows the standard Gumbel distribution for large $t$ (see, for example, Gumbel \cite{gumbel} or the excellent survey by Clusel and Bertin \cite{clusel2008global}):
\begin{equation*}
\lim\limits_{t\to\infty} \p\left(Y\left(t\right) - \log N\left(t\right) < x \right) = e^{-e^{-x}},
\end{equation*}
and therefore since $\log N\left(t\right) = ct - \log c$, it follows that
\begin{equation*}
\lim\limits_{t\to\infty} \p\left(Y\left(t\right) - ct < x \right) = e^{-e^{-\left(x + \log c \right)}}.
\end{equation*}
Thus by the connection we have
\begin{equation*}
\lim\limits_{t\to\infty} \p\left(X\left(t\right) - ct < x \right) = e^{-e^{-\left(x + \log c \right)}},
\end{equation*}
and since the distribution of $X\left(t\right) - ct$ is centered, it follows that
\begin{equation*}
c = e^{-\gamma},
\end{equation*}
where $\gamma = 0.5772\dots$ is Euler's constant.

\subsubsection{Generalized Gumbel case}

The idea above extends to the general case with a slight modification. In the following let $k = 1/ \beta$ be a fixed positive integer. We use $k$ and $1 / \beta$ interchangeably.

In general the mean field particle $X\left(t\right)$ jumps with rate $e^{\beta \left( ct - X\left(t\right) \right) }$, so
\begin{equation*}
\p\left(X\left(t\right) \text{ changes between } t \text{ and } t+dt\right) =  e^{\beta \left( ct - X\left(t\right) \right)} dt + o\left(dt\right),
\end{equation*}
and if it changes then it jumps according to Exp(1).

How can we view this as an extreme process? The idea is similar to the one in the previous section. Once again, as time goes forward, let us take gradually more and more new independent and identically distributed Exp(1) random variables. Specifically, at time $t$ let there be $N\left( t \right) = e^{\beta c t} / \beta c$ random variables altogether. Denote by $Y_{k}\left( t \right)$ the $k^{\text{th}}$ largest value among these random variables at time $t$. We focus on $Y\left( t \right) := k Y_{k} \left( t \right)$. There are $N'\left( t \right) dt = e^{\beta ct} dt$ new random variables trying to ``break the record'' between $t$ and $t+dt$, and since $Y\left( t \right)$ changes exactly at the same time $Y_{k}\left( t \right)$ does, we have that
\begin{align*}
\p\left(Y\left(t\right) \text{ changes between } t \text{ and } t+dt\right) &= 1 - \left( 1 - e^{-Y_{k}\left(t\right)} \right)^{N'\left(t\right) dt}\\
&=  N'\left(t\right) dt e^{-Y_{k}\left(t\right)} + o\left(N'\left(t\right) dt e^{-Y_{k}\left(t\right)}\right)\\
&= e^{\beta ct - Y_{k}\left(t\right)} dt + o\left(e^{\beta ct - Y_{k} \left( t \right)} dt\right)\\
&= e^{\beta \left( ct - Y\left( t \right) \right)} dt + o\left(e^{\beta \left( ct - Y\left( t \right) \right)} dt\right),
\end{align*}
and if $Y\left( t \right)$ changes, then it jumps according to Exp(1). This follows directly from the fact that when $Y_{k} \left( t \right)$ changes, then it jumps according to Exp($k$), which is due to the following: $Y_{k}\left( t \right)$ changes when a new random variable jumps past the previous $k^{\text{th}}$ largest r.v., making the previous $k^{\text{th}}$ largest r.v.\ the new $\left(k+1\right)^{\text{st}}$ largest r.v., and thus the length of the jump that $Y_{k}\left( t \right)$ makes is the difference between the new $k^{\text{th}}$ and $\left(k+1\right)^{\text{st}}$ largest random variables. Since the random variables are i.i.d.\ Exp(1) variables, we know that this difference is exponentially distributed with parameter $k$ (see, for example, Bertin \cite{bertin2005gfgs}).

Thus we can conclude that the two processes $X\left( t \right)$ and $Y\left( t \right)$ are basically the same. Furthermore, we know that $Y_{k}\left( t \right)$, once rescaled, follows the generalized Gumbel distribution for large $t$. Specifically (see Clusel and Bertin \cite{clusel2008global} for details):
\begin{equation*}
\lim\limits_{t\to\infty} \p\left(Y_{k}\left(t\right) - \log N\left(t\right) \in \left( x,x+dx \right) \right) = \frac{1}{ \left(k-1\right)! } e^{ - kx - e^{-x}}dx.
\end{equation*}
Thus for $Y\left( t \right)$ we have
\[
\lim\limits_{t\to\infty} \p\left(Y\left(t\right) - k \log N\left(t\right) \in \left( x,x+dx \right) \right) = \frac{1}{ \left(k-1\right)! } e^{ - k\frac{x}{k} - e^{-\frac{x}{k}}} \frac{dx}{k} = \frac{1/k}{ \Gamma\left( k \right) } e^{ - x - e^{ - \frac{1}{k} x } } dx,
\]
and therefore for the mean field particle $X\left( t \right)$ we have that
\begin{equation*}
\lim\limits_{t\to\infty} \p\left(X\left(t\right) - \frac{1}{\beta} \log N\left(t\right) \in \left( x,x+dx \right) \right) = \frac{\beta}{ \Gamma\left( 1/ \beta \right) } e^{ - x - e^{ - \beta x } } dx,
\end{equation*}
so using the fact that $\log N\left(t\right) = \beta ct - \log \beta c$, we finally obtain
\begin{equation*}
\lim\limits_{t\to\infty} \p\left(X\left(t\right) - ct \in \left( x,x+dx \right) \right) = \frac{\beta}{ \Gamma\left( 1/ \beta \right) } e^{ - \left( x + \frac{\log \beta c}{\beta} \right) - e^{ - \beta \left( x + \frac{\log \beta c}{\beta} \right) } } dx.
\end{equation*}
Since the distribution of $X\left( t \right) - ct$ is centered, it follows that
\begin{equation*}
c = \frac{1}{\beta} e^{ - \psi \left(\frac{1}{\beta}\right) },
\end{equation*}
where $\psi$ is the digamma function:
\begin{equation*}
\psi \left(x\right) = \frac{ \Gamma'\left(x\right) }{ \Gamma\left(x\right) }.
\end{equation*}

Thus we have rediscovered the results of the previous subsection through extreme value theory arguments.

\section{Fluid limit}\label{sec:fluid_limit}

In this section we prove Theorem \ref{thm:main}. The section is organized as follows. In order to present and prove our results, we need to introduce notation, terminology, definitions and previously known results. This necessary background is introduced in Section \ref{ch:Preliminaries}. Then in Section \ref{ch:Results} we formulate as separate theorems several intermediate steps that lead to the main result and we prove Theorem \ref{thm:main} using these steps. In the following subsections we then provide proofs of the intermediate steps stated in Section \ref{ch:Results}.

%%%%%%%%%%%%%%%%%%%%%%%%%%%%%%%%%%%%%%%%%%%%%%%%%%%%%%%%%%%%%%%%%%%%%%%%%
%%%%%%%%%%%%%%%%%%%%%%%%%%%%%%%%%%%%%%%%%%%%%%%%%%%%%%%%%%%%%%%%%%%%%%%%%

\subsection{Preliminaries}\label{ch:Preliminaries}

Recall that our goal is to show that---given some assumptions---the path of the empirical measure process $\left\{\mu_{n} \left( \cdot \right) \right\}_{n \geq 1}$ \emph{converges weakly} in the Skorohod space $D\left( [0,\infty), \mathcal{P}_{1} \left( \mathbb{R} \right) \right)$ to $\mu\left( \cdot \right)$, which solves the mean field equation \eqref{eq:Atf}. The goal of this subsection is to provide all the necessary preliminaries that we need during the proof.

We present important results from the theory of weak convergence that we use later; in particular, we focus on tightness results in Skorohod spaces. For an excellent introduction to weak convergence, we refer the reader to Billingsley \cite{billingsley1999convergence}. Many results that we use are taken from Ethier and Kurtz \cite{ethier1986markov}, and Jacod and Shiryaev \cite{jacod1987limit} is also a valuable source of information concerning convergence of processes.

As in Theorem \ref{thm:main}, from now on we assume that the jump rate function $w$ is bounded, and we denote its lowest upper bound by $a$. To put it another way: $a:= \lim_{x \to -\infty} w\left( x \right) = \sup_x w\left( x \right)$. We also assume that $w$ is differentiable and has a bounded derivative. Furthermore, in the following $Z$ always denotes a random variable from the jump length distribution. As before, we assume that $\E \left( Z \right) = 1$ and that $Z$ has a finite third moment.

%%%%%%%%%%%%%%%%%%%%%%%%%%%%%%%%%%%%%%%%%%%%%%%%%%%%%%%%%%%%%%%%%%%%%%%%%
%%%%%%%%%%%%%%%%%%%%%%%%%%%%%%%%%%%%%%%%%%%%%%%%%%%%%%%%%%%%%%%%%%%%%%%%%

\subsubsection{Probability metrics}\label{ch:ProbMetrics}

Recall the definition \eqref{eq:was_metric} of the 1-Wasserstein metric. The following lemma (from Feng and Kurtz \cite[Appendix D]{feng2006large}) summarizes what we need to know about the metric space $\left( \mathcal{P}_{1} \left( \mathbb{R} \right), d_{1} \right)$. More can be found in e.g.\ Feng and Kurtz \cite[Appendix D]{feng2006large}.
\begin{lemma}\label{lem:was_metric}
$\left( \mathcal{P}_{1} \left( \mathbb{R} \right), d_{1} \right)$ is a complete and separable metric space. For $\mu_{n}, \mu \in \mathcal{P}_{1} \left( \mathbb{R} \right)$, the following are equivalent:
\begin{enumerate}[(1)]
\item $\lim_{n \to \infty} d_{1} \left( \mu_{n}, \mu \right) = 0$.
\item For all continuous functions $\varphi$ on $\mathbb{R}$ satisfying $\left| \varphi \left( x \right) \right| \leq C \left( 1 + \left|x \right| \right)$ for some $C > 0$, 
\begin{equation*}
\lim\limits_{n \to \infty} \left\langle \varphi, \mu_{n} \right\rangle = \left\langle \varphi, \mu \right\rangle.
\end{equation*}
\end{enumerate}
\end{lemma}

We introduce another probability metric that we use later on.
For $P, Q \in \mathcal{P}_{1} \left( \mathbb{R} \right)$ and a set $\mathcal{H}$ of functions define
\begin{equation}\label{eq:prob_metric}
d_{\mathcal{H}}\left( P, Q \right) = \sup\limits_{f \in \mathcal{H}} \left| \left\langle f, P \right\rangle - \left\langle f, Q \right\rangle \right|.
\end{equation}
It is immediate that $d_{\mathcal{H}}$ is a pseudometric on $\mathcal{P}_{1}\left( \mathbb{R} \right)$: $d_{\mathcal{H}}\left( P, P \right) = 0$, $d_{\mathcal{H}} \left( P, Q \right) = d_{\mathcal{H}} \left( Q, P \right)$, and the triangle inequality also holds. The only thing that does not necessarily hold is that $d_{\mathcal{H}} \left( P, Q \right) = 0$ implies $P = Q$. In order for this to hold, $\mathcal{H}$ has to be a large enough set of functions to be able to distinguish between different probability measures. There are various choices of $\mathcal{H}$ that give rise to notable metrics. For instance, if $\mathcal{H}$ is the set of indicator functions of measurable sets, then $d_{\mathcal{H}}$ is the total variation metric. If $\mathcal{H}$ is the set of half-line indicator functions then $d_{\mathcal{H}}$ is the Kolmogorov (uniform) metric. The choice of $\mathcal{H} = \left\{ f \in C_{b}:\ \left| f \right| \leq 1 \right\}$ gives the Radon metric (that this is a metric follows from e.g.\ Billingsley \cite[Theorem 1.2.]{billingsley1999convergence}). Consequently the choice $\mathcal{H} = H$ (see \eqref{eq:H_def}) also yields a metric. This is what we use in the following: $d_H$ denotes the metric on $\mathcal{P}_{1} \left( \mathbb{R} \right)$ given by \eqref{eq:prob_metric} with $\mathcal{H} = H$. This metric is used in Theorem \ref{thm:cont_dep_init_cond}, Corollary \ref{cor:uniqueness} and their proofs in Section \ref{ch:Uniqueness}. For more on probability metrics we refer the reader to Gibbs and Su \cite{gibbs2002choosing} for an excellent review.

%%%%%%%%%%%%%%%%%%%%%%%%%%%%%%%%%%%%%%%%%%%%%%%%%%%%%%%%%%%%%%%%%%%%%%%%%
%%%%%%%%%%%%%%%%%%%%%%%%%%%%%%%%%%%%%%%%%%%%%%%%%%%%%%%%%%%%%%%%%%%%%%%%%

\subsubsection{Skorohod spaces}

Let $\left( E, r \right)$ be a general metric space. Denote by $C\left( [0,\infty), E \right)$ the space of continuous $E$-valued paths with the uniform topology. The \emph{Skorohod space} $D\left( [0,\infty), E \right)$ consists of $E$-valued paths that are right-continuous and have left limits. We do not go into details about the topology on $D\left( [0,\infty), E \right)$, but we note that there exists a metric under which $D\left( [0,\infty), E \right)$ is a complete and separable metric space if $\left( E, r \right)$ is complete and separable (Ethier and Kurtz \cite[Chapter 3]{ethier1986markov}). The topology generated by this metric is called the Skorohod topology on $D\left( [0,\infty), E \right)$. Note that $C\left( [0,\infty), E \right)$ is a subset of $D\left( [0,\infty), E \right)$ and it can be seen that the Skorohod topology relativized to $C\left( [0,\infty), E \right)$ coincides with the uniform topology there. For more on the Skorohod space we refer the reader to Billingsley \cite[Chapter 3]{billingsley1999convergence} and Ethier and Kurtz \cite[Chapter 3]{ethier1986markov}.

The reason we introduced the Skorohod space in general is because the natural space in which to consider $\mu_{n} \left( \cdot \right)$ is $D\left( [0,\infty), \mathcal{P}_{1} \left( \mathbb{R} \right) \right)$: we know that $\mu_{n} \left( \cdot \right)$ is a piece-wise constant function in $\mathcal{P}_{1} \left( \mathbb{R} \right)$, having jumps whenever a particle jumps. (We can define $\mu_{n} \left( \cdot \right)$ to be right-continuous at its discontinuities.) Since $\mathcal{P}_{1} \left( \mathbb{R} \right)$ is complete and separable under the Wasserstein metric, $D\left( [0,\infty), \mathcal{P}_{1} \left( \mathbb{R} \right) \right)$ is also complete and separable under an appropriate metric. Similarly, for any test function $f$ (for instance $f \in H$), the natural space in which to consider $\left\langle f, \mu_{n} \left( \cdot \right) \right\rangle$ is the Skorohod space $D\left( [0,\infty), \mathbb{R} \right)$.

%%%%%%%%%%%%%%%%%%%%%%%%%%%%%%%%%%%%%%%%%%%%%%%%%%%%%%%%%%%%%%%%%%%%%%%%%
%%%%%%%%%%%%%%%%%%%%%%%%%%%%%%%%%%%%%%%%%%%%%%%%%%%%%%%%%%%%%%%%%%%%%%%%%

\subsubsection{Tightness in Skorohod spaces}

Recall our plan of action from the end of Section \ref{sec:results}:
\begin{itemize}
 \item Prove tightness of $\left\{ \mu_{n}\left( \cdot \right) \right\}_{n\geq 1}$ in the Skorohod space $D\left( [0,\infty), \mathcal{P}_{1}\left( \mathbb{R} \right) \right)$.
 \item Identify the limit of $\left\{ \mu_{n}\left( \cdot \right) \right\}_{n\geq 1}$: here we essentially have to show that the limit solves the deterministic mean field equation.
 \item Prove uniqueness of the solution to the mean field equation with a given initial condition.
\end{itemize}
In view of the first point, our goal here is to find checkable sufficient conditions for tightness in Skorohod spaces. Since all metric spaces we deal with are complete and separable, by Prohorov's theorem the concepts of tightness and relative compactness are equivalent in our setting, and therefore we use these two notions interchangeably. Following Perkins \cite{perkins1999dawson}, we say that a sequence is $C$-relatively compact if it is relatively compact and all weak limit points are a.s.\ continuous.

%%%%%%%%%%%%%%%%%%%%%%%%%%%%%%%%%%%%%%%%%%%%%%%%%%%%%%%%%%%%%%%%%%%%%%%%%
%%%%%%%%%%%%%%%%%%%%%%%%%%%%%%%%%%%%%%%%%%%%%%%%%%%%%%%%%%%%%%%%%%%%%%%%%

\bigskip\noindent\textbf{Tightness in $D\left( [0, \infty), \mathbb{R} \right)$}

\bigskip\noindent First, we state conditions for tightness in $D\left( [0, \infty), \mathbb{R} \right)$. For $y\left( \cdot \right) \in  D\left( [0, \infty), \mathbb{R} \right)$ define
\begin{equation}\label{eq:biggest_jump_until_T}
J_{T} \left( y \left( \cdot \right) \right) = \sup\limits_{0 \leq t \leq T} \left| y\left( t \right) - y\left( t- \right) \right|,
\end{equation}
the largest (absolute) jump in $y\left( \cdot \right)$ until $T$. Let $\left\{ Y_{n} \left( \cdot \right) \right\}_{n\geq 1}$ be a sequence of random variables in $D\left( [0, \infty), \mathbb{R} \right)$. The following theorem, which is a combination of Billingsley \cite[Theorem 16.8.]{billingsley1999convergence} and the corollary thereafter, characterizes tightness in $D\left( [0, \infty), \mathbb{R} \right)$.
\begin{theorem}\label{thm:tightness_thm_billingsley}
The sequence $\left\{ Y_{n} \left( \cdot \right) \right\}_{n\geq 1}$ is tight in $D\left( [0, \infty), \mathbb{R} \right)$ if and only if the following three conditions hold:
\begin{enumerate}[(i)]
\item 
\begin{equation}\label{eq:B1}
\lim\limits_{L \to \infty} \limsup\limits_{n \to \infty} \p \left( \left| Y_{n}\left( 0 \right) \right| \geq L \right) = 0.
\end{equation}
\item For each $T \geq 0$,
\begin{equation}\label{eq:B2}
\lim\limits_{L \to \infty} \limsup\limits_{n \to \infty} \p \left( J_{T} \left( Y_{n} \left( \cdot \right) \right) \geq L \right) = 0.
\end{equation}
\item For each $T \geq 0$ and $\eps > 0$,
\begin{equation}\label{eq:B3}
\lim\limits_{\delta \to 0} \limsup\limits_{n \to \infty} \p \left( \inf\limits_{\left\{t_{i}\right\}} \max\limits_{1 \leq i \leq v} \sup\limits_{s,t \in [t_{i-1},t_{i})} \left| Y_{n} \left( t \right) - Y_{n} \left( s \right) \right| \geq \eps \right) = 0,
\end{equation}
where the infimum extends over all decompositions $[t_{i-1},t_{i}), 1\leq i \leq v$, of $[0,T)$ such that $t_{i} - t_{i-1} > \delta$ for $1 \leq i < v$.
\end{enumerate}
\end{theorem}
We now look at a condition which implies that any limit point is continuous, i.e.\ is in $C\left( [0,\infty), \mathbb{R} \right)$. Following Ethier and Kurtz \cite[p. 147]{ethier1986markov}, for $y\left( \cdot \right) \in  D\left( [0, \infty), \mathbb{R} \right)$ define
\begin{equation}\label{eq:jump_def}
J\left( y\left( \cdot \right) \right) = \int\limits_{0}^{\infty} e^{-u} \min\left\{ J_{u}\left( y\left( \cdot \right) \right), 1 \right\} du,
\end{equation}
where $J_{u}\left( y\left( \cdot \right) \right)$ is defined as in \eqref{eq:biggest_jump_until_T}. The following theorem characterizes continuity of a weak limit point (Ethier and Kurtz \cite[Theorem 3.10.2.]{ethier1986markov}).
\begin{theorem}\label{thm:weak_limit_cont}
Suppose $Y_{n} \left( \cdot \right) \Rightarrow_{n} Y \left( \cdot \right)$ in $D\left( [0, \infty), \mathbb{R} \right)$. Then $Y\left( \cdot \right)$ is almost surely continuous if and only if $J\left( Y_{n} \left( \cdot \right) \right) \Rightarrow_{n} 0$.
\end{theorem}
We now give sufficient conditions for a sequence to be $C$-relatively compact.
\begin{theorem}\label{thm:GK_tightness}
Suppose the following two conditions hold.
\begin{enumerate}[(i)]
\item 
\begin{equation}\label{eq:GK1}
\lim\limits_{L \to \infty} \limsup\limits_{n \to \infty} \p \left( \left| Y_{n}\left( 0 \right) \right| \geq L \right) = 0.
\end{equation}
\item For each $T \geq 0$ and $\eps > 0$,
\begin{equation}\label{eq:GK2}
\lim\limits_{\delta \to 0} \limsup\limits_{n \to \infty} \p \left( \sup\limits_{\substack{0 < t - s < \delta \\ 0 \leq s \leq t \leq T}} \left| Y_{n} \left( t \right) - Y_{n} \left( s \right) \right| \geq \eps \right) = 0.
\end{equation}
\end{enumerate}
Then $\left\{ Y_{n} \left( \cdot \right) \right\}_{n\geq 1}$ is $C$-relatively compact in $D\left( [0, \infty), \mathbb{R} \right)$.
\end{theorem}
This theorem is stated in Grigorescu and Kang \cite[Section 3]{grigorescu2008steady}, but without a proof, and so for completeness we provide a short proof.
\begin{proof}
Let us check tightness first: according to Theorem \ref{thm:tightness_thm_billingsley} we need to check that \eqref{eq:B1}, \eqref{eq:B2}, and \eqref{eq:B3} hold. Our first condition is exactly \eqref{eq:B1}. From \eqref{eq:GK2} it follows that for every $T > 0$, $J_{T} \left( Y_{n} \left( \cdot \right) \right) \Rightarrow_{n} 0$, and consequently \eqref{eq:B2} holds. To show that \eqref{eq:B3} follows from \eqref{eq:GK2}, note that
\begin{equation*}
\inf\limits_{\left\{t_{i}\right\}} \max\limits_{1 \leq i \leq v} \sup\limits_{s,t \in [t_{i-1},t_{i})} \left| Y_{n} \left( t \right) - Y_{n} \left( s \right) \right| \leq \sup\limits_{\substack{0 < t - s < 2 \delta \\ 0 \leq s \leq t \leq T}} \left| Y_{n} \left( t \right) - Y_{n} \left( s \right) \right|,
\end{equation*}
where the infimum extends over the same range as in Theorem \ref{thm:tightness_thm_billingsley}.

Now let us check that any weak limit point is a.s.\ continuous. According to Theorem \ref{thm:weak_limit_cont}, we need to check that $J\left( Y_{n} \left( \cdot \right) \right) \Rightarrow_{n} 0$. From \eqref{eq:GK2} we know that for every $T > 0$, $J_{T} \left( Y_{n} \left( \cdot \right) \right) \Rightarrow_{n} 0$. For any fixed $\eps > 0$ there exists a $T$ such that $e^{-T} < \eps / 2$. From \eqref{eq:jump_def} it follows that $J\left( Y_{n} \left( \cdot \right) \right) \leq T J_{T} \left( Y_{n} \left( \cdot \right) \right) + e^{-T}$, and consequently:
\begin{equation*}
\lim\limits_{n \to \infty} \p \left( J\left( Y_{n} \left( \cdot \right) \right) > \eps \right) \leq \lim\limits_{n \to \infty} \p \left( J_{T}\left( Y_{n} \left( \cdot \right) \right) > \eps / 2T \right) = 0. \qedhere
\end{equation*}
\end{proof}

%%%%%%%%%%%%%%%%%%%%%%%%%%%%%%%%%%%%%%%%%%%%%%%%%%%%%%%%%%%%%%%%%%%%%%%%%
%%%%%%%%%%%%%%%%%%%%%%%%%%%%%%%%%%%%%%%%%%%%%%%%%%%%%%%%%%%%%%%%%%%%%%%%%

\noindent\textbf{Tightness in Skorohod spaces containing measure-valued paths}

\bigskip\noindent Now let us turn to conditions that imply tightness in the Skorohod space of time-indexed measure-valued paths $D\left( [0,\infty), \mathcal{P}_{1} \left( \mathbb{R} \right) \right)$. Our main tool is a slight modification of a theorem of Perkins \cite[Theorem II.4.1]{perkins1999dawson}. Perkins' theorem basically states that if we can check an extra condition, then $C$-relative compactness in $D\left( [0,\infty), M_{F} \left( E \right) \right)$ (where $M_F \left( E \right)$ is the space of finite measures on the metric space $E$, with the topology of weak convergence) can be checked by checking the $C$-relative compactness of appropriate integrals in $D\left( [0,\infty), \mathbb{R} \right)$.

Our modification of Perkins' theorem is the following.
\begin{theorem}\label{thm:perkins_mod}
A sequence of cadlag $\mathcal{P}_1 \left( \mathbb{R} \right)$-valued processes $\left\{ P_{n} \left( \cdot \right) \right\}_{n\geq 1}$ is $C$-relatively compact in\linebreak[4] $D\left( [0, \infty), \mathcal{P}_1 \left( \mathbb{R} \right) \right)$ if the following conditions hold:
\begin{enumerate}[(i)]
\item For every $T > 0$ there exists a $c = c \left( T \right)$ and a $\tau = \tau \left( T \right) \in \left(0,1\right)$ such that for every $\eps > 0$ there exists a $K = K_{T,\eps} \in \mathbb{R}$ such that $K \leq c / \eps^{1 - \tau}$ and
\begin{equation*}
\sup\limits_{n} \p \left( \sup\limits_{0\leq t \leq T} P_{n} \left( t, (-\infty, -K) \cup (K, \infty) \right) > \eps \right) < \eps.
\end{equation*}
\item For every $f \in C_{b}$, $\left\{ \left\langle f, P_{n} \left( \cdot \right) \right\rangle \right\}_{n \geq 1}$ is $C$-relatively compact in $D\left( [0,\infty), \mathbb{R} \right)$.
\end{enumerate}
\end{theorem}

If one looks at the proof of Perkins' original theorem \cite[Theorem II.4.1]{perkins1999dawson}, then one can see that it holds true if one replaces $M_F \left( E \right)$ with $M_1 \left( E \right)$, the space of probability measures on $E$ with the topology of weak convergence. However, our measure-valued process lives in a different space, namely $\mu_n \left( \cdot \right) \in D\left( [0,\infty), \mathcal{P}_1 \left( \mathbb{R} \right) \right)$, where the Wasserstein space $\mathcal{P}_1 \left( \mathbb{R} \right)$ is endowed with the metric $d_1$. It turns out that we cannot simply replace $M_F \left( E \right)$ with $\mathcal{P}_1 \left( \mathbb{R} \right)$ in the theorem. Convergence in $\left( \mathcal{P}_{1} \left( \mathbb{R} \right), d_1 \right)$ is stronger than weak convergence: it is weak convergence and the convergence of the first absolute moment of the measure. It is this that causes the original proof of the theorem to break down.

The proof of Theorem \ref{thm:perkins_mod} is essentially the same as that of Perkins' theorem, but with a slight modification at one point, which requires that our condition (i) is stronger than in Perkins' original version. For completeness, we reproduce the whole proof, following the lines of Perkins \cite[Section II.4.]{perkins1999dawson}: this can be found in Section \ref{ch:Perkins}.

We use Theorem \ref{thm:perkins_mod} in order to prove tightness of the empirical measure process $\left\{ \mu_{n} \left( \cdot \right) \right\}_{n\geq 1}$ in Corollary \ref{cor:tightness_P1(R)}.

%%%%%%%%%%%%%%%%%%%%%%%%%%%%%%%%%%%%%%%%%%%%%%%%%%%%%%%%%%%%%%%%%%%%%%%%%
%%%%%%%%%%%%%%%%%%%%%%%%%%%%%%%%%%%%%%%%%%%%%%%%%%%%%%%%%%%%%%%%%%%%%%%%%

\subsubsection{A few more facts that are used}\label{ch:More_facts}

Before we move on to our results, let us finish the preliminaries with a few more things that we use later in the proofs (in particular in Section \ref{ch:LimitSolvesMF}). We rely on the continuous mapping theorem (Billingsley \cite[Section 2]{billingsley1999convergence}). We do not need the theorem in full generality, only the following:
\begin{theorem}
Suppose $h$ is a continuous function from the metric space $\left( S, d \right)$ to the metric space $\left( S', d' \right)$. If $X_{n} \Rightarrow_{n} X$ in $\left( S, d \right)$, then $h \left( X_{n} \right) \Rightarrow_{n} h\left( X \right)$ in $\left( S', d' \right)$.
\end{theorem}

Finally, we need the following two lemmas, which can be found as exercises in Ethier and Kurtz \cite[Chapter 3]{ethier1986markov}.
\begin{lemma}\label{lem:EKex1}
Let $E$ and $F$ be metric spaces, and let $f:\ E \to F$ be continuous. Then the mapping $x \to f \circ x$ from $D \left( [0, \infty), E \right)$ to $D \left( [0, \infty), F \right)$ is continuous.
\end{lemma}
\begin{lemma}\label{lem:EKex2}
The function
\begin{equation*}
f\left( x \right) \left(\cdot \right) = \int\limits_{0}^{\cdot} x\left( s \right) ds
\end{equation*}
from $D \left( [0, \infty), \mathbb{R} \right)$ to $D \left( [0, \infty), \mathbb{R} \right)$ is continuous.
\end{lemma}

%%%%%%%%%%%%%%%%%%%%%%%%%%%%%%%%%%%%%%%%%%%%%%%%%%%%%%%%%%%%%%%%%%%%%%%%%
%%%%%%%%%%%%%%%%%%%%%%%%%%%%%%%%%%%%%%%%%%%%%%%%%%%%%%%%%%%%%%%%%%%%%%%%%

\subsection{Main steps of the fluid limit procedure}\label{ch:Results}

In this subsection we collect the intermediate steps that build up the proof of the main theorem, Theorem \ref{thm:main}, leaving the proofs of these results to later subsections. We complete the proof of Theorem \ref{thm:main} at the end of the subsection. Recall our assumptions on the jump rate function $w$ and the jump length $Z$ which we assume throughout this section.

Our first results deal with showing the tightness of $\left\{ \mu_{n} \left( \cdot \right) \right\}_{n \geq 1}$ in the Skorohod space $D\left( [0,\infty), \mathcal{P}_{1}\left( \mathbb{R} \right) \right)$. 
\begin{theorem}\label{thm:tightness_R}
Suppose Assumption 1 from Section \ref{sec:results} holds. Then for any $f \in C_b$ the sequence $\left\{ \left\langle f, \mu_{n}\left( \cdot \right) \right\rangle \right\}_{n \geq 1}$ is $C$-relatively compact in $D\left( [0,\infty), \mathbb{R} \right)$.
\end{theorem}
We remark that although we only need the above result for continuous and bounded functions in the corollary below, continuity is not used in our proof, and with a slight modification in the proof the same can be proven for Lipschitz continuous functions.
\begin{corollary}\label{cor:tightness_P1(R)}
Suppose Assumptions 1 and 2 from Section \ref{sec:results} hold. Then $\left\{ \mu_{n} \left( \cdot \right) \right\}_{n \geq 1}$ is $C$-relatively compact in $D\left( [0,\infty), \mathcal{P}_{1}\left( \mathbb{R} \right) \right)$.
\end{corollary}
Theorem \ref{thm:tightness_R} and Corollary \ref{cor:tightness_P1(R)} are proved in Section \ref{ch:Tightness}.

We remark that tightness can be proved in other ways as well, using the techniques of Ethier and Kurtz \cite[Chapter 3]{ethier1986markov}. In particular, Thomas G. Kurtz communicated to us a different way of checking tightness, and we have included a sketch of his proof at the end of Section \ref{ch:Tightness}.

Our next group of results is concerned with identifying the limit of the sequence $\left\{ \mu_{n} \left( \cdot \right) \right\}_{n \geq 1}$. In particular, our goal is to show that any weak limit point of $\left\{ \mu_{n} \left( \cdot \right) \right\}_{n \geq 1}$ solves the mean field equation. We prove this in two steps. Recall the definition \eqref{eq:H_def} of $H$.
\begin{theorem}\label{thm:error_vanishes}
For every $t \geq 0$ and every $f \in H$,
\begin{equation}\label{eq:conv_to_0_inprob}
\sup\limits_{0\leq s \leq t} \left| A_{s,f} \left( \mu_{n} \left( \cdot \right) \right) \right| \xrightarrow[n\to \infty]{\mathbb{P}} 0
\end{equation}
where $\xrightarrow[n\to \infty]{\mathbb{P}} 0$ denotes convergence in probability.
\end{theorem}
\begin{theorem}\label{thm:weak_conv_of_eq}
If $\mu_{n} \left( \cdot \right) \Rightarrow_{n} \mu\left( \cdot \right)$ in $D\left( [0,\infty), \mathcal{P}_{1}\left( \mathbb{R} \right) \right)$, then for every $t\geq 0$ and every $f \in H$,
\begin{equation}\label{eq:conv_to_mf}
A_{t,f} \left( \mu_{n} \left( \cdot \right) \right) \Rightarrow_{n} A_{t,f} \left( \mu\left( \cdot \right) \right)
\end{equation}
in $\mathbb{R}$.
\end{theorem}
An immediate corollary of these two theorems is the following.
\begin{corollary}\label{cor:limit_solves_mfe}
Any weak limit point $\mu\left( \cdot \right)$ of $\left\{ \mu_{n} \left( \cdot \right) \right\}_{n \geq 1}$ in the Skorohod space $D\left( [0,\infty), \mathcal{P}_{1}\left( \mathbb{R} \right) \right)$ solves the mean field equation almost surely. That is, if $\mu_{n} \left( \cdot \right) \Rightarrow_{n} \mu\left( \cdot \right)$ in $D\left( [0,\infty), \mathcal{P}_{1}\left( \mathbb{R} \right) \right)$, then for every $t\geq 0$ and every $f \in H$, $A_{t,f} \left( \mu\left( \cdot \right) \right) = 0$ almost surely.
\end{corollary}
The proofs of Theorem \ref{thm:error_vanishes}, Theorem \ref{thm:weak_conv_of_eq}, and Corollary \ref{cor:limit_solves_mfe} can be found in Section \ref{ch:LimitSolvesMF}.

Our final result concerns the mean field equation: the solution with a given initial condition is unique.
\begin{theorem}\label{thm:cont_dep_init_cond}
Suppose $\mu^{1} \left( \cdot \right)$ and $\mu^{2} \left( \cdot \right)$ are solutions to the mean field equation with initial conditions $\mu_{0}^{1}$ and $\mu_{0}^{2}$ respectively. Denote by $d_H \left( t \right)$ the distance $d_H \left( \mu^{1} \left( t \right), \mu^{2} \left( t \right) \right)$ of the two measures at time $t$ according to the metric defined by \eqref{eq:prob_metric} with $\mathcal{H} = H$. Then there exists a constant $c$ such that
\begin{equation}\label{eq:cont_dep_init_cond}
d_H\left( t \right) \leq d_H\left( 0 \right) e^{ct}.
\end{equation}
\end{theorem}
\begin{corollary}\label{cor:uniqueness}
Suppose $\mu^{1} \left( \cdot \right)$ and $\mu^{2} \left( \cdot \right)$ are solutions to the mean field equation with the same initial condition. Then $\mu^{1} \left( \cdot \right) = \mu^{2} \left( \cdot \right)$, i.e.\ the solution to the mean field equation with a given initial condition is unique.
\end{corollary}
Let us now prove our main result.
\begin{proof}[Proof of Theorem \ref{thm:main}]
The sequence $\left\{ \mu_{n} \left( \cdot \right) \right\}_{n \geq 1}$ is tight in $D\left( [0,\infty), \mathcal{P}_{1}\left( \mathbb{R} \right) \right)$ due to Corollary \ref{cor:tightness_P1(R)}. Corollary \ref{cor:limit_solves_mfe} shows that any limit point satisfies the mean field equation. Due to Assumption 3, any limit point satisfies the mean field equation with deterministic initial condition $\nu$. We know that the solution to the mean field equation with a given initial condition is unique (Corollary \ref{cor:uniqueness}). Consequently, due to Prohorov's theorem, $\mu_{n} \left( \cdot \right) \Rightarrow_{n} \mu \left( \cdot \right)$, since we have proved tightness and the uniqueness of any limit point.
\end{proof}

%%%%%%%%%%%%%%%%%%%%%%%%%%%%%%%%%%%%%%%%%%%%%%%%%%%%%%%%%%%%%%%%%%%%%%%%%
%%%%%%%%%%%%%%%%%%%%%%%%%%%%%%%%%%%%%%%%%%%%%%%%%%%%%%%%%%%%%%%%%%%%%%%%%

\subsection{Tightness}\label{ch:Tightness}

In this subsection we prove Theorem \ref{thm:tightness_R} and Corollary \ref{cor:tightness_P1(R)}.

Before starting the proof, let us introduce a particle system $\left\{ \tilde{x}_{i}\left( \cdot \right) \right\}_{i=1}^{n}$, related to $\left\{ x_{i}\left( \cdot \right) \right\}_{i=1}^{n}$, which we use in the proof.  Initially, the two particle systems have the same configurations: $\tilde{x}_{i} \left( 0 \right) = x_{i} \left( 0 \right)$ for each $i$. Let $\left\{ \tilde{x}_{i}\left( \cdot \right) \right\}_{i=1}^{n}$ be a particle system in which all the particles are independent from each other, all of the particles jump with rate $a$ (recall that $a = \lim_{x \to -\infty} w\left( x \right) = \sup_x w\left( x \right)$), and whenever they jump, the jump length is a random variable from the same jump length distribution of the original system, independently of everything else. In other words, let $S_{i}\left( \cdot \right)$ be independent Poisson processes with rate $a$ for each $i$ , let $\xi_{1}^{i}, \xi_{2}^{i}, \dots$ be i.i.d.\ random variables from the jump length distribution, and define $\tilde{x}_{i} \left( t \right) = x_{i} \left( 0 \right) + \sum_{k=1}^{S_{i}\left( t \right)} \xi_{k}^{i}$. Moreover, $\left\{ x_{i}\left( \cdot \right) \right\}_{i=1}^{n}$ and $\left\{ \tilde{x}_{i}\left( \cdot \right) \right\}_{i=1}^{n}$ can be constructed jointly on the same probability space in the following way. Define $\left\{ \tilde{x}_{i}\left( \cdot \right) \right\}_{i=1}^{n}$ exactly as above, and whenever a particle $\tilde{x}_{i}$ jumps, let $x_{i}$ jump the same jump length with probability $w\left( x_{i} \left( t \right) - m_{n} \left( t \right) \right) / a$; with the remaining probability let $x_{i}$ stay in place. This joint construction leads to two advantageous monotonicity properties: first, \begin{equation*}
\tilde{x}_{i} \left( t \right) \geq x_{i} \left( t \right)
\end{equation*}
for any $i$ and $t\geq 0$, and second,
\begin{equation}\label{eq:xtilde2}
\tilde{x}_{i} \left( t \right) - \tilde{x}_{i} \left( s \right) \geq x_{i} \left( t \right) - x_{i} \left( s \right)
\end{equation}
for any $i$ and $t \geq s \geq 0$.

\begin{proof}[Proof of Theorem \ref{thm:tightness_R}]
Let $f \in C_b$, and denote by $K$ the upper bound of $\left|f\right|$, i.e.\ $\left| f\left( x \right) \right| \leq K$ for all $x$. Define $I_{n} \left( \cdot \right) := \left\langle f, \mu_{n} \left( \cdot \right) \right\rangle = \frac{1}{n} \sum_{i=1}^{n} f\left( x_{i} \left( \cdot \right) \right)$. Clearly $\left| I_{n} \left( t \right) \right| \leq K$ for any $t$. In order to prove tightness in $D\left( [0, \infty), \mathbb{R} \right)$ we need to show that \eqref{eq:GK1} and \eqref{eq:GK2} hold with $Y_n$ replaced with $I_n$. Since $f$ is bounded it is immediate that \eqref{eq:GK1} holds. Now let us show that \eqref{eq:GK2} holds. Our basic observation is that
\[
\left| I_{n} \left( t \right) - I_{n} \left( s \right) \right| = \frac{1}{n} \left| \sum\limits_{i=1}^{n} \left( f\left( x_{i} \left( t \right) \right) - f\left( x_{i} \left( s \right) \right) \right) \right| \leq \frac{1}{n} \sum\limits_{i=1}^{n} \left| f\left( x_{i} \left( t \right) \right) - f\left( x_{i} \left( s \right) \right) \right| \leq \frac{2K}{n} \sum\limits_{i=1}^{n} \mathbf{1}\left[ x_{i} \left( t \right) > x_{i} \left( s \right) \right].
\]
By monotonicity of $x_i \left( \cdot \right)$ and using \eqref{eq:xtilde2} we have
\[
\begin{aligned}
\p \left( \sup\limits_{\substack{0 < t - s < \delta \\ 0 \leq s \leq t \leq T}} \left| I_{n} \left( t \right) - I_{n} \left( s \right) \right| > \eps \right) &\leq \p \left( \sup\limits_{\substack{0 < t - s < \delta \\ 0 \leq s \leq t \leq T}} \frac{2K}{n} \sum\limits_{i=1}^{n} \mathbf{1} \left[ x_{i} \left( t \right) > x_{i} \left( s \right) \right] > \eps \right)\\
&\leq \p \left( \sup\limits_{0 \leq s \leq T} \frac{2K}{n} \sum\limits_{i=1}^{n} \mathbf{1} \left[ x_{i} \left( s + 2 \delta \right) > x_{i} \left( s \right) \right] > \eps \right)\\
&\leq \p \left( \sup\limits_{0 \leq s \leq T} \frac{2K}{n} \sum\limits_{i=1}^{n} \mathbf{1} \left[ \tilde{x}_{i} \left( s + 2 \delta \right) > \tilde{x}_{i} \left( s \right) \right] > \eps \right).
\end{aligned}
\]
To abbreviate notation, define $Y_{n,\delta} \left( s \right) := \frac{2K}{n} \sum_{i=1}^{n} \mathbf{1} \left[ \tilde{x}_{i} \left( s + \delta \right) > \tilde{x}_{i} \left( s \right) \right]$. Now if $\sup_{0 \leq s \leq T} Y_{n,2\delta} \left( s \right) > \eps$, then there exists an integer $k$, $0 \leq k \leq \left\lceil T / \delta \right\rceil - 1$, such that $Y_{n,\delta} \left( k \delta \right) \geq \eps/ 4$. Suppose the contrary, that for every $k$ the above expression is less than $\eps / 4$; then since an interval of length $2\delta$ can intersect at most three neighboring intervals of length $\delta$, it follows that $\sup_{0 \leq s \leq T} Y_{n,2\delta} \left( s \right) < 3 \eps / 4$. Therefore
\[
\begin{aligned}
\p \left( \sup\limits_{0 \leq s \leq T} Y_{n,2\delta} \left( s \right) > \eps \right) &\leq \p \left( \exists k, 0 \leq k \leq \left\lceil T / \delta \right\rceil - 1\ : Y_{n,\delta} \left( k \delta \right) \geq \eps/ 4 \right)\\
&= 1 - \p \left( \forall k, 0 \leq k \leq \left\lceil T / \delta \right\rceil - 1\ : Y_{n,\delta} \left( k \delta \right) < \eps/ 4 \right)\\
&= 1 - \left( \p \left( Y_{n,\delta} \left( 0 \right) < \eps/ 4 \right) \right)^{\left\lceil T / \delta \right\rceil}\\
&= 1 - \left( 1 - \p \left( Y_{n,\delta} \left( 0 \right) \geq \epsilon / 4 \right) \right)^{\left\lceil T / \delta \right\rceil},
\end{aligned}
\]
where we used that for every $i$, $\tilde{x}_{i} \left( \cdot \right)$ has independent and stationary increments. We bound the probability in the last line using Markov's inequality
\begin{equation*}
\p \left( \sum\limits_{i=1}^{n}  \mathbf{1} \left[ \tilde{x}_{i} \left( \delta \right) > \tilde{x}_{i} \left( 0 \right) \right] > \frac{n\eps}{8K} \right) \leq \E \left( \left( \sum\limits_{i=1}^{n} \mathbf{1} \left[ \tilde{x}_{i} \left( \delta \right) > \tilde{x}_{i} \left( 0 \right) \right] \right)^{2} \right) / \frac{n^{2} \eps^{2}}{64 K^{2}},
\end{equation*}
where
\[
\begin{aligned}
\E \left( \left( \sum\limits_{i=1}^{n} \mathbf{1} \left[ \tilde{x}_{i} \left( \delta \right) > \tilde{x}_{i} \left( 0 \right) \right] \right)^{2} \right) &= n \p \left( \tilde{x}_{1} \left( \delta \right) > \tilde{x}_{1} \left( 0 \right) \right) + \left( n^{2} - n \right) \p \left( \tilde{x}_{1} \left( \delta \right) > \tilde{x}_{1} \left( 0 \right) \text{ and } \tilde{x}_{2} \left( \delta \right) > \tilde{x}_{2} \left( 0 \right) \right)\\
&= n \left( 1 - e^{-a \delta} \right) + \left( n^{2} - n \right) \left( 1 - e^{-a \delta} \right)^{2}\\
&\approx na \delta + \left( n^{2} - n \right) \left( a \delta \right)^{2},
\end{aligned}
\]
from which \eqref{eq:GK2} follows immediately.
\end{proof}

\begin{proof}[Proof of Corollary \ref{cor:tightness_P1(R)}]
In order to prove this corollary we need to check the two conditions of Theorem \ref{thm:perkins_mod}. The second condition follows immediately from Theorem \ref{thm:tightness_R}. It remains to show the first condition of Theorem \ref{thm:perkins_mod}. Due to monotonicity of the $x_{i} \left( t \right)$ in $t$, we have
\[
\begin{aligned}
\sup_{0\leq t \leq T} \mu_{n} \left( t, (-\infty, -K) \cup (K, \infty) \right) &\leq \sup_{0\leq t \leq T} \mu_{n} \left( t, (-\infty, -K) \right) + \sup_{0\leq t \leq T} \mu_{n} \left( t, (K, \infty) \right)\\
&= \mu_{n} \left( 0, (-\infty, -K) \right) + \mu_{n} \left( T, (K, \infty) \right),
\end{aligned}
\]
and consequently
\begin{multline*}
\sup\limits_{n} \p \left( \sup\limits_{0\leq t \leq T} \mu_{n} \left( t, (-\infty, -K) \cup (K, \infty) \right) > \eps \right)\\
\begin{aligned}
&\leq \sup\limits_{n} \p \left( \mu_{n} \left( 0, (-\infty, -K) \right) + \mu_{n} \left( T, (K, \infty) \right) > \eps \right)\\
&\leq \sup\limits_{n} \p \left( \mu_{n} \left( 0, (-\infty, -K) \right) > \eps / 2 \right) + \sup\limits_{n} \p \left( \mu_{n} \left( T, (K, \infty) \right) > \eps / 2 \right).
\end{aligned}
\end{multline*}
Let us look at the second term:
\begin{align*}
\mu_{n} \left( T, (K, \infty) \right) &= \frac{1}{n} \sum \limits_{i=1}^{n} \mathbf{1} \left[ x_{i} \left( T \right) > K \right]\\
&\leq \frac{1}{n} \sum \limits_{i=1}^{n} \mathbf{1} \left[ x_{i} \left( 0 \right) > K / 2 \right] +  \frac{1}{n} \sum \limits_{i=1}^{n} \mathbf{1} \left[ x_{i} \left( T \right) - x_{i} \left( 0 \right) > K / 2 \right]\\
&\leq \frac{1}{n} \sum \limits_{i=1}^{n} \mathbf{1} \left[ x_{i} \left( 0 \right) > K / 2 \right] +  \frac{1}{n} \sum \limits_{i=1}^{n} \mathbf{1} \left[ \tilde{x}_{i} \left( T \right) - \tilde{x}_{i} \left( 0 \right) > K / 2 \right],
\end{align*}
and so consequently
\[
\begin{aligned}
\sup\limits_{n} \p \left( \mu_{n} \left( T, (K, \infty) \right) > \eps / 2 \right) &\leq \sup\limits_{n} \p \left( \mu_{n} \left( 0, \left( K / 2, \infty \right) \right) > \eps / 4 \right)\\
&\quad + \sup\limits_{n} \p \left( \frac{1}{n} \sum \limits_{i=1}^{n} \mathbf{1} \left[ \tilde{x}_{i} \left( T \right) - \tilde{x}_{i} \left( 0 \right) > K / 2 \right] > \eps / 4 \right).
\end{aligned}
\]
Let us look at the second term on the right hand side. Using Markov's inequality and the fact that the $\tilde{x}_i$'s are identically distributed, we have
\[
\p \left( \frac{1}{n} \sum \limits_{i=1}^{n} \mathbf{1} \left[ \tilde{x}_{i} \left( T \right) - \tilde{x}_{i} \left( 0 \right) > K / 2 \right] > \eps / 4 \right) \leq \frac{4}{\eps} \p \left( \tilde{x}_1 \left( T \right) - \tilde{x}_1 \left( 0 \right) > K/2 \right).
\]
Now using Markov's inequality again, we get that 
\[
\frac{4}{\eps} \p \left( \tilde{x}_1 \left( T \right) - \tilde{x}_1 \left( 0 \right) > K/2 \right) = \frac{4}{\eps} \p \left( \left( \tilde{x}_1 \left( T \right) - \tilde{x}_1 \left( 0 \right) \right)^{3} > K^{3} / 8 \right) \leq \frac{32}{K^3 \eps} \E \left( \left( \tilde{x}_1 \left( T \right) - \tilde{x}_1 \left( 0 \right) \right)^{3} \right) =: \frac{\tilde{c} \left( T \right) }{K^3 \eps}.
\]
Now putting everything together we get
\[
\begin{aligned}
\sup\limits_{n} \p \left( \sup\limits_{0\leq t \leq T} \mu_{n} \left( t, (-\infty, -K) \cup (K, \infty) \right) > \eps \right) &\leq \sup\limits_{n} \p \left( \mu_{n} \left( 0, (-\infty, -K) \right) > \eps / 2 \right)\\
&\quad+ \sup\limits_{n} \p \left( \mu_{n} \left( 0, \left( K / 2, \infty \right) \right) > \eps / 4 \right) + \frac{\tilde{c} \left( T \right)}{K^3 \eps},
\end{aligned}
\]
and we need the left hand side to be at most $\eps$, for which it is enough that the right hand side is at most $\eps$.

So let $T > 0$ and $\eps > 0$ be fixed, and let us choose $c = c\left( T \right)$, $\tau = \tau \left( T \right)$, and $K = K_{T, \eps}$ as follows:
\begin{align*}
c\left( T \right) &:= \max \left\{ 3 \tilde{c} \left( T \right), 8 \hat{c}^3 \right\}\\
\tau \left( T \right) &:= 1/3\\
K_{T,\eps} &:= c\left( T \right)^{1/3} \eps^{-2/3},
\end{align*}
where $\hat{c}$ is given by Assumption 2. Then $K^{3}  \geq 3 \tilde{c}\left( T \right) \eps^{-2}$, and so $ \tilde{c} \left( T \right) / \left( K^{3} \eps \right) \leq \eps / 3$. On the other hand $K \geq 2 \hat{c} \eps^{-2/3}$, and so using Assumption 2 we have
\begin{align*}
\sup\limits_{n} \p \left( \mu_{n} \left( 0, (-\infty, -K) \right) > \eps / 2 \right) &\leq \eps / 3\\
\sup\limits_{n} \p \left( \mu_{n} \left( 0, \left( K / 2, \infty \right) \right) > \eps / 4 \right) &\leq \eps / 3.
\end{align*}
Putting all these together concludes the proof.
\end{proof}

As mentioned before, there are other ways to check tightness, and here we sketch Thomas G. Kurtz's proof (see also Kotelenez and Kurtz \cite[Appendix]{kotelenez_kurtz2010macrolim}). First note that
\[
d_1 \left( \mu, \nu \right) \leq \E \left| X - Y \right|,
\]
where $X$ and $Y$ are any random variables with distributions $\mu$ and $\nu$ respectively. The paths of our particles $\left\{ x_i \left( \cdot \right) \right\}_{i=1}^{n}$ are exchangeable, and so for any $\left\{ \mathcal{F}_{t}^{\mu_n \left( \cdot \right)} \right\}$-stopping time $\tau$ we have
\begin{equation}\label{eq:Kurtz1}
\E \left( d_1 \left( \mu_n \left( \tau \right), \mu_n \left( \tau + u \right) \right) \right) \leq \E \left( x_1 \left( \tau + u \right) - x_1 \left( \tau \right) \right).
\end{equation}
For each $T > 0$, let $S_n \left( T \right)$ be the collection of $\left\{ \mathcal{F}_{t}^{\mu_n \left( \cdot \right)} \right\}$-stopping times $\tau$ satisfying $\tau \leq T$. Then by \eqref{eq:Kurtz1} and Aldous's criterion (see e.g.\ Ethier and Kurtz \cite[Theorem 8.6]{ethier1986markov}), $\left\{ \mu_n \left( \cdot \right) \right\}_{n \geq 1}$ is tight in $D \left( [0, \infty), \mathcal{P}_1 \left( \mathbb{R} \right) \right)$ if for each $T > 0$,
\[
\lim_{u \to 0} \limsup_{n \to \infty} \sup_{\tau \in S_{n} \left( T \right)} \E \left( x_1 \left( \tau + u \right) - x_1 \left( \tau \right) \right) = 0,
\]
and also condition (a) of \cite[Theorem 7.2]{ethier1986markov} holds. In our case, for a bounded jump rate function $w$, these conditions can be checked just like before, in the previous way of showing tightness.

%%%%%%%%%%%%%%%%%%%%%%%%%%%%%%%%%%%%%%%%%%%%%%%%%%%%%%%%%%%%%%%%%%%%%%%%%
%%%%%%%%%%%%%%%%%%%%%%%%%%%%%%%%%%%%%%%%%%%%%%%%%%%%%%%%%%%%%%%%%%%%%%%%%

\subsection{Identifying the limit}\label{ch:LimitSolvesMF}

In this section we prove Theorem \ref{thm:error_vanishes}, Theorem \ref{thm:weak_conv_of_eq}, and Corollary \ref{cor:limit_solves_mfe}.

\begin{proof}[Proof of Theorem \ref{thm:error_vanishes}]
Let $t > 0$ and $f \in H$ be fixed from now on. As in the previous section, define $I_{n} \left( \cdot \right) := \left\langle f, \mu_{n} \left( \cdot \right) \right\rangle = \frac{1}{n} \sum_{i=1}^{n} f\left( x_{i} \left( \cdot \right) \right)$. Calculation shows that
\begin{equation*}
L I_{n} \left( t \right) = \left\langle \left( \E \left( f \left( x + Z \right) \right) - f\left( x \right) \right) w\left( x - m_{n}\left( t \right) \right), \mu_{n} \left( t \right) \right\rangle,
\end{equation*}
which means that $M_{n} \left( t\right):= A_{t,f} \left( \mu_{n} \left( \cdot \right) \right)$ is a martingale in $t$ with $M_{n} \left( 0 \right) = 0$. Therefore by Doob's inequality (see e.g.\ Jacod and Shiryaev \cite[p. 11]{jacod1987limit}) we have that for any $c > 0$,
\begin{equation*}
\p\left( \sup\limits_{0 \leq s \leq t} \left| M_{n} \left( s \right) \right| \geq c \right) \leq \frac{\E\left( \sup\limits_{0 \leq s \leq t}  M_{n}^{2} \left( s\right) \right)}{c^{2}} \leq \frac{4 \E \left( M_{n}^{2} \left( t\right) \right)}{c^{2}}.
\end{equation*}
So in order to show \eqref{eq:conv_to_0_inprob} we need to show that
\begin{equation}\label{eq:E_Mn2_to_0}
\lim\limits_{n \to \infty} \E \left( M_{n}^{2} \left( t \right) \right) = 0.
\end{equation}
We know that 
\begin{equation*}
N_{n} \left( t \right) := M_{n}^{2} \left( t \right) - \int\limits_{0}^{t} L M_{n}^{2} \left( s \right) ds
\end{equation*}
is a martingale with $N_{n} \left( 0 \right) = 0$, and therefore
\begin{equation}
\E\left( M_{n}^{2} \left( t \right) \right) = \E \left( \int\limits_{0}^{t} L M_{n}^{2} \left( s \right) ds \right) = \int\limits_{0}^{t} \E \left( L M_{n}^{2} \left( s \right) \right) ds.\label{eq:fubini}
\end{equation}
The order of the expectation and the integral can be switched due to the Fubini-Tonelli theorem, because $L M_{n}^{2} \left( s \right) \geq 0$ as we will see shortly. Let us calculate $L M_{n}^{2} \left( t \right)$:
\begin{align*}
L M_{n}^{2} \left( t \right) &= \lim_{dt \to 0} \frac{1}{dt} \E \left( M_{n}^{2} \left( t + dt \right) - M_{n}^{2} \left( t \right) \middle| \mathcal{F}_{t} \right)\\
&= \lim_{dt \to 0} \frac{1}{dt} \E \left( \left( M_{n} \left( t + dt \right) - M_{n} \left( t \right) \right)^{2} \middle| \mathcal{F}_{t} \right),
\end{align*}
where we used that $M_{n} \left( t \right)$ is a martingale. Since $L I_{n} \left( s \right)$ is bounded, we have
\begin{align*}
M_{n}\left( t + dt \right) - M_{n} \left( t \right) &= \left( I_{n} \left( t + dt \right) - I_{n} \left( t \right) \right) - \int\limits_{t}^{t + dt} L I_{n} \left( s \right) ds\\
&= \left( I_{n} \left( t + dt \right) - I_{n} \left( t \right) \right) + O\left( dt \right)
\end{align*}
and so consequently
\begin{align*}
L M_{n}^{2} \left( t \right) &= \lim_{dt \to 0} \frac{1}{dt} \E \left( \left( M_{n} \left( t + dt \right) - M_{n} \left( t \right) \right)^{2} \middle| \mathcal{F}_{t} \right)\\
&= \lim_{dt \to 0} \frac{1}{dt} \E \left( \left( I_{n} \left( t + dt \right) - I_{n} \left( t \right) \right)^{2} \middle| \mathcal{F}_{t} \right)\\
&= \frac{1}{n^{2}}\lim_{dt \to 0} \frac{1}{dt} \E \left( \left( \sum\limits_{i=1}^{n} \left( f\left( x_{i}\left( t + dt \right) \right) - f \left( x_{i}\left( t \right) \right) \right) \right)^{2}  \middle| \mathcal{F}_{t} \right)\\
&= \frac{1}{n^{2}} \sum\limits_{i=1}^{n} w\left( x_{i}\left( t \right) - m_{n} \left( t \right) \right) \E \left( \left( f\left( x_{i}\left( t \right) + Z \right) - f\left( x_{i}\left( t \right) \right) \right)^{2} \right).
\end{align*}
Now if $f \in C_{b}$, $\left| f \right| \leq 1$, then $L M_{n}^{2} \left( t \right) \leq 4a / n$. If $f = Id$, then $L M_{n}^{2} \left( t \right) \leq a E\left( Z^{2} \right) / n$. In both cases this shows, via \eqref{eq:fubini}, that \eqref{eq:E_Mn2_to_0} holds.
\end{proof}

\begin{proof}[Proof of Theorem \ref{thm:weak_conv_of_eq}]
Recall that $w \leq a$ and that $\left| w' \right| \leq a'$.

We first show that for every $t\geq 0$ and every $f \in H$,
\begin{equation}\label{eq:conv1}
\left\langle f, \mu_{n}\left( t \right) \right\rangle \Rightarrow_{n} \left\langle f, \mu\left( t \right) \right\rangle
\end{equation}
in $\mathbb{R}$. First of all, we know that the projection $\pi_{t}\ :\ D\left( [0,\infty), \mathcal{P}_{1}\left( \mathbb{R} \right) \right) \to \mathcal{P}_{1}\left( \mathbb{R} \right)$, given by $\pi_{t} \left( \nu \left( \cdot \right) \right) = \nu\left( t \right)$, is continuous at $\nu \left( \cdot \right)$ if and only if $\nu \left( \cdot \right)$ is continuous at $t$. Since $\mu\left( \cdot \right) \in C\left( [0,\infty), \mathcal{P}_{1}\left( \mathbb{R} \right) \right)$ a.s. (by Corollary \ref{cor:tightness_P1(R)}), the continuous mapping theorem and the fact that $\mu_{n} \left( \cdot \right) \Rightarrow_{n} \mu\left( \cdot \right)$ in $D\left( [0,\infty), \mathcal{P}_{1}\left( \mathbb{R} \right) \right)$ together imply that for every $t \geq 0$,
\begin{equation*}
\mu_{n} \left( t \right) \Rightarrow_{n} \mu \left( t \right)
\end{equation*}
in $\mathcal{P}_{1}\left( \mathbb{R} \right)$. Now for any $f \in H$, the function $h_{f}\ :\ \mathcal{P}_{1}\left( \mathbb{R} \right) \to \mathbb{R}$, given by $h_{f} \left( \nu \right) = \left\langle f, \nu \right\rangle$, is continuous. (This follows easily from the equivalence of (1) and (2) in Lemma \ref{lem:was_metric}.) Using the continuous mapping theorem, this shows \eqref{eq:conv1}. 

We introduce $g_{f} \left( x \right) := \left( \E \left( f \left( x + Z \right) \right) - f\left( x \right) \right)$ in order to abbreviate notation. What remains to be shown is that for every $t\geq 0$ and every $f \in H$,
\begin{equation}\label{eq:conv2}
\int\limits_{0}^{t} \left\langle g_{f}\left( x \right) w \left( x - m_{n} \left( s \right) \right), \mu_{n} \left( s \right) \right\rangle ds \Rightarrow_{n} \int\limits_{0}^{t} \left\langle g_{f}\left( x \right) w \left( x - m \left( s \right) \right), \mu \left( s \right) \right\rangle ds
\end{equation}
in $\mathbb{R}$. First, let us show that for any $f \in H$, the function $h_{f}\ :\ \mathcal{P}_{1}\left( \mathbb{R} \right) \to \mathbb{R}$, given by\linebreak[4] $h_{f} \left( \nu \right) = \left\langle g_{f}\left( x \right) w \left( x - M \right), \nu \right\rangle$, is continuous, where $M = \int x \nu \left( dx \right)$. So suppose $\lim_{n \to \infty} d_{1} \left( \nu_{n}, \nu \right) = 0$, and denote by $M_{n}$ the mean of $\nu_{n}$, i.e.\ $M_{n} = \int x \nu_{n} \left( dx \right)$. We need to show that in this case
\begin{equation*}
\lim\limits_{n \to \infty} \left| \left\langle g_{f}\left( x \right) w \left( x - M_{n} \right), \nu_{n} \right\rangle - \left\langle g_{f}\left( x \right) w \left( x - M \right), \nu \right\rangle \right| = 0.
\end{equation*}
We use the triangle inequality to arrive at
\[
\begin{aligned}
\left| \left\langle g_{f}\left( x \right) w \left( x - M_{n} \right), \nu_{n} \right\rangle - \left\langle g_{f}\left( x \right) w \left( x - M \right), \nu \right\rangle \right| &\leq \left| \left\langle g_{f}\left( x \right) \left[ w \left( x - M_{n} \right) - w \left( x - M \right) \right], \nu_{n} \right\rangle \right|\\
&\quad+ \left| \left\langle g_{f}\left( x \right) w \left( x - M \right), \nu_{n} \right\rangle - \left\langle g_{f}\left( x \right) w \left( x - M \right), \nu \right\rangle \right|.
\end{aligned}
\]
The first term on the right hand side can be bounded above by $2a' \left| M_{n} - M \right|$. This term goes to zero as $n \to \infty$, due to $\lim_{n \to \infty} d_{1} \left( \nu_{n}, \nu \right) = 0$ and the equivalence of (1) and (2) in Lemma \ref{lem:was_metric}. For every $f \in H$, $g_{f}\left( x \right) w \left( x - M \right)$ is a bounded and continuous function, and consequently the second term goes to zero as $n \to \infty$ as well, due to $\lim_{n \to \infty} d_{1} \left( \nu_{n}, \nu \right) = 0$ and the equivalence of (1) and (2) in Lemma \ref{lem:was_metric}. Thus $h_{f}$ is continuous for every $f \in H$.

Consequently the continuous mapping theorem and Lemma \ref{lem:EKex1} together imply that
\begin{equation*}
\left\langle g_{f}\left( x \right) w \left( x - m_{n} \left( \cdot \right) \right), \mu_{n} \left( \cdot \right) \right\rangle \Rightarrow_{n} \left\langle g_{f}\left( x \right) w \left( x - m \left( \cdot \right) \right), \mu \left( \cdot \right) \right\rangle
\end{equation*}
in $D\left( [0,\infty), \mathbb{R} \right)$. Lemma \ref{lem:EKex2} then implies that
\begin{equation*}
\int\limits_{0}^{\cdot} \left\langle g_{f}\left( x \right) w \left( x - m_{n} \left( s \right) \right), \mu_{n} \left( s \right) \right\rangle ds \Rightarrow_{n} \int\limits_{0}^{\cdot} \left\langle g_{f}\left( x \right) w \left( x - m \left( s \right) \right), \mu \left( s \right) \right\rangle ds
\end{equation*}
in $D\left( [0,\infty), \mathbb{R} \right)$. The integral function on the right hand side is continuous, which implies that the one-dimensional distributions converge weakly in $\mathbb{R}$, i.e.\ \eqref{eq:conv2} holds.
\end{proof}

\begin{proof}[Proof of Corollary \ref{cor:limit_solves_mfe}]
Fix $t \geq 0$ and $f \in H$. Theorem \ref{thm:error_vanishes} implies that
\begin{equation*}
A_{t,f} \left( \mu_{n} \left( \cdot \right) \right) \Rightarrow_{n} 0
\end{equation*}
in $\mathbb{R}$. On the other hand, Theorem \ref{thm:weak_conv_of_eq} implies that
\begin{equation*}
A_{t,f} \left( \mu_{n} \left( \cdot \right) \right) \Rightarrow_{n} A_{t,f} \left( \mu \left( \cdot \right) \right)
\end{equation*}
in $\mathbb{R}$. However, we know that a sequence cannot converge weakly to two different limits (Billingsley \cite[Section 2]{billingsley1999convergence}), and hence $A_{t,f} \left( \mu \left( \cdot \right) \right) = 0$ almost surely.
\end{proof}

%%%%%%%%%%%%%%%%%%%%%%%%%%%%%%%%%%%%%%%%%%%%%%%%%%%%%%%%%%%%%%%%%%%%%%%%%
%%%%%%%%%%%%%%%%%%%%%%%%%%%%%%%%%%%%%%%%%%%%%%%%%%%%%%%%%%%%%%%%%%%%%%%%%

\subsection{Uniqueness of the solution to the mean field equation}\label{ch:Uniqueness}

In this subsection we prove Theorem \ref{thm:cont_dep_init_cond} and Corollary \ref{cor:uniqueness}.

\begin{proof}[Proof of Theorem \ref{thm:cont_dep_init_cond}]
We have to show that there exists a constant $c$ such that \eqref{eq:cont_dep_init_cond} holds. In order to do this we prove that there exists a constant $c$ such that
\begin{equation}\label{eq:Gronwall_condition}
d_H \left( t \right) \leq d_H\left( 0 \right) + c \int\limits_{0}^{t} d_H \left( s \right) ds,
\end{equation}
from which \eqref{eq:cont_dep_init_cond} follows using Gr\"{o}nwall's lemma.

So let us now prove \eqref{eq:Gronwall_condition}. Denote by $m^{1} \left( t \right)$ and $m^{2} \left( t \right)$ the mean of $\mu^{1}\left( t \right)$ and $\mu^{2}\left( t \right)$ respectively: $m^{i} \left( t \right) = \int x \mu^{i} \left( t , dx \right)$ for $i=1,2$. We know that $\mu^{1} \left( \cdot \right)$ and $\mu^{2} \left( \cdot \right)$ are solutions to the mean field equation, with initial conditions $\mu_{0}^{1}$ and $\mu_{0}^{2}$ respectively, and so
\begin{equation*}\label{eq:mean_field_gen}
A_{t,f} \left( \mu^{i} \left( \cdot \right) \right) = 0
\end{equation*}
holds for all $f \in H$, for $i=1,2$ (recall \eqref{eq:A_def}). Using this, the definition of the metric $d_H$, and simple inequalities, we have
\begin{align*}
d_H \left( t \right) &= \sup\limits_{f \in H} \left| \left\langle f, \mu^{1} \left( t \right) \right\rangle - \left\langle f, \mu^{2} \left( t \right) \right\rangle \right|\\
&\leq \sup\limits_{f \in H} \left| \left\langle f, \mu_{0}^{1} \right\rangle - \left\langle f, \mu_{0}^{2} \right\rangle \right| \\
&\quad+ \int\limits_{0}^{t} \sup\limits_{f \in H} \left|  \left\langle \left( \E \left( f \left( x + Z \right) \right) - f\left( x \right) \right) w\left( x - m^{1} \left( s \right) \right), \mu^{1} \left( s \right) \right\rangle  \right.\\
&\quad- \left.  \left\langle \left( \E \left( f \left( x + Z \right) \right) - f\left( x \right) \right) w\left( x - m^{2} \left( s \right) \right), \mu^{2} \left( s \right) \right\rangle \right| ds.
\end{align*}
The first term on the right hand side is exactly $d_H\left(0\right)$, and so what we have to show is that the supremum inside the integral is at most a constant multiple of $d_H\left( s \right)$. In order to abbreviate notation, we introduce
\[
 g_{f} \left( x \right) := \left( \E \left( f \left( x + Z \right) \right) - f\left( x \right) \right).
\]
We know that $g_{Id} = \E\left( Z \right) = 1$ and that for every $f \in C_{b}$, $\left| f \right| \leq 1$, $g_{f}$ is continuous and $\left|g_{f} \right| \leq 2$. First we use the triangle inequality:
\begin{multline*}
\sup\limits_{f \in H} \left|  \left\langle g_{f} \left( x \right) w\left( x - m^{1} \left( s \right) \right), \mu^{1} \left( s \right) \right\rangle - \left\langle g_{f} \left( x \right) w\left( x - m^{2} \left( s \right) \right), \mu^{2} \left( s \right) \right\rangle \right|\\
\begin{aligned}
&\leq \sup\limits_{f \in H} \left|  \left\langle g_{f} \left( x \right) w\left( x - m^{1} \left( s \right) \right), \mu^{1} \left( s \right) \right\rangle - \left\langle g_{f} \left( x \right) w\left( x - m^{1} \left( s \right) \right), \mu^{2} \left( s \right) \right\rangle \right|\\
&\quad+ \sup\limits_{f \in H} \left| \left\langle g_{f} \left( x \right)  w\left( x - m^{1} \left( s \right) \right) , \mu^{2} \left( s \right) \right\rangle - \left\langle g_{f} \left( x \right) w\left( x - m^{2} \left( s \right) \right) , \mu^{2} \left( s \right) \right\rangle \right|.
\end{aligned}
\end{multline*}
We bound the two terms separately by a constant multiple of $d_H\left( s \right)$. Let us start with the first term. If $f \in C_{b}$, $\left| f \right| \leq 1$, then $g_{f} \left( x \right) w\left( x - m^{1} \left( s \right) \right)$ is continuous in $x$ and its absolute value is at most $2a$, so consequently the expression in the supremum is at most $2a d_H\left( s \right)$. If $f = Id$, then the expression in the supremum is $\left| \left\langle w\left( x - m^{1} \left( s \right) \right), \mu^{2} \left( s \right) \right\rangle - \left\langle w\left( x - m^{1} \left( s \right) \right), \mu^{1} \left( s \right) \right\rangle \right|$ which is at most $a d_H\left( s \right)$ since $w\left( x - m^{1}\left( s \right) \right)$ is continuous in $x$ and $\left| w \right| \leq a$. Thus the first term is at most $2a d_H\left( s \right)$.

Now let us look at the second term. Using the fact that $w'$ is bounded we have
\[
\left| w\left( x - m^{1} \left( s \right) \right) - w\left( x - m^{2} \left( s \right) \right)  \right| \leq a' \left| m^{1} \left( s \right) - m^{2}\left( s \right) \right| \leq a' d_H\left( s \right),
\]
and since $\left| g_{f} \right| \leq 2$ for all $f\in H$, it follows that the second term is at most $2a' d_H\left( s \right)$. Therefore this shows \eqref{eq:Gronwall_condition} with $c:= 2a + 2a'$.
\end{proof}

\begin{proof}[Proof of Corollary \ref{cor:uniqueness}]
If $d_H\left( 0 \right) = 0$ then, due to Theorem \ref{thm:cont_dep_init_cond}, $d_H\left( t \right) = 0$ for all $t\geq 0$, and consequently $\mu^{1}\left( t \right) = \mu^{2} \left( t \right)$ for all $t\geq 0$, so $\mu^{1}\left( \cdot \right) = \mu^{2}\left( \cdot \right)$.
\end{proof}

%%%%%%%%%%%%%%%%%%%%%%%%%%%%%%%%%%%%%%%%%%%%%%%%%%%%%%%%%%%%%%%%%%%%%%%%%
%%%%%%%%%%%%%%%%%%%%%%%%%%%%%%%%%%%%%%%%%%%%%%%%%%%%%%%%%%%%%%%%%%%%%%%%%

\subsection{Proof of Theorem \ref{thm:perkins_mod}}\label{ch:Perkins}

Before proving Theorem \ref{thm:perkins_mod}, let us present a result for general Polish space state spaces that we use during the proof. Suppose $E$ is a Polish space and $d$ is a complete metric on $E$. For $x\left( \cdot \right) \in D\left( [0, \infty), E \right)$, $T > 0$, and $\delta > 0$, let
\begin{equation*}
w\left( x\left( \cdot \right), \delta, T \right) := \sup\limits_{0 \leq s,t \leq T, \left|s-t\right| \leq \delta} d \left( x\left( t \right), x\left( s \right) \right).
\end{equation*}
(We note that this $w$ should not be confused with our jump rate function $w$---the distinction will be obvious from the context.) Then:
\begin{theorem}\label{thm:general_C-rel_comp}
Suppose $\left\{ X_{n} \left( \cdot \right) \right\}_{n \geq 1}$ is a sequence of cadlag $E$-valued processes, where $E$ is Polish. $\left\{ X_{n} \left( \cdot \right) \right\}_{n \geq 1}$ is $C$-relatively compact in $D\left( [0,\infty), E \right)$ if and only if the following two conditions hold:
\begin{enumerate}[(a)]
\item For every $T > 0$ and every $\eps > 0$ there exists a compact set $K^{0} = K_{T,\eps}^{0} \subset E$ such that
\begin{equation*}
\sup\limits_{n} \p \left( X_{n} \left( t \right) \notin K^{0} \text{ for some } t \leq T \right) \leq \eps.
\end{equation*}
\item For every  $T > 0$ and every $\eps > 0$ there exists a $\delta > 0$ such that
\begin{equation*}
\limsup\limits_{n \to \infty} \p \left( w\left( X_{n} \left( \cdot \right), \delta, T \right) \geq \eps \right) \leq \eps.
\end{equation*}
\end{enumerate}
\end{theorem}
This theorem follows from Corollary 3.7.4, Remark 3.7.3, and Theorem 3.10.2 of Ethier and Kurtz \cite{ethier1986markov}. Note: (a) is called the \textit{compact containment condition}.

\begin{proof}[Proof of Theorem \ref{thm:perkins_mod}]
To prove that $\left\{P_{n} \left( \cdot \right) \right\}_{n \geq 1}$ is C-relatively compact in $D\left( [0, \infty), \mathcal{P}_1 \left( \mathbb{R} \right) \right)$ we need to show that conditions (a) and (b) of Theorem \ref{thm:general_C-rel_comp} follow from conditions (i) and (ii) of Theorem \ref{thm:perkins_mod}. Notice that it is enough to demonstrate this for small $\eps$'s. First, let us start with the compact containment condition.
Let $T > 0$ and $\eps \in \left( 0,\frac13 \right)$ be fixed. To this $T > 0$, there exists a $c = c \left( T \right)$ and a $\tau = \tau\left( T \right) \in \left(0, 1\right)$ given by condition (i). For every $m \geq 1$, let $\eps_m := \tilde{c} \eps / m^{2/\tau}$, where $\tilde{c}^{-1} = \sum_{m \geq 1} m^{-2/\tau}$ so that $\sum_{m \geq 1} \eps_m = \eps$. Now by condition (i), to the fixed $T$ and $\eps_{m}$ there exists a $K_{m} := K_{T, \eps_m} \in \mathbb{R}$ such that 
\begin{equation*}
K_m \leq \frac{c}{\eps_{m}^{1-\tau}} = \frac{c}{\left( \tilde{c} \eps \right)^{1-\tau}} m^{2/\tau - 2}
\end{equation*}
and
\begin{equation*}
\sup\limits_{n} \p \left( \sup\limits_{0 \leq t \leq T} P_{n} \left( t, (-\infty, -K_{m}) \cup (K_{m}, \infty) \right) > \eps_m \right) < \eps_m.
\end{equation*}
This implies that 
\begin{equation*}
\sup\limits_{n} \p \left( \sup\limits_{0 \leq t \leq T} P_{n} \left( t, (-\infty, -K_{m}) \cup (K_{m}, \infty) \right) > 1 / m^{2/\tau} \right) < \eps_m,
\end{equation*}
since $\eps_m < 1 / m^{2/\tau}$. Now given the $K_m$'s, we define the following set of probability measures:
\begin{equation*}
C^0 := \left\{ \mu \in \mathcal{P}_{1} \left( \mathbb{R} \right) : \mu \left( (-\infty, -K_{m}) \cup (K_{m}, \infty) \right) \leq 1 / m^{2/\tau} \text{ for every } m \geq 1 \right\}.
\end{equation*}
(Of course $C^0 = C_{T, \eps}^0$.) The choice of the $K_m$'s implies that
\begin{equation*}
\sup\limits_{n} \p \left( P_{n} \left( t \right) \notin C^{0} \text{ for some } t \leq T \right) \leq \eps,
\end{equation*}
since
\[
\begin{aligned}
\p \left( P_{n} \left( t \right) \notin C^{0} \text{ for some } t \leq T \right) &= \p \left( \exists m \geq 1, \exists t \geq T : P_{n} \left( t, (-\infty, -K_{m}) \cup (K_{m}, \infty) \right) > 1 / m^{2/\tau} \right)\\
&= \p \left( \exists m \geq 1 : \sup\limits_{0 \leq t \leq T} P_{n} \left( t, (-\infty, -K_{m}) \cup (K_{m}, \infty) \right) > 1 / m^{2/\tau} \right)\\
&\leq \sum\limits_{m=1}^{\infty} \p \left( \sup\limits_{0 \leq t \leq T} P_{n} \left( t, (-\infty, -K_{m}) \cup (K_{m}, \infty) \right) > 1 / m^{2/\tau} \right)\\
&\leq \sum\limits_{m=1}^{\infty} \eps_m = \eps.
\end{aligned}
\]
Consequently,
\begin{equation*}
\sup\limits_{n} \p \left( P_{n} \left( t \right) \notin K^{0} \text{ for some } t \leq T \right) \leq \eps,
\end{equation*}
for $K^0 := \overline{C^0}$. This $K^0$ will be good as the compact set needed in condition (a)---what is left to show is that $K^0$ is compact in $\mathcal{P}_{1} \left( \mathbb{R} \right)$. For this, it is enough to show that $C^0$ is relatively compact in $\mathcal{P}_{1} \left( \mathbb{R} \right)$. We know (see, for instance, Feng and Kurtz \cite[Appendix D]{feng2006large}) that $C^0 \subset \mathcal{P}_{1} \left( \mathbb{R} \right)$ is relatively compact if and only if
\begin{equation}\label{eq:rel_comp_feltetel}
\lim\limits_{N \to \infty} \sup\limits_{\mu \in C^0} \int\limits_{\left|x\right| > N} \left|x\right| \mu \left( dx \right) = 0.
\end{equation}
So let us now check \eqref{eq:rel_comp_feltetel}. Note that this is where the original form of Perkins' theorem \cite[Theorem II.4.1]{perkins1999dawson} is not enough and we have to use our modified version, Theorem \ref{thm:perkins_mod}. If $N > K_m$, then
\begin{align*}
\sup\limits_{\mu \in C^0} \int\limits_{\left|x\right| > N} \left|x\right| \mu \left( dx \right) &\leq \sup\limits_{\mu \in C^0} \sum\limits_{l = m}^{\infty} \int\limits_{[-K_{l+1}, -K_{l}) \cup ( K_{l}, K_{l+1} ] } \left|x\right| \mu\left(dx\right)\\
&\leq \sup\limits_{\mu \in C^0} \sum\limits_{l = m}^{\infty} K_{l+1} \mu\left( \left( - \infty, -K_{l}\right) \cup \left( K_{l}, \infty \right) \right)\\
&\leq \sum\limits_{l = m}^{\infty} \frac{c}{\left( \tilde{c} \eps \right)^{1-\tau}} \left( l + 1 \right)^{2/\tau - 2} \frac{1}{l^{2/\tau}}.
\end{align*}
This last expression goes to 0 as $m \to \infty$, which shows \eqref{eq:rel_comp_feltetel}, and consequently we have shown that condition (i) implies condition (a).

Now let us show that conditions (i) and (ii) imply condition (b). First we show that for every $f \in C_{b} \left( \mathcal{P}_{1} \left( \mathbb{R} \right) \right)$, $\left\{ f \circ P_{n} \left( \cdot \right) \right\}_{n \geq 1}$ is $C$-relatively compact in $D\left( [0,\infty), \mathbb{R} \right)$. By the characterization of $C$-relatively compactness we need to check conditions (a) and (b) of Theorem \ref{thm:general_C-rel_comp}. Due to the boundedness of $f$, (a) is immediate. Now let us show (b). Fix $f \in C_{b} \left( \mathcal{P}_{1} \left( \mathbb{R} \right) \right)$, $T > 0$, and $\eps > 0$. Choose $K^0 = K_{T, \eps}^{0}$ as in (a). Define
\[
A = \Bigl\{ h : \mathcal{P}_1 \left( \mathbb{R} \right) \to \mathbb{R}\,\Big| h \left( \mu \right) = \sum\limits_{i = 1}^{k} a_i \prod\limits_{j=1}^{m_i} \left\langle f_{i,j}, \mu \right\rangle, a_i \in \mathbb{R}, f_{i,j} \in C_{b}, k, m_i \in \mathbb{Z}_{+} \Bigr\} \subset C_{b} \left( \mathcal{P}_1 \left( \mathbb{R} \right) \right).
\]
Then $A$ is an algebra containing the constant functions and separating points in $\mathcal{P}_1 \left( \mathbb{R} \right)$. So by the Stone-Weierstrass theorem (see for instance Willard \cite{willard2004general}) there exists an $h \in A$ such that
\begin{equation*}
\sup\limits_{\mu \in K^0} \left| h\left( \mu \right) - f \left( \mu \right) \right| < \eps.
\end{equation*}
If $\left\{ Y_{n} \left( \cdot \right) \right\}_{n \geq 1}$ and $\left\{ Z_{n} \left( \cdot \right) \right\}_{n \geq 1}$ are $C$-relatively compact in $D\left( [0,\infty), \mathbb{R} \right)$ then so are $\left\{ a Y_{n} \left( \cdot \right) + b Z_{n} \left( \cdot \right) \right\}_{n \geq 1}$ and $\left\{ Y_{n} \left( \cdot \right) Z_{n} \left( \cdot \right) \right\}_{n \geq 1}$ for any $a,b \in \mathbb{R}$ (see the characterization of Theorem \ref{thm:general_C-rel_comp}). Therefore condition (ii) of the theorem implies that $\left\{ h \circ P_{n} \left( \cdot \right) \right\}_{n \geq 1}$ is $C$-relatively compact in $D\left( [0, \infty), \mathbb{R} \right)$. Now using the characterization of Theorem \ref{thm:general_C-rel_comp} again, there exists a $\delta > 0$ such that
\begin{equation*}
\limsup\limits_{n \to \infty} \p \left( w \left( h \circ P_{n} \left( \cdot \right), \delta, T \right) \geq \eps \right) \leq \eps.
\end{equation*}
To abbreviate notation, we introduce $P_n \left( \left[0,T\right] \right) := \left\{ P_n \left( t \right) : t \in \left[0,T\right) \right\}$. If $s,t \leq T$, then
\begin{align*}
\left| f\left( P_{n} \left( t \right) \right)  - f\left( P_{n} \left( s \right) \right) \right| &= \left| f\left( P_{n} \left( t \right) \right)  - f\left( P_{n} \left( s \right) \right) \right| \mathbf{1} \left[ P_{n} \left( \left[0,T\right] \right) \subseteq K^0 \right]\\
&+ \left| f\left( P_{n} \left( t \right) \right)  - f\left( P_{n} \left( s \right) \right) \right| \mathbf{1} \left[ P_{n} \left( \left[0,T\right] \right) \nsubseteq K^0 \right]\\
&\leq 2 \sup\limits_{\mu \in K^0} \left| h\left( \mu \right) - f\left( \mu \right) \right| + \left| h \left( P_{n} \left( t \right) \right) - h \left( P_{n} \left( s \right) \right) \right|\\
&+ 2 \left\| f \right\|_{\infty} \mathbf{1} \left[ P_{n} \left( \left[0,T\right] \right) \nsubseteq K^0 \right]
\end{align*}
and so by the definition of $w \left( f \circ P_{n} \left( \cdot \right), \delta, T \right)$ we have
\begin{equation*}
w \left( f \circ P_{n} \left( \cdot \right), \delta, T \right) \leq 2 \eps + w \left( h \circ P_{n} \left( \cdot \right), \delta, T \right) + 2 \left\| f \right\|_{\infty} \mathbf{1} \left[ P_{n} \left( \left[0,T\right] \right) \nsubseteq K^0 \right].
\end{equation*}
Therefore
\[
\begin{aligned}
\limsup\limits_{n \to \infty} \p \left( w \left( f \circ P_{n} \left( \cdot \right), \delta, T \right) \geq 3 \eps \right) &\leq \limsup\limits_{n \to \infty} \p \left( P_{n} \left( \left[0,T\right] \right) \nsubseteq K^0 \right) + \limsup\limits_{n \to \infty} \p \left( w \left( h \circ P_{n} \left( \cdot \right), \delta, T \right) \geq \eps \right)\\
&\leq \eps + \eps = 2 \eps < 3 \eps,
\end{aligned}
\]
by the choice of $K^0$ and the fact that $\left\{ h \circ P_{n} \left( \cdot \right) \right\}_{n \geq 1}$ is $C$-relatively compact in $D\left( [0, \infty), \mathbb{R} \right)$.

So let us recap what we have done up until now: we have shown that (i) implies the compact containment condition and that (i) and (ii) imply that for every $f \in C_{b} \left( \mathcal{P}_{1} \left( \mathbb{R} \right) \right)$, $\left\{ f \circ P_{n} \left( \cdot \right) \right\}_{n \geq 1}$ is $C$-relatively compact in $D\left( [0, \infty), \mathbb{R} \right)$. These imply that $\left\{ P_{n} \left( \cdot \right) \right\}_{n \geq 1}$ is relatively compact in $D\left( [0, \infty), \mathbb{R} \right)$ (see Theorem 3.9.1 of Ethier and Kurtz \cite{ethier1986markov}). What remains is to show that it is not only relatively compact, but also $C$-relatively compact.

We now show that (i) and (ii) imply (b). Let us define the metric $d := \min \left\{ d_{1}, 1 \right\}$ and let 
\begin{equation*}
w_{d} \left( P_{n} \left( \cdot \right), \delta, T \right) := \sup\limits_{0 \leq s,t \leq T, \left|s-t\right| \leq \delta} d \left( P_{n}\left( t \right), P_{n}\left( s \right) \right).
\end{equation*}
Let $T > 0$ and $\eps > 0$ be fixed, and choose $K^0$ as in (a) again. Since $K^{0}$ is compact, we can choose $\mu_i \in K^0$, $i = 1, \dots, M$, so that $K^0 \subseteq \cup_{i=1}^{M} B\left( \mu_{i}, \eps \right)$, where
\[
 B\left( \mu, \eps \right) = \left\{ \nu \in \mathcal{P}_{1} \left( \mathbb{R} \right) : d_{1} \left( \mu, \nu \right) < \eps \right\}.
\]
Since $\eps < 1$,
\[
 B\left( \mu, \eps \right) = \left\{ \nu \in \mathcal{P}_{1} \left( \mathbb{R} \right) : d \left( \mu, \nu \right) < \eps \right\}.
\]
Now let $f_{i} \left( \mu \right) := d\left( \mu_{i}, \mu \right)$. Clearly $f_{i} \in C_{b} \left( \mathcal{P}_{1} \left( \mathbb{R} \right) \right)$.  We have proved that for every $f \in C_{b} \left( \mathcal{P}_{1} \left( \mathbb{R} \right) \right)$, $\left\{ f \circ P_{n} \left( \cdot \right) \right\}_{n \geq 1}$ is $C$-relatively compact in $D\left( [0, \infty), \mathbb{R} \right)$, which implies that there exists a $\delta > 0$ such that
\begin{equation}\label{eq:utolso}
\sum\limits_{i=1}^{M} \limsup\limits_{n \to \infty} \p \left( w \left( f_{i} \circ P_{n} \left( \cdot \right), \delta, T \right) \geq \eps \right) \leq \eps.
\end{equation}
If $\mu, \nu \in K^0$, choose $\mu_j$ so that $d\left( \nu, \mu_j \right) \leq \eps$. Then
\[
d\left( \mu, \nu \right) \leq d\left( \mu, \mu_j \right) + d\left( \mu_j, \nu \right) \leq \left| d\left( \mu, \mu_j \right) - d\left( \mu_j, \nu \right) \right| + 2 d\left( \mu_j, \nu \right) \leq \max\limits_{i} \left| f_{i} \left( \mu \right) - f_{i} \left( \nu \right) \right| + 2 \eps.
\]
Now let $s,t \leq T$; the above inequality implies
\begin{align*}
d\left( P_{n} \left( t \right), P_{n} \left( s \right) \right) &= d\left( P_{n} \left( t \right), P_{n} \left( s \right) \right) \mathbf{1} \left[ P_{n} \left( \left[0,T\right] \right) \subseteq K^0 \right] + d\left( P_{n} \left( t \right), P_{n} \left( s \right) \right) \mathbf{1} \left[ P_{n} \left( \left[0,T\right] \right) \nsubseteq K^0 \right]\\
&\leq \max\limits_{i} \left| f_{i} \circ P_{n} \left( t \right) - f_{i} \circ P_{n} \left( s \right) \right| + 2 \eps + \mathbf{1} \left[ P_{n} \left( \left[0,T\right] \right) \nsubseteq K^0 \right],
\end{align*}
where we used that $d \leq 1$. Consequently
\begin{equation*}
w_{d} \left( P_{n} \left( \cdot \right), \delta, T \right) \leq \max\limits_{i} w \left( f_{i} \circ P_{n} \left( \cdot \right), \delta, T \right) + 2 \eps + \mathbf{1} \left[ P_{n} \left( \left[0,T\right] \right) \nsubseteq K^0 \right].
\end{equation*}
Now note that $\left\{ w \left( P_{n} \left( \cdot \right), \delta, T \right) \geq 3 \eps \right\}$ and $\left\{ w_d \left( P_{n} \left( \cdot \right), \delta, T \right) \geq 3 \eps \right\}$ are exactly the same events, and consequently
\begin{multline*}
\limsup\limits_{n \to \infty} \p \left( w \left( P_{n} \left( \cdot \right), \delta, T \right) \geq 3 \eps \right) = \limsup\limits_{n \to \infty} \p \left( w_d \left( P_{n} \left( \cdot \right), \delta, T \right) \geq 3 \eps \right)\\
\begin{aligned}
&\leq \limsup\limits_{n \to \infty} \p \left( \max\limits_{i} w \left( f_{i} \circ P_{n} \left( \cdot \right), \delta, T \right) \geq \eps \right) + \limsup\limits_{n \to \infty} \p \left( P_{n} \left( \left[0,T\right] \right) \nsubseteq K^0 \right)\\
&\leq \limsup\limits_{n \to \infty} \sum\limits_{i=1}^{M} \p \left(  w \left( f_{i} \circ P_{n} \left( \cdot \right), \delta, T \right) \geq \eps \right) + \limsup\limits_{n \to \infty} \p \left( P_{n} \left( \left[0,T\right] \right) \nsubseteq K^0 \right)\\
&\leq \eps + \eps = 2 \eps < 3 \eps,
\end{aligned}
\end{multline*}
using \eqref{eq:utolso} and the choice of $K^0$. This shows condition (b), and therefore completes the proof.
\end{proof}

\section{Simulation Results}\label{sec:sim}

As mentioned before, we believe that the empirical measure of the particle system for a large number of particles is well approximated by the mean field equation for a wide class of jump rate functions $w$. In this section we support this conjecture via simulation results in cases when Theorem \ref{thm:main} does not apply due to the assumptions made on the jump rate function $w$. We test the simulated particle distribution around the mean position against the mean field stationary formulas calculated in Examples \ref{ex:exp} and \ref{ex:step}. We see that in both cases the particle system moves as a traveling wave with speed very close to the one predicted by the mean field equation, and the empirical density of the particle system fluctuates around the stationary density \eqref{eq:stationary_dist}.

Suppose the stationary density derived from the mean field approximation is $\rho$. We fix the number of particles ($n = 10^{4}$ in the examples) and run the simulation for the finite number of particles for time $t \in [0,T]$. At $t = 0$ we suppose that every particle is at $0$. For every time $t$ there exists an empirical density histogram $\hat{\rho}_{n}\left(\cdot,t\right)$ of the particles around their empirical mean position, and from this we derive the time average for the empirical density histograms
\begin{equation*}
\tilde{\rho}_{n,T}\left(x\right) = \frac{1}{T} \int\limits_{0}^{T} \hat{\rho}_{n}\left(x,t\right) dt.
\end{equation*}
Since after a certain relaxation time $t_{0}$ the distribution of the particles around their mean position is very close to the stationary distribution, if we choose $T$ big enough compared to $t_{0}$ then we do not need to know the exact value of $t_{0}$, and $\tilde{\rho}_{n}$ will be very close to the stationary distribution of the $n$ particles around their mean position. In both examples considered, we show plots of the time average $\tilde{\rho}_{n,T}$, and we also give a snapshot of the empirical density histogram $\hat{\rho}_{n}\left(\cdot,t\right)$ at a given time $t$ to give the reader an idea of how this fluctuates around the stationary density.

One question remains: how do we create the empirical density histogram? We choose an interval $I = \left[a_{0},a_{1}\right]$ so that $\int_{I} \rho\left(x\right)dx \approx 1$ and divide $I$ into $N = \frac{a_{1} - a_{0}}{h_{n}}$ intervals of length $h_{n}$. In our cases the tails of the distribution are always at least exponential, so if $I$ is chosen big enough then we do not have to deal with what happens outside $I$. From this the empirical density histogram of the particles around their mean position at time $t$ is
\[
\hat{\rho}_{n}\left(x,t\right) = \frac{1}{n h_{n}} \left|\{i : x_{i}\left(t\right) - m_{n}\left(t\right) \in \Delta_{j}, i = 1, \dots, n\}\right|,\qquad \text{if } x\in \Delta_{j}, j = 1, \dots, n,
\]
where $\Delta_{j} = [a_{0} + \left(j-1\right)h_{n}, a_{0} + j h_{n})$. We know that if $x_{1}\left(t\right),\dots,x_{n}\left(t\right)$ are independent identically distributed random variables, then if $\lim_{n\to\infty} h_{n} = 0$ and  $\lim_{n\to\infty} n h_{n} = \infty$ then the empirical density histogram converges at every continuity point of the density as $n\to\infty$. However, in our case the $x_{i}\left(t\right)$ random variables are not independent. Nonetheless, we choose $N = 2 \sqrt{n}$ and therefore $h_{n} = \frac{a_{1} - a_{0}}{2 \sqrt{n}}$ so that the criteria above is satisfied.

\subsection{Jump rate function is exponential}

Let $w\left(x\right) = e^{-\beta x}$ as in Example \ref{ex:exp}, in this case the stationary mean field density is given by \eqref{eq:rhobeta1.1}. Figure \ref{fig:exp_time_avg} shows a plot of $\rho_{\beta}$ in blue and the time average of the empirical density histograms $\tilde{\rho}_{n,T,\beta}$ in red with the following parameters: $\beta = 1, n = 10^{4}, a_{0} = -10, a_{1} = 10, T = 1000$. We conclude that the two fit extremely well and note that simulations show the same for different values of the parameter $\beta$. Figure \ref{fig:exp_pillanat} shows the empirical density histogram $\hat{\rho}_{n,\beta}\left(\cdot,t\right)$ at time $t = 500$ with the other parameters the same as in Figure \ref{fig:exp_time_avg}. Finally, Figure \ref{fig:speed_exp} demonstrates that the particle system moves at a constant velocity $c$ predicted by the mean field model.
\begin{figure}[ht]
\centering
	\includegraphics[width=0.8\textwidth]{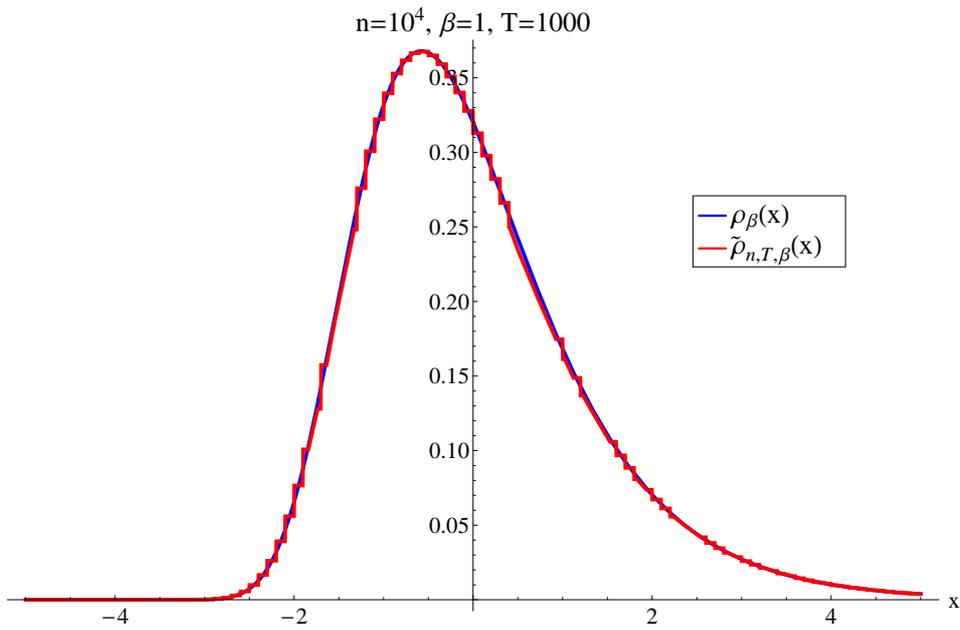}
	\caption{Time average of empirical density histograms (red curve) in the case when $w$ is exponential. Mean field stationary distribution is shown in blue.}\label{fig:exp_time_avg}
\end{figure}
\begin{figure}[ht]
\centering
	\includegraphics[width=0.8\textwidth]{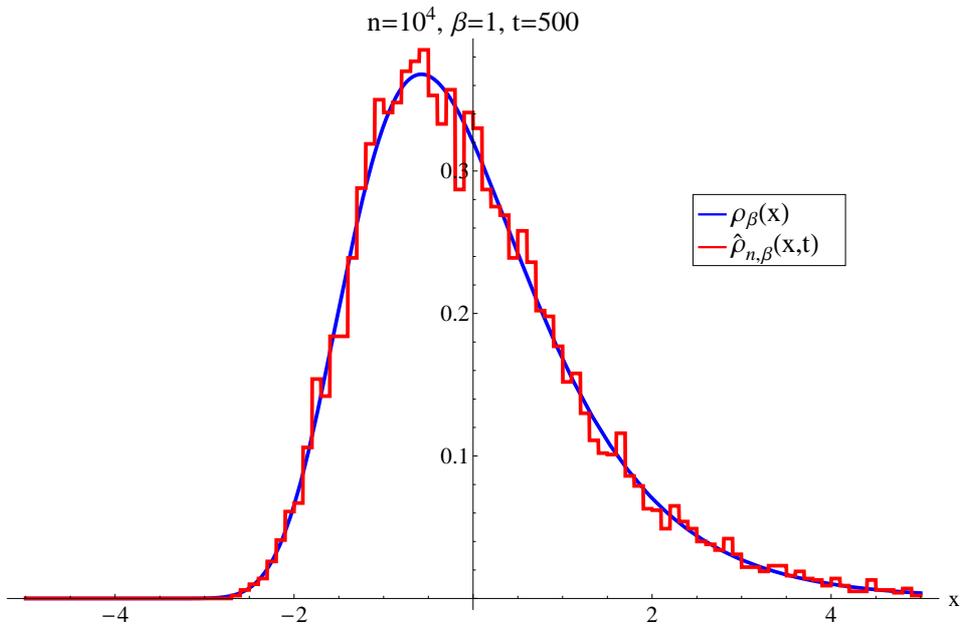}
	\caption{Snapshot of empirical density histogram at time $t = 500$ (red curve) in the case when $w$ is exponential. Mean field stationary distribution is shown in blue.}\label{fig:exp_pillanat}
\end{figure}
\begin{figure}[ht]
\centering
	\includegraphics[width=0.8\textwidth]{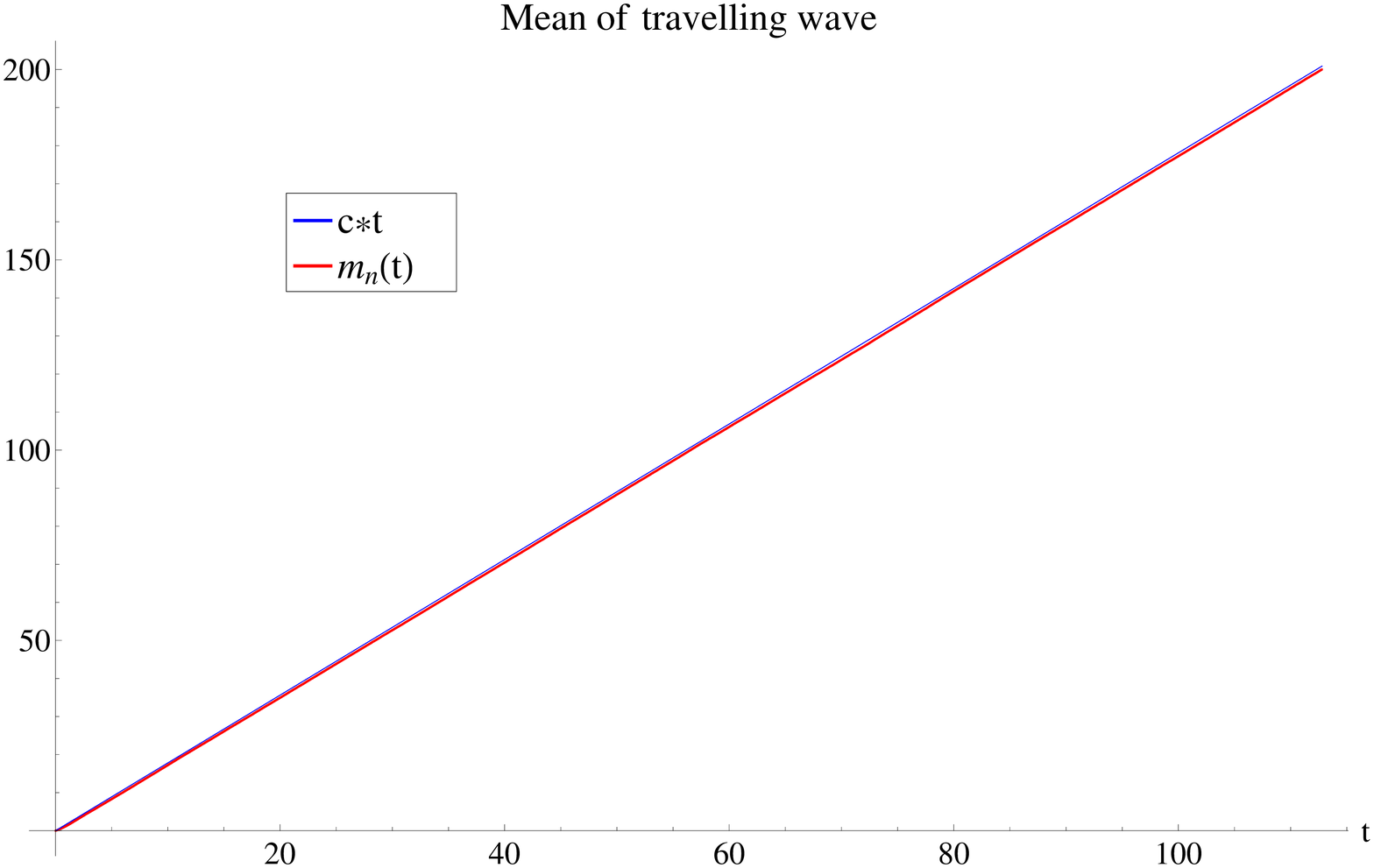}
	\caption{Mean position of the particles in the case when $w$ is exponential. Blue line is from the mean field result, red line is from simulation with $n=10^{4}$ particles.}\label{fig:speed_exp}
\end{figure}
\afterpage{\clearpage}

\subsection{Jump rate function is a step function}

Let us consider the jump rate function analyzed in Example \ref{ex:step} when $w$ is a step function. In this case the stationary mean field density is given by \eqref{eq:stac_step}. Figure \ref{fig:step_time_avg} shows a plot of $\rho_{a,b}$ in blue and the time average of the empirical density histograms $\tilde{\rho}_{n,T,a,b}$ in red with the following parameters: $a = 2, b = 1, n = 10^{4}, a_{0} = -10, a_{1} = 10, T = 1000$. Once again we conclude that the two fit extremely well and note that simulations show the same for different values of the parameters $a$ and $b$. Figure \ref{fig:step_pillanat} shows the empirical density histogram $\hat{\rho}_{n,a,b}\left(\cdot,t\right)$ at time $t = 500$ with the other parameters the same as in Figure \ref{fig:step_time_avg}. Finally, Figure \ref{fig:speed_step} demonstrates that the particle system moves at a constant velocity $c$ predicted by the mean field model.
\begin{figure}[ht]
\centering
	\includegraphics[width=0.8\textwidth]{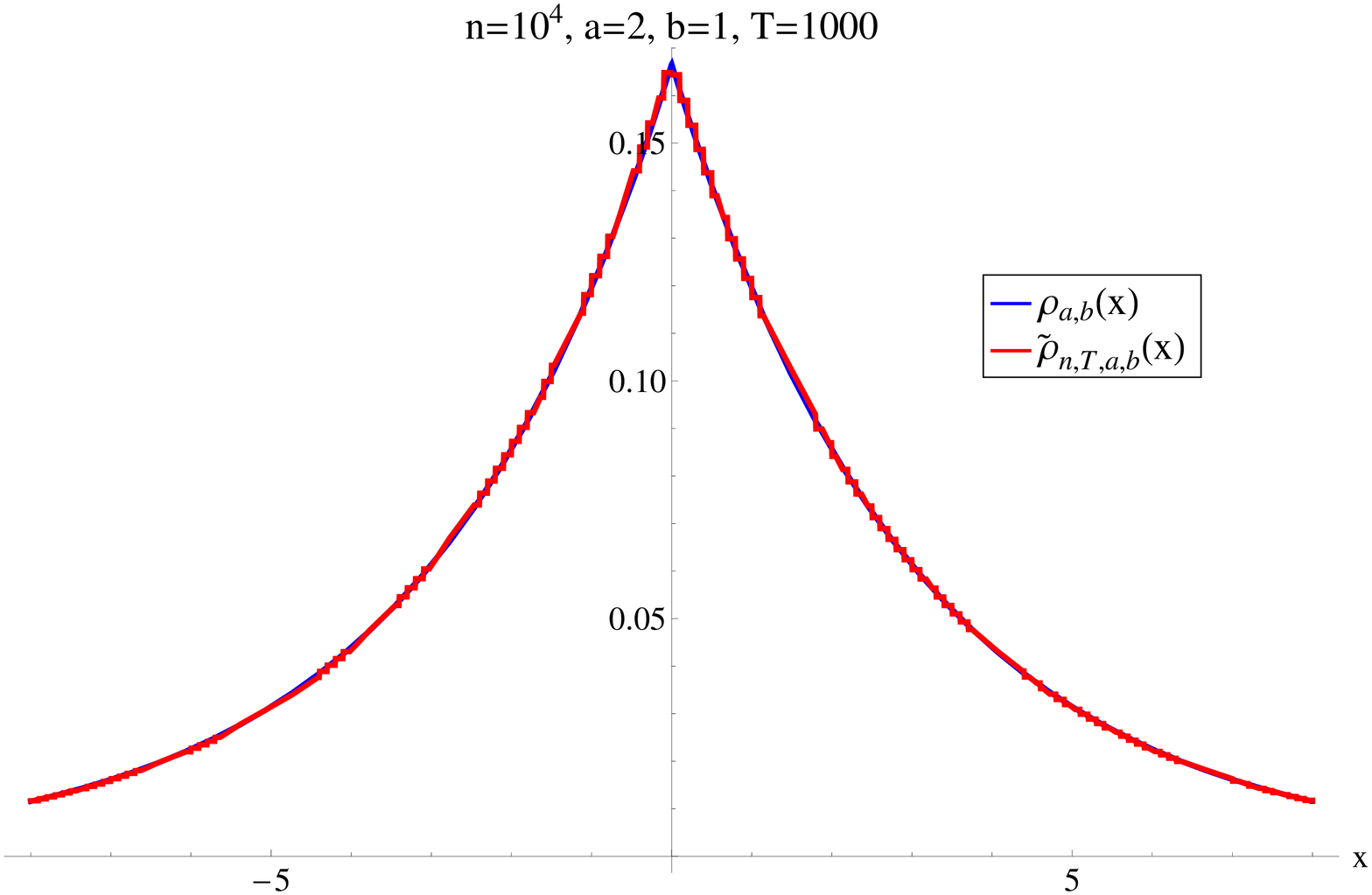}
	\caption{Time average of empirical density histograms (red curve) in the case when $w$ is a step function. Mean field stationary distribution is shown in blue.}\label{fig:step_time_avg}
\end{figure}
\begin{figure}[ht]
\centering
	\includegraphics[width=0.8\textwidth]{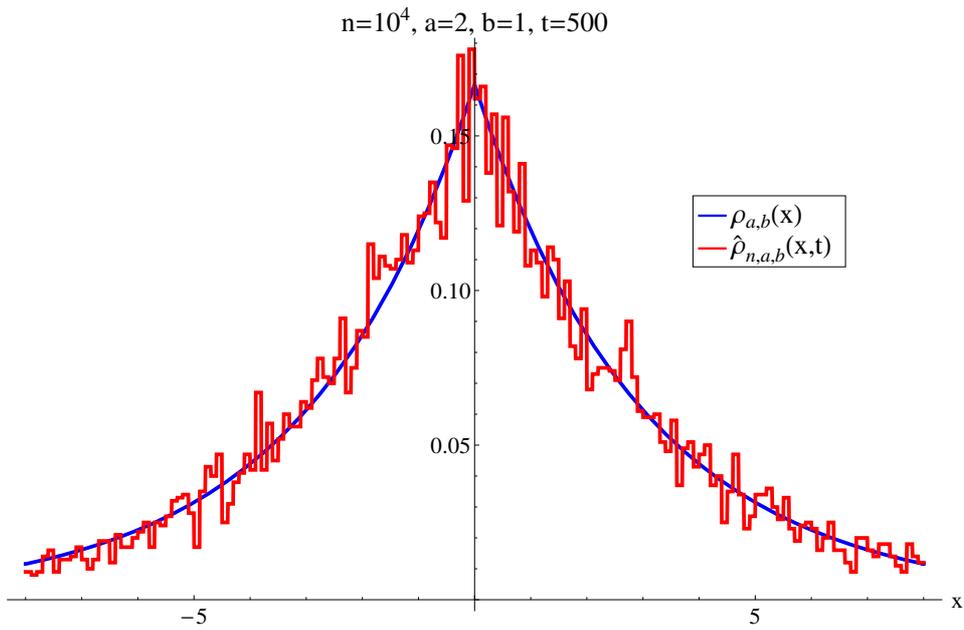}
	\caption{Snapshot of empirical density histogram at time $t = 500$ (red curve) in the case when $w$ is a step function. Mean field stationary distribution is shown in blue.}\label{fig:step_pillanat}
\end{figure}
\begin{figure}[ht]
\centering
	\includegraphics[width=0.8\textwidth]{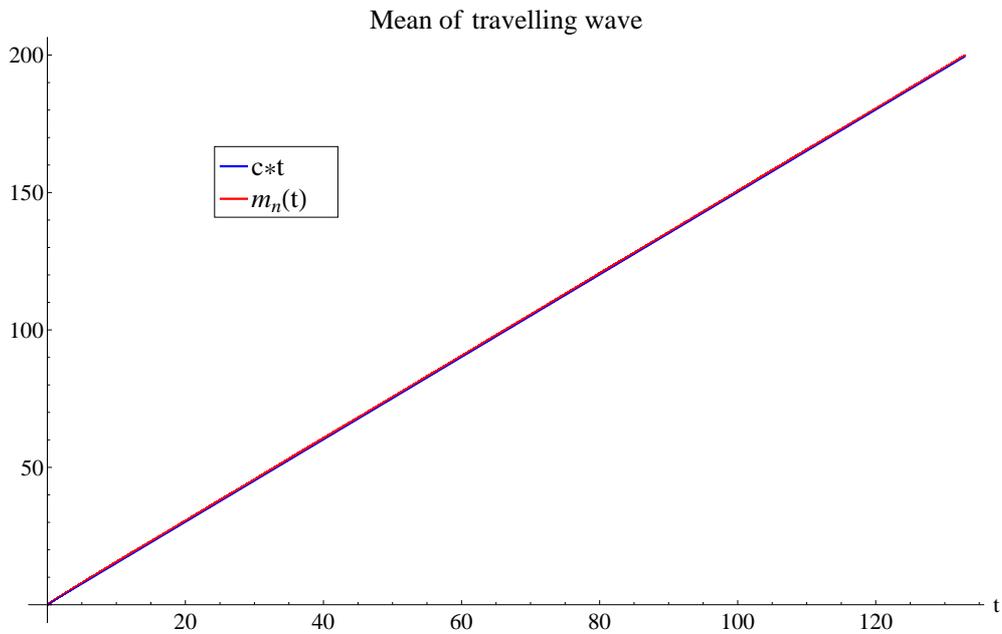}
	\caption{Mean position of the particles in the case when $w$ is a step function. Blue line is from the mean field result, red line is from simulation with $n=10^{4}$ particles.}\label{fig:speed_step}
\end{figure}
\afterpage{\clearpage}

\section*{Acknowledgments}

The authors wish to thank J\'anos Engl\"ander, Attila R\'akos, Bal\'azs R\'ath and Imre P\'eter T\'oth for stimulating discussions and their continuous interest in this project. We also thank Thomas G. Kurtz for bringing to our attention other ways of proving tightness, and Marianna Bolla for helping us with statistics-related problems we encountered during simulations.

% Bibliography
\bibliographystyle{plain}
\bibliography{biblio}

\end{document}